\documentclass[10pt]{article}
\title{Norm estimates}

\usepackage[a4paper,top=3cm,bottom=2cm,left=3cm,right=3cm,marginparwidth=1.75cm]{geometry}

\usepackage{amsmath}
\usepackage{amsthm}
\usepackage{amssymb}
\usepackage{verbatim}
\usepackage{mathtools}
\usepackage{enumitem}
\usepackage[colorlinks]{hyperref}
\usepackage{xcolor}
\usepackage{txfonts}
\usepackage{bbm}
\usepackage{mathabx}

\usepackage{fullpage}
\usepackage{color, tikz}
\usepackage{changepage}
\usepackage{esint}
\usepackage{lscape}
\usepackage{authblk}
\allowdisplaybreaks

\theoremstyle{plain}
\newtheorem{theorem}{Theorem}[section]
\newtheorem{lemma}[theorem]{Lemma}
\newtheorem{proposition}[theorem]{Proposition}
\newtheorem{corollary}{Corollary}

\theoremstyle{definition}
\newtheorem{definition}{Definition}[section]

\theoremstyle{definition}
\newtheorem{remark}{Remark}[section]
\newtheorem{assumption}{Assumption}[section]
\theoremstyle{remark}

\newcommand{\N}{\mathbb{N}}
\newcommand{\Z}{\mathbb{Z}}

\newcommand{\R}{\mathbb{R}}
\newcommand{\C}{\mathbb{C}}
\newcommand{\A}{\mathcal{A}}

\newcommand{\simgrad}{\sym\nabla}

\newcommand{\eps}{{\varepsilon}}

\DeclareMathOperator{\sym}{sym}

\newcommand{\vect}[1]{\boldsymbol #1}

\newcommand{\RRR}{\color{black}}

\newcommand{\CCC}{\color{black}} 
 
\mathtoolsset{showonlyrefs}

\allowdisplaybreaks

\let\oldsqrt\sqrt
\def\sqrt{\mathpalette\DHLhksqrt}
\def\DHLhksqrt#1#2{%
\setbox0=\hbox{$#1\oldsqrt{#2\,}$}\dimen0=\ht0
\advance\dimen0-0.2\ht0
\setbox2=\hbox{\vrule height\ht0 depth -\dimen0}%
{\box0\lower0.4pt\box2}}

\AtBeginDocument{

}

\makeatletter
\DeclareRobustCommand\widecheck[1]{{\mathpalette\@widecheck{#1}}}
\def\@widecheck#1#2{%
    \setbox\z@\hbox{\m@th$#1#2$}%
    \setbox\tw@\hbox{\m@th$#1%
       \widehat{%
          \vrule\@width\z@\@height\ht\z@
          \vrule\@height\z@\@width\wd\z@}$}%
    \dp\tw@-\ht\z@
    \@tempdima\ht\z@ \advance\@tempdima2\ht\tw@ \divide\@tempdima\thr@@
    \setbox\tw@\hbox{%
       \raise\@tempdima\hbox{\scalebox{1}[-1]{\lower\@tempdima\box
\tw@}}}%
    {\ooalign{\box\tw@ \cr \box\z@}}}
\makeatother

\begin{document}


\title{\sc Operator-norm resolvent estimates for thin elastic periodically heterogeneous rods in moderate contrast}

\def\correspondingauthor{\footnote{Corresponding author: k.cherednichenko@bath.ac.uk}}

\author[1]{Kirill Cherednichenko\correspondingauthor{}}
\author[2]{Igor Vel\v{c}i\'{c}\,}
\author[2]{Josip Žubrinić}
\affil[1]{Department of Mathematical Sciences, University of Bath, Claverton Down, Bath,\qquad\qquad  BA2 7AY, United Kingdom}
\affil[2]{Faculty of Electrical Engineering and Computing, University of Zagreb, Unska 3,\qquad\qquad 10000 Zagreb, Croatia}
\maketitle

\begin{abstract}
  We provide resolvent asymptotics as well as various operator-norm estimates  for the system of linear partial differential equations describing thin infinite elastic rods with material coefficients that rapidly oscillate along the rod. The resolvent asymptotics is derived simultaneously with respect to the rod thickness and the period of material oscillations, which are taken to be of the same order. The analysis is carried out separately on two invariant subspaces pertaining to the out-of-line and in-line displacements, under the assumption on material symmetries as well as in the general case when these two types of displacements are intertwined. 
 
 \vskip 0.5cm


\vskip 0.5cm


\end{abstract}

\section{Introduction}

\label{intro_section}

In this paper we analyse the asymptotic behaviour of solutions 
to linear systems of partial differential equations governing the motion of thin infinite  heterogeneous  rods, when the rod thickness and the typical size of heterogeneities are small relative to some fixed length, for example, one representing the spatial scale on which the applied forces vary between their maxima and minima. We assume that the heterogeneity of the rod appears in a periodic manner along the rod and that it is of moderate contrast, namely, that the tensor of material coefficients is uniformly positive definite. Our goal is to construct an operator-norm asymptotic approximation of the solution operators (``norm-resolvent asymptotics") 
with respect to a small parameter that simultaneously plays the roles of the rod thickness and of the period of material oscillations (by setting the macroscopic length-scale to unity.) We first focus on the case when certain material symmetries are assumed, which yields a separation of the full problem into the two simpler problems that are mutually orthogonal in some sense. The setup for which we assume material symmetries serves as a motivation for tackling the general case.
The two mentioned orthogonal problems pertain to describing the in-line and out-of-line displacements, which in the general case intertwine. The norm-resolvent estimates are obtained in various operator norms, for which one can see interesting new non-standard corrector terms appearing in the approximation,  see Remark \ref{remnonstandcorr}.
Recently there is a growing interest in obtaining higher-order approximations for homogenisation problems and in modelling so-called dispersive effects, see e.g. \cite{ABV} and references therein.  Dispersive effects are important if one wants to do quantitative analysis of evolution problems in long time.

The operator-theoretic approach to homogenisation theory was initiated in \cite{BirmanSuslina, BirmanSuslina_corrector, BirmanSuslina_hyperbolic}, where  spectral analysis is used for the derivation of norm-resolvent convergence estimates (see also \cite{Sevostianova, Zhi89} for earlier examples of its use in homogenisation.) The technique, which was initially used in the whole-space setting, was subsequently developed to obtain operator-norm estimates on bounded domains \cite{Suslina_Dirichlet, Suslina_Neumann} and has proved useful for obtaining operator-norm and energy estimates for a number of related contexts:  boundary-value operators \cite{Suslina_Dirichlet, Suslina_Neumann}, parabolic semigroups \cite{Suslina_parabolic, Suslina_parabolic_corrector, Meshkova_Suslina}, hyperbolic groups \cite{BirmanSuslina_hyperbolic, Meshkova_hyperbolic_Math_Notes, Meshkova_hyperbolic_full}, perforated domains \cite{Suslina_perforated}. The key technical milestones for this progress are boundary-layer analysis for bounded domains (as in \cite{Suslina_Dirichlet, Suslina_Neumann}) and two-parametric operator-norm estimates \cite{Suslina_two_parametric}. It seems natural to conjecture that similar developments could be pursued in the context of thin plates and rods, both infinite and bounded, by taking either the spectral approach or the one we use in \cite{cherednichenkovelcic} and here (see, however, Section \ref{sectionobjasnjenje} for a comparison of the two.)  

An overview of the existing approaches to obtaining operator-norm estimates would not be complete without mentioning also the works \cite{Griso_2006, ZhikovPastukhova, Kenig}, whose methods could also be considered in the context of thin structures. However, here we refrain from pursuing the related discussion.

In order to gain error bounds that are uniform with respect to the data
(and hence obtain a sharp quantitative description of the asymptotic behaviour of the spectrum), one has to replace the usual series in macro \CCC and  micro-variable by a family of power-series expansions parametrised by the ``quasimomentum'' $\chi,$ which represents variations over intermediate scales of the length of several periods, see  \cite{ChCoARMA, Quasi_Cooper}. The corresponding asymptotic procedure can be viewed as a combination of the classical perturbation theory with ``matched asymptotics'' on the domain of the quasimomentum. In this approach, the control of the resolvent in the sense of the operator norm is obtained by means of a careful analysis of the remainder estimates for the power series, taking advantage of the related Poincar\'{e}-type inequalities (or Korn-type inequalities for vector problems) that bound the $L^2$-norm of the solution by its energy norm.  Importantly, in order to provide the required uniform estimates, such inequalities must reflect the fact that the lowest eigenvalue of the 
$\chi$-parametrised ``fibre" operator tends to zero as $\vert\chi\vert\to0,$ and hence the $L^2$-norm of the corresponding eigenfunction with unit energy blows up.

In the present work we tackle a problem where uniform estimates of the above kind need to be controlled with respect to an additional length-scale parameter, which represents one of the overall dimensions of the medium, namely the thickness of a thin rod in our case.  In \cite{cherednichenkovelcic} we analysed the behaviour of thin plates, and the present paper develops an extension of the new method  introduced in that paper
 to obtaining norm-resolvent estimates for rods. The relation of the present work and the technique used here to existing approaches is given in Section \ref{sectionobjasnjenje}. We emphasise the fact that our method, which can also be applied to the problems studied in \cite{BirmanSuslina, BirmanSuslina_corrector}, is particularly convenient for application to thin structures in elasticity, where the limit equations are necessarily dispersive, being of fourth order in a part of the limiting deformation field and of second order in the complementary part.    

 Interesting results on the norm-resolvent asymptotics  in high-contrast  homogenisation  are obtained in \cite{ChCoARMA,CherErshKis}. 
  The resemblance of problems with high contrast to those for moderate-contrast thin structures comes from the fact that in the analysis of high-contrast problems one also needs to take into account the fact that eigenvalues having different orders with respect to quasimomentum contribute to the leading-order approximation of the operator, cf. Section \ref{spec_analysis_sec} below. This is also the reason why the approach of \cite{BirmanSuslina,BirmanSuslina_corrector,BirmanSuslina_hyperbolic} is not applicable to these kind of problems.  
 In order to obtain the first-order approximation with respect to $\varepsilon,$ the authors of \cite{ChCoARMA},  which was an inspiration for \cite{cherednichenkovelcic}, used $\varepsilon$-dependent asymptotics for each fixed value of the quasimomentum $\chi.$ In contrast, in \cite{cherednichenkovelcic} and the present paper, a natural $\chi$-scaling of the problem has enabled us to develop an asymptotics with respect to $\chi$ and, as a consequence, go further in the $\varepsilon$-expansion by identifying the correctors that yield higher-order precision.  Furthermore, it enabled us to express the approximation results depending on two additional parameters $\gamma$ and $\delta$ scaling the operator (hence, the spectrum) and the applied loads, respectively. We find that introducing these two parameters is an important part of the result since it is well known that in time evolution of thin elastic structures the "bending" and "stretching" waves propagate on different time scales, see, e.g., \cite{BCVZ} and Remark \ref{remparam} below.

Rigorous study of thin elastic structures  is an old subject, 
see  \cite{ciarlet}  and references therein for the linear theory.
Derivation of various linear rod models in static and evolution case can be found in \cite{juraktutek, juraktambaca1, juraktambaca2, tambaca}. 
 These works, however, do not contain error estimates and consequently do not introduce higher-order correctors. 
Spectral analysis for the case of finite plates, including a derivation of estimates on the eigenvalues,  is carried out in \cite{Dauge} for the case of a homogeneous, isotropic material. In the case of small-frequency spectrum (namely, the spectrum of order $h^2$, where $h$ is the thickness), the constant in the estimate blows up on any fixed compact interval as $h$ goes go zero,  
 We believe that our work opens the possibility for quantitative analysis of time evolution for thin elastic structures in the context of linear elasticity on finite and infinite domain in different time scales.  

The derivation of different non-linear models of rods, starting from $\rm 3D$ non-linear elasticity, is carried out in \cite{mora1, mora2, scardia1, scardia2} by means of $\Gamma$-convergence. 
 Non-linear problems are challenging and one also has to deal with non-uniqueness of the solution, which is one of the obstacles for the quantitative analysis.  

 Next we outline some works  dealing with  simultaneous homogenisation and dimension reduction.  
In  \cite{Caill} the author derives a limit plate model, where the material is assumed to be isotropic and the oscillations are periodic. In \cite{Dam}, the authors  also do the simultaneous homogenisation and dimension reduction in the case of plates without the assumption on periodicity and using material (planar) symmetries of the elasticity tensor, by introducing the notion of $H$-convergence adapted to dimension reduction. 
Derivation of the non-linear plate model in von K\'{a}rm\'{a}n regime by simultaneous homogenisation and dimension reduction is obtained  in \cite{NV_vK}.
In \cite{BV} the authors derive limit plate models by doing simultaneous homogenisation and dimension reduction in the general case (without the assumption on periodicity and material symmetries) by means of $\Gamma$-convergence (the analysis presented there also covers some non-linear models).  The derivation of the model of the non-linear rod in the bending regime by doing simultaneous homogenisation and dimension reduction and without the assumption on periodicity is  given in \cite{Marohnicvelcic}.   However, these works do not provide any error estimates. 

For an extensive overview of models of composite structures, one can consider the book by Panasenko \cite{Panasenko_book} in which one can find thorough exposure of asymptotic expansions for the models of thin heterogeneous elastic structures (with periodically oscillating material), where the full asymptotics with error estimates and boundary layer analysis is given. However, the constants in the error estimates obtained there in the case of heterogeneous plates and rods with oscillating material depend non-linearly on the loads, which makes these estimates not useful for the spectral analysis.

We next briefly outline the structure of the paper. In Section \ref{section2} we explain the problem, introduce the methods and state the main results and we explain the strategy of the proofs.  In Section \ref{section3} we provide apriori estimates necessary for the asymptotic expansions of the resolvents, as well as spectral estimates which serve as the motivation for different problem scalings. In Section \ref{section4} we establish the resolvent asymptotics with respect to the parameter of quasimomentum in the case of additional assumptions on the material symmetries. In Section \ref{section5} we combine the obtained results into the norm-resolvent estimates in the real domain, but only in the case of additional material symmetries. In Section \ref{section6} we finally are able to repeat the procedure and derive the norm-resolvent estimates for the case of general tensor.

We conclude this section by making some remarks concerning the notation used in the paper. In bounds and estimates, we denote by $C,$ $C_1, C_2$ constants that are independent of the data of interest; they may depend on some parameters of the problem under study, in which case we will make this clear. Throughout the paper, the symbol $\perp$ stands for the property of orthogonality of subspaces of a larger space and also, as a superscript, denotes the orthogonal complement with respect to a specified inner product. 

We use boldface for vector fields and normal type for independent variables. For the components of a vector-valued function, we interchangeably use normal type with an index indicating the component and boldface with parenthesis with a similar index. For example, $u_2$ and $({\vect u})_2$ both denote the second component of a vector function ${\vect u}.$

 For 3-component vector functions and points in a 3-dimensional Euclidean space, we use $\,\widehat{\phantom{a}}\,$  above the vector of the first two components or above the point in a 2-dimensional Euclidean space generated by the first two coordinates.   
 
 Inner products in finite-dimensional vector spaces are denoted by $\,\cdot\,$ and, in the particular case of a linear space of matrices, by $\,:\,$ (semicolon). Additionally, for finite-dimensional spaces over $\mathbb C,$ we write a bar $\,\overline{\phantom{a}}\,$ over the argument with respect to which the inner product is anti-linear. For inner products in infinite-dimensional spaces, we use the brackets $\,\langle\ ,\ \rangle\,$, usually with a subscripts indicating the relevant space. For a matrix $M,$ we use the notation $\sym M=(M+M^\top)/2.$ The symbol $\sym$ is also used on its own to denote the operation itself; for example $\simgrad$ stands for the operation of taking the symmetrised gradient of a function and $(\simgrad)^*$ stands for the corresponding adjoint operation.
 
 The notation for spaces of vector functions usually shows the destination space (such as ${\mathbb R}^2$ or ${\mathbb C}^3$), for example $H^1({\mathbb R}; {\mathbb C}^3)$ is the Sobolev space $H^1$ of vector functions defined on ${\mathbb R}$ and taking values in ${\mathbb C}^3.$ The destination space of scalar functions (i.e. $\mathbb R$ or $\mathbb C$) is normally omitted; so, for example $L^2(\omega\times Y)$ is, depending on the context, the space of ${\mathbb R}$-valued or ${\mathbb C}$-valued functions defined on $\omega\times{\mathbb R}.$ 

Finally, for linear spaces  $X_1,$ $X_2,$   we write  $X_1<X_2$ whenever $X_1$ is a subspace of $X_2.$ In estimates, we will denote by $C$ the generic constant whose value is not important but can be inferred if necessary.

\section{Setting and main results}
\label{section2}

\subsection{Elastic heterogeneous rod}

\begin{figure}[htb]
	\begin{center}
		\includegraphics[width = 12cm]{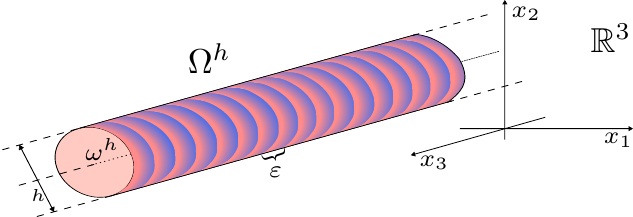}
		\caption{Rod with cross-section of width $h$ and oscillations of material properties of period $\varepsilon$.}
	\end{center}
\end{figure}

We begin by defining the spatial region that will henceforth represent an infinite thin rod. Fix $h>0$ (the width of the rod), a bounded Lipschitz domain $\omega\subset \mathbb{R}^2,$ and denote  by  $\omega^h$ the  homothetic  contraction of $\omega$ such that $|\omega^h|=h^2|\omega|.$ 

Without loss of generality, one has
\begin{equation}
    \label{coordinatesymmetries}
    \int_\omega x_1 = 0, \quad \int_\omega x_2 =0,\quad \int_\omega x_1x_2 = 0,
\end{equation}
which can be achieved by a suitable translation and rotation of $\omega.$
Again without loss of generality, by utilising an appropriate scaling, we can assume that $|\omega|=1$ and define the constants
\begin{equation}
\label{crossectionconstants}
    {\mathfrak c}_1(\omega):= \int_\omega x_1^2, \qquad {\mathfrak c}_2(\omega):=\int_\omega x_2^2 .
\end{equation}
The thin infinite rod that we consider in what follows is represented by the region $\Omega^h:= \omega^h\times \mathbb{R}.$ The material heterogeneity of the rod is introduced as follows. Fix $\varepsilon>0$ (the period of the oscillations of material properties) and consider a ``unit cell"  $Y:=[-1/2,1/2] \subset \R.$
The elastic properties of the heterogeneous material filling the region $\Omega^h$ are given by the elasticity tensor  $\mathbb{A}^h$. 
In order to define it, we introduce the elasticity tensor

$$
 L^\infty(Y;\R^{3\times 3\times 3\times 3})\ni{\mathbb A}: Y\to \R^{3\times 3\times 3\times 3}, 
$$ 
defined on the unit cell
and then extended to  $\mathbb{R}$  by $Y$-periodicity. We assume that  ${\mathbb A}$ is uniformly positive definite on symmetric matrices, namely there exists $\nu>0$ such that
\begin{equation}
    \label{pointwisecoercivity}
    \nu|\xi|^2 \leq {\mathbb A}(y)\xi : \xi \leq\nu^{-1}|\xi|^2, \qquad {\rm a.e.\ } y\in Y,\quad \forall \xi \in \R^{3\times 3},\quad \xi^\top = \xi.
\end{equation}
In addition, we impose the following restrictions on the material coefficients,  which are the consequence of symmetry of Cauchy stress tensor and frame indifference: 
\begin{equation}
    {\mathbb A}_{ijkl}(y)={\mathbb A}_{jikl}(y)={\mathbb A}_{klij}(y),\qquad {\rm a.e.\ } y\in Y, \quad  \forall i,j,k,l\in\left\{1,2,3\right\}.
    \label{test1}
\end{equation}
For any point  $(x_1^h,x_2^h,x_3^h)\in \Omega^h,$ the elasticity tensor $\mathbb{A}^h$ of the $h$-problem  is now given by ${\mathbb A}^h(x_1^h,x_2^h,x_3^h):={\mathbb A}(x_3^h/\varepsilon)$.  
We introduce the following assumption, which yields a substantial simplification of the analysis, as we shall see later. Note that it is satisfied, e.g., by isotropic materials but also covers more general material setups. We carry out the analysis first with and then without this assumption, to showcase the different phenomena occurring in the rod dynamics. 
\begin{assumption}
\label{matsym} 
\begin{enumerate}
\item The cross-section $\omega$ is centrally symmetric with respect to the origin.
	
\item  The elasticity tensor satisfies the following material symmetries:
\begin{equation}
\label{materialsym}
{\mathbb A}_{ijk3}(y) = 0,\  {\mathbb A}_{i333}(y) = 0 \qquad {\rm a.e.\ } y\in Y, \quad  \forall i,j,k \in \left\{1,2\right\}.
\end{equation}
\end{enumerate} 
\end{assumption}

In the present work, we are interested in the regime where the period of material oscillations is of the same order as the thickness of the rod, and 
we assume for simplicity that $\varepsilon = h$.
We study the system of resolvent equations for the operator of three-dimensional linear elasticity 
defined via the bilinear form
\begin{equation*}
    H^1(\omega^\varepsilon\times \R;\R^3) \times H^1(\omega^\varepsilon\times \R;\R^3) \ni (\vect u, \vect v) \to \int_{\Omega^\varepsilon} {\mathbb A}\left(\frac{x_3^\varepsilon}{\varepsilon}\right) \simgrad \vect u : \simgrad \vect v dx.
\end{equation*}
As is standard in the context of dimension reduction, we transform the problem onto the ``canonical" domain $\omega\times{\mathbb R}:$ 
\begin{equation*}
    \omega^\varepsilon \times \R \ni (x^\varepsilon_1,x^\varepsilon_2,x^\varepsilon_3) = x^\varepsilon \to x = (x_1,x_2,x_3) = (\varepsilon^{-1}x^\varepsilon_1,\varepsilon^{-1}x^\varepsilon_2,x^\varepsilon_3) \in  \omega \times \R.
\end{equation*}
This change of coordinates allows us to work on a fixed, i.e. $\varepsilon$-independent domain. We now consider the following bilinear form: 
\begin{equation}
	\label{bilinear_form}
 {\mathfrak a}_\varepsilon:H^1(\omega\times \R;\R^3) \times H^1(\omega\times \R;\R^3)\to \mathbb{R},   \quad {\mathfrak a}_\varepsilon(\vect u, \vect v)= \int_{\omega \times \R} {\mathbb A}\left(\frac{x_3}{\varepsilon}\right) \simgrad_\varepsilon \vect u : \simgrad_\varepsilon \vect v dx,
\end{equation}
where the scaled gradient $\nabla_\varepsilon$ is defined by
\begin{equation}
    \nabla_\varepsilon \vect u(x) := \begin{bmatrix}
       \varepsilon^{-1} 
       \partial_1 u_1 & 
       \varepsilon^{-1}\partial_2 u_1 & \partial_3 u_1 \\[0.45em]
        \varepsilon^{-1}\partial_1 u_2 & 
        \varepsilon^{-1}\partial_2 u_2 & \partial_3 u_2 \\[0.45em]
        \varepsilon^{-1}
        \partial_1 {u_3} & 
        \varepsilon^{-1}
        \partial_2 {u_3} & \partial_3 {u_3} 
    \end{bmatrix}.
\end{equation}
The associated operator $\mathcal{A}_\varepsilon : \mathcal{D}(\mathcal{A}_\varepsilon)\to L^2(\omega\times \R;\R^3)$ is closed, densely defined in $L^2(\omega\times \R;\R^3),$ and self-adjoint.

\begin{figure}[htb]
	\begin{center}
		\includegraphics[width = 15cm]{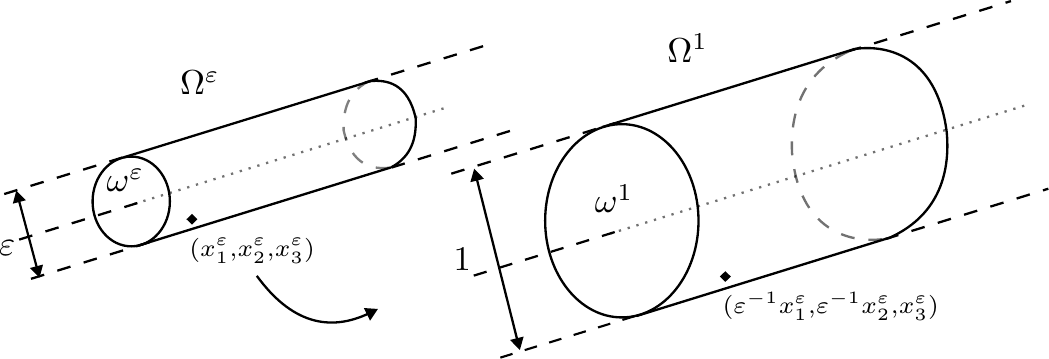}
		\caption{Rescaling to the rod of width $1$.}
	\end{center}
\end{figure}

In our analysis, we will make use of the orthogonal decomposition of the space $L^2(\omega \times \R;\R^3)$ into two subspaces $L^2_{\rm bend}$ and $L^2_{\rm stretch}$ defined by 
\begin{equation}
\label{invariant_subspaces_real}
\begin{aligned}
        L^2_{\rm bend}\equiv L^2_{\rm bend}(\omega\times \R;\R^3)
        &:=\left\{\vect u \in L^2(\omega\times \R;\R^3):\ \widehat{u}\bigl(-\widehat{x}\bigr) = \widehat{u}(\widehat{x}), \quad {u_3}\bigl(-\widehat{x}\bigr) = -{u_3}(\widehat{x})\right\},
  \\[0.35em]
        L^2_{\rm stretch}\equiv L^2_{\rm stretch}(\omega\times \R;\R^3)
        &:=\left\{\vect u \in L^2(\omega\times \R;\R^3):\  \widehat{u}\bigl(-\widehat{x}\bigr) = \vect -\widehat{u}(\widehat{x}), \quad {u_3}\bigl(-\widehat{x}\bigr) = {u_3}(\widehat{x})\right\}.  
 \end{aligned} 
\end{equation}
We often refer to the force densities belonging to $L^2_{\rm bend}$ and $L^2_{\rm stretch}$ as out-of-line and in-line forces, respectively. 
 The reason why Assumption \ref{matsym} simplifies the analysis is that when it is satisfied the 
 subspaces $L^2_{\rm bend}$ and $L^2_{\rm stretch}$ are invariant under the elasticity operator,  i.e.,  whenever the density of applied forces is an element of one of these two subspaces, the corresponding deformation belongs to the same subspace; 
we refer to such deformations as bending and stretching deformations, respectively.

Although we do not assume the dependence of the elasticity tensor on the coordinates $x_1$, $x_2,$ it is possible to obtain the results even for such generality. For simplicity of exposition we do not do that. 

\subsection{Homogenised operators}
\label{hom_op_sec}

In order to define homogenised limit operators,  see \eqref{bend_stretch}--\eqref{3domains} below,  we make use of the following matrices defined for each ${\vect m} = (m_1,m_2,m_3,m_4) \in \R^4$ (or $\C^4$), 
$\widehat{x}:=(x_1, x_2)\in \omega:$
\begin{align}
 \mathcal{J}^{\rm bend}_{m_1,m_2}(\widehat{x})&= \begin{bmatrix}
     0 & 0 & 0 \\[0.3em]
     0 & 0 & 0 \\[0.3em]
     0 & 0 &   -x_1 m_1 - x_2 m_2
\end{bmatrix}, \quad  \mathcal{J}_{m_3,m_4}^{\rm stretch}(\widehat{x}) = \begin{bmatrix}
     0 & 0 & \dfrac{x_2 m_3}{2} \\[0.7em]
     0 & 0 & \dfrac{-x_1 m_3}{2} \\[0.7em]
     \dfrac{x_2 m_3}{2} & \dfrac{-x_1 m_3}{2} &  m_4 
\end{bmatrix}\label{I_bend_stretch} \\
        \mathcal{J}^{\rm rod}_{\vect m}(\widehat{x})&=  \mathcal{J}^{\rm bend}_{m_1,m_2}(\widehat{x})+\mathcal{J}_{m_3,m_4}^{\rm stretch}(\widehat{x}). 
\nonumber
\end{align}
 We denote by \CCC $H_{\#}^1( Y;H^1(\omega;\C^3))$ and  $H_{\#}^1( Y;H^1(\omega;\R^3))$  the Sobolev spaces of, respectively, the complex-valued and  real-valued  $H^1$ functions of $(x_1, x_2, y)\in\omega\times Y$ that are $Y$-periodic in the variable $y.$ 
  Similarly, we define $L_{\#}^2( Y;H^1(\omega;\C^3))$ and  $L_{\#}^2( Y;H^1(\omega;\R^3)),$ whose elements are assumed to be extended $Y$-periodically to whole of $\mathbb{R}$.

 The homogenised tensor ${\mathbb A}^{\rm rod}$ describing the material properties of homogeneous rod is defined as follows. 
We consider the bilinear form
\begin{equation}
    {\mathfrak a}^{\rm rod}({\vect m}, {\vect d}):= \int_{\omega \times Y} {\mathbb A}(y) \left( \mathcal{J}^{\rm rod}_{\vect m}(\widehat{x}) + \simgrad{\vect u}_{\vect m}(\widehat{x}, y)\right):\mathcal{J}^{\rm rod}_{\vect d}(\widehat{x}) d\widehat{x}dy,\qquad{\vect m}, {\vect d}\in{\mathbb R}^4,
\end{equation}
where ${\vect u}_{\vect m} \in H_{\#}^1( Y;H^1(\omega;\R^3))$ is the unique solution of the integral identity  (i.e., of the ``cell problem"):
\begin{equation}
\label{correctordefinition}
\int_{\omega \times Y} {\mathbb A}(y) \left( \mathcal{J}^{\rm rod}_{\vect m}(\widehat{x}) + \simgrad {\vect u}_{\vect m}(\widehat{x}, y)\right) 
 :  \simgrad \vect v(\widehat{x},y)\,d\widehat{x}dy = 0 \qquad \forall \vect v \in  H_{\#}^1\bigl(Y;H^1(\omega;\R^3)\bigr).
\end{equation}
 The following proposition implies that the limit operators are non-degenerate. We provide its proof in Section \ref{appproof}. 

\begin{proposition}
	\label{proposition21}
The form ${\mathfrak a}^{\rm rod}$ is 
uniquely represented by a \CCC symmetric  tensor ${\mathbb A}^{\rm rod}\in \R^{4\times 4},$ in the sense that
\begin{equation}
    {\mathfrak a}^{\rm rod}({\vect m},{\vect d}) =  {\mathbb A}^{\rm rod}{\vect m}\cdot{\vect d}\qquad \forall {\vect m}, {\vect d}\in{\mathbb R}^4. 
\end{equation}
The tensor ${\mathbb A}^{\rm rod}$ is positive definite, so there exists $\eta>0$ such that
\CCC
\begin{equation}
\eta |{\vect m}|^2 \leq {\mathbb A}^{\rm rod}{\vect m}\cdot {\vect m}\leq \eta^{-1} |{\vect m}|^2 \qquad\forall {\vect m}\in{\mathbb R}^4.
\label{coerc}
\end{equation}

\end{proposition}

Next, we define  symmetric matrices ${\mathbb A}^{\rm bend}, $ $ {\mathbb A}^{\rm stretch} \in \R^{2 \times 2}$ as follows:
\begin{equation}  \label{gotovo2}
	\begin{split}
		 {\mathbb A}^{\rm bend} (m_1,m_2)^\top \cdot (d_1,d_2)^\top:={\mathbb A}^{\rm rod} (m_1,m_2,0,0)^\top \cdot (d_1,d_2,0,0)^\top \qquad \forall (m_1, m_2)^\top, (d_1, d_2)^\top\in{\mathbb R}^2, 
		\\[0.35em]
		 {\mathbb A}^{\rm stretch} (m_3,m_4)^\top \cdot (d_3,d_4)^\top:={\mathbb A}^{\rm rod} (0,0,m_3,m_4)^\top \cdot (0,0,d_3,d_4)^\top \qquad \forall (m_3, m_4)^\top, (d_3, d_4)^\top\in{\mathbb R}^2.
	\end{split}
\end{equation}
In the case of the cross-section geometry and material symmetries of Assumption \ref{matsym} one has 
\begin{align*}
	{\mathbb A}^{\rm rod}{\vect m}\cdot {\vect d}= {\mathbb A}^{\rm bend} (m_1,m_2)^\top\cdot (d_1,d_2)^\top&+ {\mathbb A}^{\rm stretch} (m_3,m_4)^\top \cdot (d_3,d_4)^\top\\[0.35em]
	&\forall\,{\vect m}=(m_1, m_2,m_3, m_4)^\top,\ {\vect d}=(d_1, d_2,d_3, d_4)^\top\in{\mathbb R}^4.
\end{align*}

The following estimates are a straightforward consequence of Proposition \ref{proposition21}.
\begin{corollary}
	\label{coerc_corrol}
	The matrices ${\mathbb A}^{\rm bend}$ and ${\mathbb A}^{\rm stretch}$ are uniformly bounded and positive definite, with the related bounds provided by the constant $\eta$ of Proposition \ref{proposition21}: 
\CCC
\begin{equation}
\begin{aligned}
   &\eta  |(m_1,m_2)^\top|^2 \leq {\mathbb A}^{\rm bend} (m_1,m_2)^\top\cdot(m_1,m_2)^\top\leq \eta^{-1}  |(m_1,m_2)^\top|^2\qquad \forall(m_1,m_2)^\top\in \R^2
    \\[0.4em]
   &  \eta |(m_3,m_4)^\top|^2 \leq{\mathbb A}^{\rm stretch} (m_3,m_4)^\top\cdot (m_3,m_4)^\top \leq \eta^{-1} |(m_3,m_4)^\top|^2\qquad \forall(m_3,m_4)^\top\in \R^2.
\end{aligned}
\end{equation}

\end{corollary}

We conclude this section by defining the homogenised limit differential operators. These correspond to the differential expressions
\begin{align}
&\mathcal{A}^{\rm bend} = 
{\mathbb A}^{\rm bend}{\frac{d^4}{dx_3^4}}, \quad \mathcal{A}^{\rm stretch} = 
-{\mathbb A}^{\rm stretch}{\frac{d^2}{dx_3^2}},\label{bend_stretch}\\
   &\mathcal{A}^{\rm rod}_{\varepsilon}= \left(\varepsilon \frac{d^2}{dx_3^2}, \varepsilon \frac{d^2}{dx_3^2}, -\frac{d}{dx_3}, -\frac{d}{dx_3} \right)^\top{\mathbb A}^{\rm rod} \left(\varepsilon \frac{d^2}{dx_3^2}, \varepsilon \frac{d^2}{dx_3^2}, \frac{d}{dx_3}, \frac{d}{dx_3} \right)^\top\label{rod}
\end{align}
and  are  considered on the domains 
\begin{equation}
    \mathcal{D}(\mathcal{A}^{\rm bend}):=H^4(\R;\R^2), \quad  \mathcal{D}(\mathcal{A}^{\rm stretch}):=H^2(\R;\R^2), \quad \RRR \mathcal{D}(\mathcal{A}^{\rm rod}_{\varepsilon})\subset H^2(\R;\R^2) \times H^1(\R;\R^2), 
    \label{3domains}
\end{equation}
respectively.\footnote{\RRR The operators \eqref{bend_stretch}, \eqref{rod} with domains \eqref{3domains} can be defined through the associated bilinear forms.  Note that,  due to the coupling of the components, we cannot stipulate the domain of the operator $\mathcal{A}^{\rm rod}_{\varepsilon}$ to be $H^4(\R;\R^2) \times H^2(\R;\R^2)$. } In the expressions \eqref{bend_stretch} the derivatives {$d^4/dx_3^4,$ $d^2/dx_3^2$} are applied to every component of the vector field, while in the expression 
\eqref{rod} the components of the vectors
\[
\left(\varepsilon \frac{d^2}{dx_3^2}, \varepsilon \frac{d^2}{dx_3^2}, \frac{d}{dx_3}, \frac{d}{dx_3} \right), \qquad \left(\varepsilon \frac{d^2}{dx_3^2}, \varepsilon \frac{d^2}{dx_3^2}, -\frac{d}{dx_3}, -\frac{d}{dx_3} \right)
\]
are applied to the respective components of the vector field  (thus resulting in a 4-dimensional vector field.) 


\subsection{Main results}
\label{main_results_sec}
In approximating  ${\mathcal A}_\varepsilon$ by the homogenised operators of Section \ref{hom_op_sec}, the following force-and-momentum operators acting on functions \RRR ${\vect f}\in L^2({\mathbb R}; {\mathbb R}^3)$  prove useful: 
\begin{equation}
\label{forcemomentumrealdomain}
\begin{aligned}
        (\mathcal{M}_\varepsilon^{\rm bend})
         {\vect f}(x_3)&:= \int_{\omega}\biggl\{\,\widehat{\!\vect f}(\widehat{x}, x_3)
         -\varepsilon\biggl({\dfrac{\partial}{\partial x_3}}{f_3}(\widehat{x}, x_3)\biggr)\widehat{x}\biggr\},\qquad x_3\in  \mathbb{R}  ,\\[0.9em]
         (\mathcal{M}^{\rm stretch} 
        {\vect f})(x_3)&:=\int_{\omega}\begin{bmatrix} x_2 {f_1}(\widehat{x}, x_3)-x_1 {f_2}(\widehat{x}, x_3) \\[0.2em] 
        	{f_3}(\widehat{x}, x_3) \end{bmatrix},\qquad x_3\in  \mathbb{R}  ,
        \\[0.9em]
        (\mathcal{M}_\varepsilon^{\rm rod} 
        {\vect f})(x_3)&:= \begin{bmatrix}
           \mathcal{M}_\varepsilon^{\rm bend} \vect f \\[0.3em] \mathcal{M}^{\rm stretch} \vect f
        \end{bmatrix}(x_3) = \int_{\omega}\begin{bmatrix}
          \,\widehat{\!\vect f}(\widehat{x},x_3)-\varepsilon\biggl({\dfrac{\partial}{\partial x_3}}{f_3} (\widehat{x}, x_3)\biggr)\widehat{x}
          \\[0.8em]
          x_2 {f_1}(\widehat{x}, x_3) - x_1 {f_2}(\widehat{x}, x_3) \\[0.4em] {f_3}(\widehat{x}, x_3)
        \end{bmatrix}d\widehat{x},\qquad x_3\in  \mathbb{R}  .
\end{aligned}
\end{equation}
The operator $\mathcal{M}^{\rm stretch}$ takes values in $L^2(\R; \R^2),$ while the operators $\mathcal{M}_\varepsilon^{\rm  bend}$ and $\mathcal{M}_\varepsilon^{\rm  rod}$ take values in $H^{-1}(\R; \R^2)$ and $H^{-1}(\R; \R^3),$ respectively. However, in what follows these operators will appear in combination with the smoothing operator $\Xi_\varepsilon$ (see Section \ref{smoothing_op_sec}), and hence  the values of all three of them  will happen to be found in appropriate $L^2$ spaces.

 In order to allow for the possibility of scaling the loads, we introduce the following scaling matrix: 
\begin{equation}
    S_\varepsilon:= \begin{bmatrix}
        1 & 0 & 0 \\[0.25em] 0 & 1 & 0 \\[0.25em] 0 & 0 & \varepsilon^{-1}
    \end{bmatrix}, \quad \varepsilon > 0.
\label{Svarepsilon}
\end{equation}
We label with $P_i:\R^3 \to \R$ the projection on the $i$-th \CCC coordinate.


We also define the following matrices, which
contain information about the cross-section of the domain $\omega:$
\begin{equation}
	\label{cstretchrodbend1}
		\mathfrak{C}^{\rm stretch}(\omega)= 
		\begin{bmatrix}
			{\mathfrak c}_1(\omega)+{\mathfrak c}_2(\omega) & 0 \\[0.3em]
			0 & 1
		\end{bmatrix}, \quad 
		 \mathfrak{C}^{\rm rod} (\omega):= \begin{bmatrix}
			I & 0 \\[0.3em]
			0 & \mathfrak{C}^{\rm stretch}(\omega)
		\end{bmatrix}
\end{equation}
In what follows, we usually drop `$(\omega)$' in the notation.

We define a smoothing operator $\Xi_\varepsilon :L^2(\omega \times \R) \to L^2(\omega \times \R)$ by the formula
\begin{equation}
\label{smootheningoperator}
\Xi_\varepsilon f:= \Bigl({\mathcal F}^{-1}\bigl[\mathbbm{1}_{\left[
	\frac{-1}{2\varepsilon},\frac{1}{2\varepsilon}\right]}\bigr] * f\Bigr)(x),\qquad x\in\omega\times{\mathbb R}.
\end{equation}
Here $\mathcal{F}:L^2(\omega \times \R)\to L^2(\omega \times \R)$ denotes Fourier transform with respect to the third variable,  $\mathcal{F}^{-1}$ is its inverse and $*$ is a convolution  with respect to the third variable. We will also apply this operator to vector-valued functions by assuming that it acts component-wise.

The operator $\Xi_\varepsilon$ appears in the definition of the approximating operator, see Theorems \ref{THML2L2}, \ref{l2h1theorem}, \ref{thm_gen_L2L2high}. The purpose of $\Xi_\varepsilon$ is cutting off high-frequency contributions to a function, by removing the Fourier components with frequencies higher than $(2\varepsilon)^{-1}$.

Note that estimates of resolvents that depend on the spectral parameter $z \in \C$ can usually be reduced to a single resolvent estimate where the value of $z$ is fixed, by appropriately modifying the corresponding resolvent problem, see Section \ref{appRes} and Lemma \ref{lemres}. We express our results by fixing $z=-1$.

The following result, which concerns the resolvent estimates in the $L^2\to L^2$ operator norm, is proved in Sections \ref{L2toL2}, \ref{gen_tens_sec} \CCC and is one of the three main theorems  of the paper.   
  For greater generality, our quantitative result is expressed in terms of two additional scaling exponents, which we denote by $\gamma$ and $\delta$.  The exponent $\gamma$ is used to scale (as $\varepsilon^\gamma$) the operator (in particular, its spectrum), while the parameter $\delta$ is used to scale the loads. On the one hand, since the limit equations are necessarily partially dispersive, it is important to obtain a result depending on $\gamma\ge0$, which would correspond to different time scalings in the analysis of time evolution. On the other hand, scaling of the loads is characteristic of the dimension reduction problems, see Remark \ref{remparam} below. Our error estimates depend on the parameter $\delta$ under Assumption \ref{matsym}, which thereby enables a stronger quantitative result.

\begin{theorem}[$L^2 \to L^2$ norm-resolvent estimate]
\label{THML2L2}
 Let the parameters $\gamma,\delta$ satisfy $\gamma>-2$ and $\delta \geq 0,$ respectively.  There exists $C>0$ such that for every $\varepsilon> 0$ one has
\begin{equation}
	\label{genres}
\left\Vert 
	P_i 
		\left( 
			\left( \frac{1}{\varepsilon^\gamma}\mathcal{A}_\varepsilon + I \right)^{-1} 
- (\mathcal{M}^{\rm rod}_\varepsilon)^*\left(
     \frac{1}{\varepsilon^\gamma} \mathcal{A}^{\rm rod}_{\varepsilon} + \mathfrak{C}^{\rm rod} \right)^{-1}\mathcal{M}^{\rm rod}_\varepsilon \Xi_\varepsilon
		\right) 
 \right\Vert_{L^2 \to L^2} \leq  
 {C}\left\{ \begin{array}{ll}
         \varepsilon^{\tfrac{\gamma + 2}{4}}, & \mbox{ $i = 1,2$};\\[0.7em]
         \varepsilon^{\tfrac{\gamma + 2}{2}}, & \mbox{ $i = 3$}.\end{array} \right. 
\end{equation}
Under the additional assumption of the cross-section geometry and material symmetries (Assumption \ref{matsym}), one has
\begin{align}
&\left\Vert  \left( \frac{1}{\varepsilon^\gamma}\mathcal{A}_\varepsilon + I \right)^{-1}\bigg|_{L^2_{\rm stretch}} - (\mathcal{M}^{\rm stretch})^*\left(\frac{1}{\varepsilon^\gamma}\mathcal{A}^{\rm stretch} + \mathfrak{C}^{\rm stretch} \right)^{-1}\mathcal{M}^{\rm stretch}\Xi_\varepsilon \right\Vert_{L^2 \to L^2} \leq C \varepsilon^{\tfrac{\gamma + 2}{2}} ,\label{stretch23}
\\[1.0em]
&\begin{aligned}
&\left\Vert 
	P_i\left(\left( \frac{1}{\varepsilon^\gamma}\mathcal{A}_\varepsilon + I \right)^{-1}\bigg|_{L^2_{\rm bend}}
- (\mathcal{M}_\varepsilon^{\rm bend})^*\left(\frac{1}{\varepsilon^{\gamma - 2}}\mathcal{A}^{\rm bend} + I \right)^{-1}\mathcal{M}_\varepsilon^{\rm bend} 
	\Xi_\varepsilon	\right)  S_{\!\varepsilon^\delta}
 \right\Vert_{L^2 \to L^2} \\[1.1em]
  &\hspace{+20em}\leq  
  C{\max\Bigl\{\varepsilon^{\tfrac{\gamma + 2}{4}-\delta}, 1\Bigr\}}
\left\{ \begin{array}{ll}
         \varepsilon^{\tfrac{\gamma + 2}{4}},
         & i = 1,2;\\[0.9em]
         \varepsilon^{\tfrac{\gamma + 2}{2}},
         & i = 3.\end{array} \right.
\end{aligned} 
\label{bend23}
\end{align}
\end{theorem}

\begin{remark}
\label{newnotationformomentums}
In Theorem \ref{THML2L2}, the operators  $(\mathcal{M}_\varepsilon^{\rm rod})^*$, $(\mathcal{M}_\varepsilon^{\rm bend})^*$ and $(\mathcal{M}^{\rm stretch})^*$ denote the adjoints of the corresponding force-and-momentum operators,  which are discussed next.

Recall that the  classical Bernoulli-Navier ansatz for  { rod displacements} is of the form 
$$ \begin{bmatrix}{u_1+x_2 u_3} \\[0.3em]
u_2-x_1 u_3 \\[0.2em] u_4 -\varepsilon {\dfrac{\partial}{\partial x_3}}(x_1 {u_1} + x_2 {u_2})\end{bmatrix},  $$
where $u_3$ accounts for torsion while $u_1$, $u_2$, $u_4$ are the components of the leading term of the displacement. 
Thus this form of the approximation for the displacement must be contained in the approximating operator on the left-hand side of \eqref{genres}, \eqref{stretch23}, \eqref{bend23}. To see this, we introduce the operators

\begin{equation}
	\RRR\begin{aligned}
& \mathcal{I}^{\rm bend}_\varepsilon:H^1(\R;\R^2) \to L^2_{\rm bend}(\omega\times \R;\R^3), \qquad \mathcal{I}^{\rm bend}_\varepsilon\begin{bmatrix}{u_1} \\ {u_2}\end{bmatrix} = \begin{bmatrix}{u_1} \\[0.3em]
	u_2 \\[0.2em] -\varepsilon{\dfrac{\partial}{\partial x_3}}(x_1 {u_1} + x_2 {u_2})\end{bmatrix},\\[1.0em]
&\mathcal{I}^{\rm stretch}:H^1(\R;\R^2) \to L^2_{\rm stretch}(\omega\times \R;\R^3), \qquad   \mathcal{I}^{\rm stretch}\begin{bmatrix}{u_3} \\ u_4\end{bmatrix} = \begin{bmatrix}x_2 {u_3} \\[0.25em] -x_1 {u_3} \\[0.25em] u_4\end{bmatrix}. 
\end{aligned} 
\end{equation}
Note also that the following duality relations hold for smooth functions ${\vect f}$:
\begin{align}
\RRR
\left\langle \mathcal{I}^{\rm bend}_\varepsilon\begin{bmatrix}{u_1} \\ {u_2}\end{bmatrix}, {\vect f}
\right\rangle&= \int_{\omega\times \R} \begin{bmatrix}{u_1} \\[0.3em] {u_2} \\[0.3em] -\varepsilon{\dfrac{\partial}{\partial x_3}}(x_1 {u_1} + x_2 {u_2})\end{bmatrix}\cdot{\vect f}
 \nonumber
  \\[0.9em] 
&=\int_{\R}\begin{bmatrix}{u_1} \\ {u_2}\end{bmatrix}\cdot\int_{\omega}
\begin{bmatrix}
	 \,\widehat{\!\vect f}-\varepsilon\biggl({\dfrac{\partial}{\partial x_3}}{f_3}\biggr)\widehat{x}
  \end{bmatrix} = \left\langle \begin{bmatrix}{u_1} \\ {u_2}\end{bmatrix},\mathcal{M}^{\rm bend}_\varepsilon 
{\vect f}\right\rangle,\label{dual_bend}
\\[0.9em]
\left\langle \mathcal{I}^{\rm stretch}\begin{bmatrix}{u_3} \\ u_4\end{bmatrix}, \vect{f}
\right\rangle&= \int_{\omega\times \R} \begin{bmatrix}x_2 {u_3} \\ -x_1 {u_3} \\ u_4\end{bmatrix}\cdot 
{\vect f} = \int_{ \R}\begin{bmatrix}{u_3} \\ u_4\end{bmatrix}\cdot\int_{\omega}\begin{bmatrix} x_2 {f_1} - x_1 {f_2} \\[0.25em] {f_3}  \\ \end{bmatrix} = \left\langle \begin{bmatrix}{u_3} \\ u_4\end{bmatrix},\mathcal{M}^{\rm stretch}{\vect f}
 \right\rangle.\nonumber
\end{align}
Thus, we  have 
$$
(\mathcal{M}^{\rm bend}_\varepsilon)^*=\mathcal{I}^{\rm bend}_\varepsilon,\qquad
(\mathcal{M}^{\rm stretch})^*=\mathcal{I}^{\rm stretch}, 
$$

where the second relation is rigorous, and the first relation is understood formally in the sense of \eqref{dual_bend}. The latter convention is justified by our earlier observation that in the analysis to follow the operator $\mathcal{M}^{\rm bend}_\varepsilon$  only appears in combination with the smoothing operator $\Xi_\varepsilon.$ Finally, we have
\[\RRR
(\mathcal{M}^{\rm rod}_\varepsilon)^*=\left[\begin{matrix}(\mathcal{M}^{\rm bend}_\varepsilon)^*,\ 
	(\mathcal{M}^{\rm stretch})^*\end{matrix}\right]=\left[\begin{matrix}\mathcal{I}^{\rm bend}_\varepsilon,\ 
	\mathcal{I}^{\rm stretch}\end{matrix}\right]. 
\]

\end{remark}
\begin{figure}[htb]
	\begin{center}
		\includegraphics[width = 15cm]{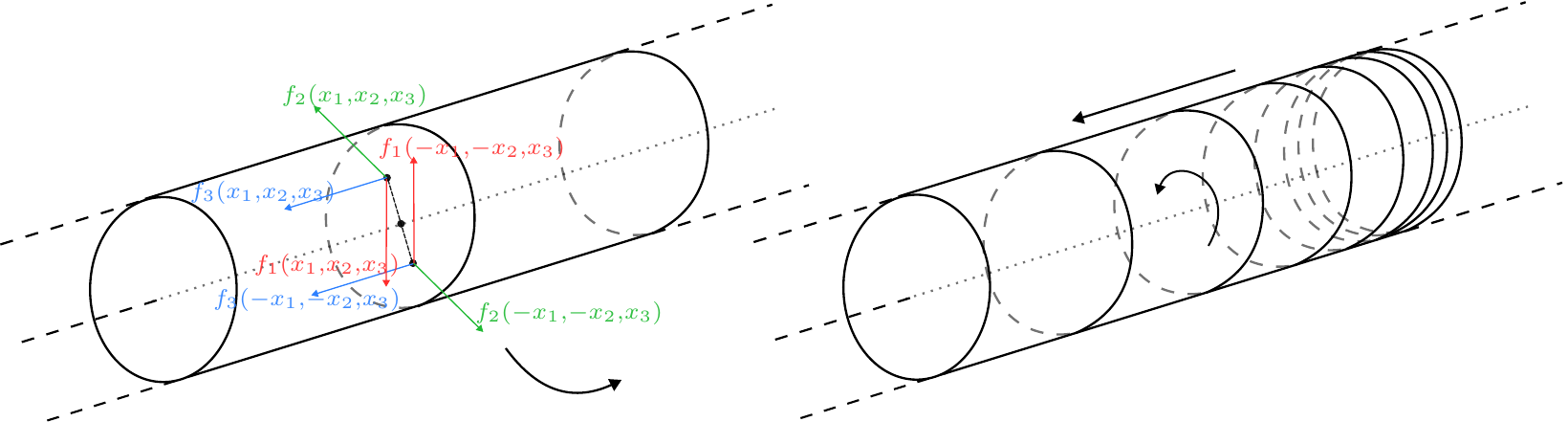}
		\caption{Stretching deformation (with torsion) caused by $\vect f \in L^2_{\rm stretch}$.}
	\end{center}
\end{figure}
\begin{figure}[htb]
	\begin{center}
		\includegraphics[width = 15cm]{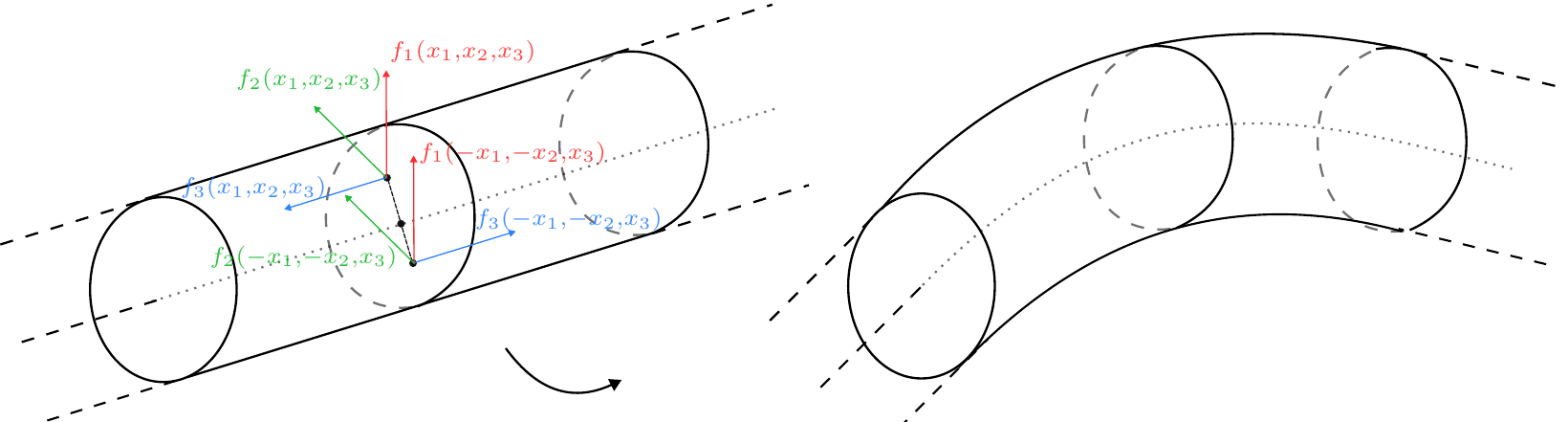}
		\caption{Bending deformation caused by $\vect f \in L^2_{\rm bend}$.}
	\end{center}
\end{figure}

\begin{remark}\label{remparam} 
The role of the parameter $\delta$ in Theorem \ref{THML2L2} is to widen the scope of our analysis, by admitting a variety of force scalings. It is known in the analysis of thin structures (plates and rods) \CCC that,  as a consequence of anisotropy, adopting different load scalings in different directions yields a richer structure of the limiting model, see \cite{ciarlet, Panasenko_book, tambaca}. One could argue that the most interesting cases of parameters $\gamma$, $\delta$  are  $\gamma = 0$, $\delta= 0$ and $\gamma = 2$, $\delta = 1$, since these are the standard regimes  that  emerge when studying thin structures on a finite domain, see \cite{tambaca}  and \cite{BCVZ} for plates in high contrast.  However, in the case of  an  infinite rod (similarly  to the setting  of an infinite plate \cite{cherednichenkovelcic}), there is no natural spectral scaling. The role of the parameter $\gamma$ becomes clear in the time evolution setting, where it serves  for obtaining  the models  for  different time scales. (Note that evolution models of plates and rods are usually analysed with $\gamma=2$, see \cite{Raoult, tambaca}.)
\end{remark}
\begin{remark}
 As shown in Corollary \ref{absenceofforceterms} and Remark \ref{lastremark}, the smoothing operator $\Xi_\varepsilon$ can be removed from the estimates in Theorem \ref{THML2L2} while preserving the order of the estimates. 
\RRR
Also, as shown in Corollary \ref{corr3}, $\mathcal{M}_\varepsilon^{\rm bend}$ can be replaced with $\mathcal{M}_0^{\rm bend}$ at the cost of worsening the estimate \eqref{bend23} in the third component. 
\end{remark}

\begin{remark}
It is possible to obtain the  same estimates as in the Theorem \ref{THML2L2}, even in the case when the ratio $h/\varepsilon$ belongs to  the  interval $[\alpha, \beta]$, where $0<\alpha<\beta$,    but the constant in the estimates will depend on the $\alpha$ and $\beta$ (\CCC see Section 8 in \cite{cherednichenkovelcic}. ) The same is true for Theorem \ref{l2h1theorem} and Theorem \ref{thm_gen_L2L2high}.
\end{remark}

\subsubsection{Higher-precision estimates}

We also establish resolvent estimates in the $L^2 \to H^1$ operator norm.
The proof of the following result is discussed in Section \ref{L2toH1}, (estimates \eqref{precision_stretch}, \eqref{precision_bend}), and in Remark \ref{rem61} (estimate \eqref{genres_L2toH1}) of Section \ref{gen_tens_sec}.

\begin{theorem}[$L^2 \to H^1$ norm-resolvent estimate] 
\label{l2h1theorem}
Suppose that the spectral and force scaling parameters satisfy the conditions $\gamma>-2$ and $\delta \geq 0.$
 Then there exists $C>0,$ independent of $\gamma,$ $\delta,$ such that for every $\varepsilon> 0$ one has 
\begin{equation}
\left\Vert 
	P_i\left(\left( \frac{1}{\varepsilon^\gamma}\mathcal{A}_\varepsilon + I \right)^{-1} 
- (\mathcal{M}^{\rm rod}_\varepsilon)^*\left(\frac{1}{\varepsilon^\gamma} \mathcal{A}^{\rm rod}_{\varepsilon} + \mathfrak{C}^{\rm rod} \right)^{-1}\mathcal{M}^{\rm rod}_\varepsilon \Xi_\varepsilon - \mathcal{A}_{\rm rod}^{\rm \rm corr}(\varepsilon)
		\right) 
 \right\Vert_{L^2 \to H^1} \leq  
 { C}\left\{\begin{array}{ll}
         \varepsilon^{\tfrac{\gamma + 2}{4}}, & \mbox{ $i = 1,2$},\\[0.5em]         
         \varepsilon^{\tfrac{\gamma + 2}{2}}, & \mbox{ $i = 3$}.\end{array} \right.
     \label{genres_L2toH1}
\end{equation}
Under the additional assumption (Assumption \ref{matsym}) of the cross-section and material symmetries, one has
\begin{align}
&\begin{aligned}
&\left\Vert \left( \frac{1}{\varepsilon^{\gamma}}\mathcal{A}_\varepsilon + I \right)^{-1}\bigg|_{L^2_{\rm stretch}} - (\mathcal{M}^{\rm stretch})^*\left(\frac{1}{\varepsilon^{\gamma}}\mathcal{A}^{\rm stretch} + \mathfrak{C}^{\rm stretch} \right)^{-1}\mathcal{M}^{\rm stretch}\Xi_\varepsilon -\mathcal{A}^{\rm \rm corr}_{\rm stretch}(\varepsilon)  \right\Vert_{L^2 \to H^1}\\[0.6em]  
&\hspace{40pt}\leq  C \max\Bigl\{\varepsilon^{\gamma + 1}, \varepsilon^{\tfrac{\gamma + 2}{2}} \Bigr\},
\end{aligned}
\label{precision_stretch}
\\[0.8em]
&\begin{aligned}
&\left\Vert 
	P_i 
		\left( 
			\left( \frac{1}{\varepsilon^\gamma}\mathcal{A}_\varepsilon + I \right)^{-1}\bigg|_{L^2_{\rm bend}} 
- (\mathcal{M}^{\rm bend}_\varepsilon)^*\left(\frac{1}{\varepsilon^{\gamma-2}}\mathcal{A}^{\rm bend} + I \right)^{-1}\mathcal{M}^{\rm bend}_\varepsilon \Xi_\varepsilon
		 -\mathcal{A}^{\rm \rm corr}_{\rm bend}(\varepsilon)\right)S_{\!\varepsilon^\delta}
 \right\Vert_{L^2 \to H^1} \\[0.5em] 
 &\hspace{40pt}\leq 
 {C\max\Bigl\{\varepsilon^{\tfrac{\gamma + 2}{4}- \delta}, 1\Bigr\}} 
 \left\{ \begin{array}{ll}
         \max\Bigl\{\varepsilon^{\tfrac{\gamma + 2}{4}},\varepsilon^{\tfrac{\gamma }{2}}\Bigr\},
         & \mbox{$i = 1,2$},\\[0.9em]
          \max\Bigl\{\varepsilon^{\tfrac{\gamma + 2}{2}},\varepsilon^{\tfrac{3\gamma + 2}{4}}\Bigr\},
           & \mbox{$i = 3$}.\end{array} \right. 
\end{aligned}
\label{precision_bend}
\end{align}
\end{theorem}
The operators $\mathcal{A}_{\rm rod}^{\rm \rm corr}(\varepsilon),$ $\mathcal{A}_{\rm bend}^{\rm \rm corr}(\varepsilon),$ and $\mathcal{A}_{\rm stretch}^{\rm \rm corr}(\varepsilon)$ are the standard 
first-order correctors of the theory of homogenisation, defined by the expression \eqref{trans_back_corr1_str} below.

Finally, our asymptotic analysis yields further correctors $\widetilde{\mathcal{A}}_{\rm rod}^{\rm \rm corr}(\varepsilon),$ $\widetilde{\mathcal{A}}_{\rm bend}^{\rm \rm corr}(\varepsilon),$ and $\widetilde{\mathcal{A}}_{\rm stretch}^{\rm \rm corr}(\varepsilon)$ that allow us to calculate $L^2 \to L^2$ norm resolvent estimates with even higher precision. These correctors, which are expressed explicitly as pseudodifferential operators, have been unknown in the theory of heterogeneous elastic rods. They resemble the higher-order correctors appearing in the works of Birman and Suslina in the standard homogenisation setting, see e.g. \cite{BirmanSuslina_corrector}. For the precise definition of the corrector operators in the context of rods, see \eqref{higherorderl2l2correctors} and Remark \ref{rem61}. The proof of the following result is discussed in Section \ref{higher_order_sec} (stretching and bending cases, see estimates \eqref{stretch_higher}, \eqref{bend_higher}) and Remark \ref{rem61} (general case, see estimate \eqref{gen_higher}).

\begin{theorem}[Higher-order $L^2 \to L^2$ norm-resolvent estimate]
\label{thm_gen_L2L2high}
 Suppose that the spectral scaling and force scaling parameters satisfy the conditions $\gamma>-2$ and $\delta\geq 0,$ respectively. 
Then there exists $C>0,$ independent of $\gamma,$ $\delta,$ such that for every $\varepsilon> 0$ one has 
\begin{equation}
\begin{split}
\left\Vert 
	P_i 
		\left\{
			\left( \frac{1}{\varepsilon^\gamma}\mathcal{A}_\varepsilon + I \right)^{-1} 
- (\mathcal{M}^{\rm rod}_\varepsilon)^*\left(\frac{1}{\varepsilon^\gamma} \mathcal{A}_{\rm rod, \varepsilon}^{\rm hom} + \mathfrak{C}^{\rm rod} \right)^{-1}\mathcal{M}^{\rm rod}_\varepsilon \Xi_\varepsilon - \mathcal{A}_{\rm rod}^{\rm \rm corr}(\varepsilon) - \mathcal{\widetilde{A}}^{\rm \rm corr}_{\rm rod}(\varepsilon)
		\right\}
 \right\Vert_{L^2 \to L^2} \\ \leq  
 {C}\left\{ \begin{array}{ll}
         \varepsilon^{\tfrac{\gamma + 2}{2}}, & \mbox{ $i = 1,2$},\\[0.6em]
         \varepsilon^{\tfrac{3(\gamma + 2)}{4}}, & \mbox{ $i = 3$}.\end{array} \right.
\end{split}
\label{gen_higher}
\end{equation}
Under the additional assumption (Assumption \ref{matsym}) on the cross-section and material symmetries, one has
\begin{align}
&\begin{aligned}
&\left\Vert \left( \frac{1}{\varepsilon^\gamma}\mathcal{A}_\varepsilon + I \right)^{-1}\bigg|_{L^2_{\rm stretch}} - (\mathcal{M}^{\rm stretch})^*\left(\frac{1}{\varepsilon^\gamma} \mathcal{A}^{\rm stretch} + \mathfrak{C}^{\rm stretch} \right)^{-1}\mathcal{M}^{\rm stretch}\Xi_\varepsilon  -\mathcal{A}^{\rm \rm corr}_{\rm stretch}(\varepsilon) -\mathcal{\widetilde{A}}^{\rm \rm corr}_{\rm stretch}(\varepsilon)  \right\Vert_{L^2 \to L^2}\\[0.9em] 
&\hspace{60pt}\leq C \varepsilon^{\gamma + 2},
\end{aligned}\quad\ 
\label{stretch_higher}
\\[0.3em]
&\begin{aligned}
&\Bigg\|P_i 
		\left\{
			\left( \frac{1}{\varepsilon^\gamma}\mathcal{A}_\varepsilon + I \right)^{-1}\bigg|_{L^2_{\rm bend}} 
-  (\mathcal{M}^{\rm bend}_\varepsilon)^*\left(\frac{1}{\varepsilon^{\gamma-2}} \mathcal{A}^{\rm bend} + I \right)^{-1}\mathcal{M}^{\rm bend}_\varepsilon \Xi_\varepsilon
		 -\mathcal{A}^{\rm \rm corr}_{\rm bend}(\varepsilon)     -\mathcal{\widetilde{A}}^{\rm \rm corr}_{\rm bend}(\varepsilon)\right\} S_{\!\varepsilon^\delta}
 \Bigg\|_{L^2 \to L^2} \\[0.4em] 
&\hspace{60pt}\leq
{C\max\Bigl\{\varepsilon^{\tfrac{\gamma + 2}{4}-\delta}, 1\Bigr\}}\left\{ \begin{array}{ll}
         \varepsilon^{\tfrac{\gamma + 2}{2}},
         & i = 1,2,\\[0.6em]
         \varepsilon^{\tfrac{3(\gamma + 2)}{4}},
         & i = 3.\end{array} \right. 
     \end{aligned}
 \label{bend_higher} 
\end{align}

\end{theorem}

\subsection{Methodology and the relationship with existing literature}
\label{sectionobjasnjenje}
 The approach to norm-resolvent estimates we use in the present paper resembles those adopted in \cite{BirmanSuslina,BirmanSuslina_corrector,ChCoARMA} and is different from those of \cite{Kenig,ZhikovPastukhova}, which do not rely on spectral analysis.
The spectral approach seems to be most convenient for obtaining adequate results for elastic thin structure (specially for the estimates depending on the parameters $\gamma,\delta$). On the other hand, neither the spectral approach of \cite{ABV} nor the two-scale expansion discussed in the same paper seems to be convenient for  obtaining the results of \cite{cherednichenkovelcic} and those of the present paper. First, the error estimates derived in \cite{ABV} are not uniform with respect to the data, and second, when using the spectral approach, the authors of \cite{ABV} assume that a power-series expansion in the quasimomentum is valid for the smallest eigenvalue, which does not apply, in general, to systems. 
We extend the approach introduced in \cite{cherednichenkovelcic} to the case of thin elastic rods, which has special features from the geometric and spectral points of view.

The starting point of our approach consists in applying the scaled Gelfand transform (see Section \ref{Gelfand_sec}) to the resolvent of the operator ${\mathcal A}_\varepsilon$ defined by the bilinear form \eqref{bilinear_form}, which results in the direct integral decomposition \eqref{gelfanddecomposition}.
 The next step is to provide an approximation, in the operator-norm topology and uniformly in $\chi \in [-\pi,\pi)$, for resolvent of the ``fibre operator" ${\mathcal A}_\chi,$  see \eqref{nak50} below,  which we scale by the order (with respect to $|\chi|$) of one of its lowest eigenvalues.
Note that even though the results we aim for (see Section \ref{main_results_sec}) are estimates with respect to the parameter $\varepsilon$ (playing the role of the thickness of the rod as well as the period of material oscillations), the estimates for the $|\chi|$-scaled resolvent of ${\mathcal A}_\chi$ are expressed in terms of the quasimomentum $\chi.$ These are then translated into the desired norm-resolvent estimates for ${\mathcal A}_\varepsilon$ by the means of the Cauchy integral formula followed by the inverse Gelfand transform. The mentioned estimates for ${\mathcal A}_\chi$ with respect to the quasimomentum $\chi$ are obtained by performing an asymptotic procedure adapted for the fact that the rod is thin and its properties oscillate in one direction only. 

The above approach is first implemented under Assumption \ref{matsym}, so the problem for the operator ${\mathcal A}_\varepsilon$ (and consequently for ${\mathcal A}_\chi$)   splits  into two simpler problems in the invariant subspaces. The asymptotic procedure can be carried out separately for these two problems and is simpler than a similar procedure for the full problem.

The structure of the paper is as follows:
\begin{itemize}
    \item In Section \ref{section3} we derive the Korn-type inequalities  for quasiperiodic functions  in  the general case as well as for the bending and stretching deformations in the case of material symmetries. Using these inequalities, we estimate the spectrum by the means of Rayleigh quotients. The outcome of the analysis is information about the orders of magnitudes of eigenvalues, hence about convenient scalings of the operators for the next step.   The difference with respect to \cite{BirmanSuslina}, \cite{BirmanSuslina_corrector} comes from the fact that the lowest eigenvalues appear with different orders in $|\chi|$ (two of them of order ${\chi^4}$ and two of them of order ${\chi^2}$). Due to this, the spectral germ technique  used in these works  is not directly applicable, since the usual assumption on regularity of the spectral germ implies that all lowest eigenvalues have the same order ${\chi^2}$. Notably, in the context of elastic plates \cite{cherednichenkovelcic} one is presented with a situation where the set of lowest eigenvalues consists of one eigenvalue of order ${\chi^4}$ and two eigenvalues of order $\RRR \chi^2 $ (naturally giving rise to two different eigenspaces subject to the so-called Kirchhoff-Love ansatz). 
     
    \item In Section \ref{section4} we approximate the ($\chi$-dependent) resolvent operators by performing an iterative asymptotic procedure for defining successive (i.e. with correctors in increasing powers of $|\chi|$) approximations for the solutions of the related problems, for which the error bounds depend explicitly on  the quasimomentum $\chi$ and the norm of the (Gelfand-transformed) loads. Here the analysis is carried out under Assumption \ref{matsym}, separately for the two invariant subspaces. The case of stretching deformations somewhat resembles the ``bulk" homogenisation setup of \cite{BirmanSuslina}, primarily with regard to the order of the operator scaling. Similarly to \cite{cherednichenkovelcic}, we use $\chi$-dependent asympotics, which is a natural choice as a consequence of apriori bounds. \RRR  For this reason, the approach of \cite{cherednichenkovelcic} and the present paper is different to the technique developed in \cite{ChCoARMA}, where $\varepsilon$-dependent asymptotics was considered  and where the first-order approximation was obtained. As already mentioned, scaling of the problem with respect to $\chi$ in the natural way (after apriori analysis), enables us to go further in the asymptotic expansion in a systematic way as well as to express the error with respect to the scaling parameters $\gamma,\delta$.  
    
    \item In Section \ref{section5} we combine the estimates in Section \ref{section4} with the Cauchy integral formula, which yields resolvent estimates for ${\mathcal A}_\chi$ in terms of the ``physical" parameter $\varepsilon$. 
   \RRR We then use these estimates and translate them back onto the original physical domain.  
    
    \item In Section \ref{section6} we repeat the procedure developed in the preceding sections, but this time under no additional assumptions on cross-section geometry or material symmetries. This requires us to perform two simultaneous asymptotic procedures, with different scalings and then combine them together. The final two steps consist, as before, in utilising the Cauchy integral formula and translating the obtained estimates back onto the original physical domain.  
\end{itemize}

We plan to use the results of this paper to obtain quantitative time evolution analysis for thin (heterogeneous) elastic structures, including the higher-order perturbation as in \cite{ABV}.

\subsection{\CCC Gelfand transform }
\label{Gelfand_sec}
The first key step of the analysis consists in decomposing, by means of the so-called Gelfand transform, the original problem on ${\mathbb R}^3$ into a family of problems (parametrised by $\chi \in [-\pi,\pi)$) on bounded domains so the associated operators have compact resolvents, cf.~\eqref{gelfanddecomposition} below. 
In this section we provide the definition of the Gelfand transform as well as its basic properties. We also explain how Gelfand transform changes the bilinear form introduced  in \eqref{bilinear_form}.

\RRR For  every value of the parameter $\chi \in [-\pi,\pi)$, we consider the following Sobolev space of $\chi$-quasiperiodic functions:  
\begin{equation}
H_{\chi}^1\bigl(Y;H^1(\omega;\C^3)\bigr) := \left\{{\rm e}^{{\rm i}\chi y}\vect u(\widehat{x},y),\quad  \vect u \in H_{\#}^1\bigl(Y;H^1(\omega;\C^3)\bigr) \right\},\quad \chi \in [-\pi,\pi).
\end{equation}
 Similarly, we define the space  $L_{\chi}^2( Y;L^2(\omega;\C^3))$. 

\subsubsection{Formulation in terms of quasimomenta $\chi\in[-\pi,\pi)$}
\label{unscaled_sec}

For every $\varepsilon>0,$ we define an operator $\mathfrak{G}_\varepsilon$ on $L^2(\omega\times \R;\R^3)$ by the formula ($\vect u\in L^2(\omega\times \R;\R^3)$)
\begin{equation}
(\mathfrak{G}_\varepsilon \vect u)(\widehat{x},y,\chi):= \sqrt{\frac{\varepsilon}{2\pi}} \sum_{n\in \Z}{\rm e}^{-{\rm i}\chi(y+n)}\vect u\bigl(\widehat{x},\varepsilon(y+n)\bigr), \quad (\widehat{x},y) \in \omega\times \R, \quad \chi\in [-\pi,\pi).
\end{equation}
We refer to this operator as the scaled Gelfand transform. Note that it
transforms 
into $Y$-periodic functions in variable $y$, namely:
\begin{equation}
(\mathfrak{G}_\varepsilon \vect u)(\widehat{x},y+1,\chi) = (\mathfrak{G}_\varepsilon \vect u)(\widehat{x},y,\chi) \qquad {\rm a.e.}\ \ (\widehat{x},y) \in \omega\times \R, \quad \forall\chi\in [-\pi,\pi).
\end{equation}
The operator 
\begin{equation}
\mathfrak{G}_\varepsilon : L^2(\omega\times \mathbb{R};\mathbb{R}^3) \to L^2\bigl([-\pi,\pi);L_{\#}^2
\bigl(Y; L^2(\omega;\C^3)\bigr)\bigr) = \int_{ [-\pi,\pi)}^\oplus L_{\#}^2\bigl(Y,\chi;L^2(\omega;\C^3)\bigr)d\chi
\end{equation}
is unitary, in the sense that
\begin{equation}
\left\langle \vect u, \vect v \right\rangle_{L^2(\omega\times \R;\R^3)} = \int_{-\pi}^{\pi}\left\langle \mathfrak{G}_\varepsilon \vect u,\mathfrak{G}_\varepsilon \vect v \right\rangle_{L^2_{\#}(\omega \times Y; \C^3)}d\chi \qquad \forall \vect u, \vect v \in L^2(\omega\times \R;\R^3).
\end{equation}
If the Gelfand transform $\mathfrak{G}_\varepsilon \vect u$ is known, the function \RRR $\vect u$  can be reconstructed by the formula
\begin{equation} \label{gotovo1}
\vect  u(\widehat{x},x_3) = \frac{1}{\sqrt{2\pi\varepsilon}} \int_{-\pi}^\pi {\rm e}^{{\rm i}\chi x_3/\varepsilon} (\mathfrak{G}_\varepsilon \vect u)(\widehat{x},x_3/\varepsilon,\chi)d\chi\qquad {\rm a.e.}\ \  (\widehat{x}, x_3)\in\omega\times{\mathbb R}.
\end{equation}
The above construction can  be  interpreted in the sense that
\begin{equation}
L^2(\omega\times \R;\R^3)\sim\int_{ [-\pi,\pi)}^\oplus L^2_{\#}\bigl(Y,\chi;L^2(\omega; \C^3)\bigr)d\chi,
\end{equation}
where $\sim$ stands for equality up to a unitary transform. 

Furthermore, by noting that the following formulae hold for the compositions of the scaled Gelfand transform with operators of differentiation:
\begin{equation}
\mathfrak{G}_\varepsilon (\partial_{x_\alpha}) \vect u = \partial_{x_\alpha}(\mathfrak{G}_\varepsilon \vect u), \quad \mathfrak{G}_\varepsilon (\partial_{x_3} \vect u ) =\varepsilon^{-1} \bigl(\partial_{y}(\mathfrak{G}_\varepsilon \vect u) + {\rm i}\chi \mathfrak{G}_\varepsilon \vect u\bigr),
\label{gelfand_y}
\end{equation}
one obtains
\begin{equation}
{\mathfrak a}_\varepsilon(\vect u,\vect v)=\varepsilon^{-2}{\mathfrak a}_\chi(\mathfrak{G}_\varepsilon\vect u,\mathfrak{G}_\varepsilon\vect v) \quad \forall \chi \in [-\pi,\pi), 
\end{equation}
where
\begin{equation}
{\mathfrak a}_\chi (\vect u,\vect v) := 	\int_{\omega\times Y} {\mathbb A}(y)(\simgrad+{\rm i}X_\chi)\vect u(\widehat{x},y): \overline{(\simgrad  + {\rm i}X_\chi )\vect v(\widehat{x},y)}\,d\widehat{x}dy, \qquad \vect u,\vect v\in H^1_{\#}\bigl(Y;H^1(\omega;\C^3)\bigr).
\label{a_chi_form}
\end{equation}
The operator  of multiplication  $X_{\chi}$ in (\ref{a_chi_form}), acting on the space $L^2(\omega\times Y; \C^3)$, is defined by
\begin{equation*}
X_{\chi}\vect u  = \begin{bmatrix}
0 & 0 & \dfrac{1}{2}\chi u_1 \\[0.7em]
0 & 0 & \dfrac{1}{2}\chi u_2 \\[0.7em]
\dfrac{1}{2}\chi u_1 & \dfrac{1}{2}\chi u_2&  \chi{u_3}
\end{bmatrix}.
\end{equation*}
For a fixed $\chi \in [-\pi,\pi),$ we define the operator
\begin{equation} \label{nak50}
\mathcal{A}_\chi: \mathcal{D}(\mathcal{A}_\chi) \subset H^1_{\#}\bigl(Y;H^1(\omega;\C^3)\bigr) \to L^2(\omega\times Y,\C^3),
\end{equation}
associated with the differential expression 
\[
(\simgrad+{\rm i}X_\chi)^* {\mathbb A}(y) (\simgrad + {\rm i}X_\chi)
\]
as the one defined by the form \eqref{a_chi_form}.

The result of applying the scaled Gelfand transform to the resolvent of the operator ${\mathcal A}_\varepsilon$ can be represented by the equality  
\begin{equation}
\label{gelfanddecomposition}
\mathfrak{G}_\varepsilon\left(\mathcal{A}_\varepsilon + I \right)^{-1}\mathfrak{G}_\varepsilon^{-1} = \int_{ [-\pi,\pi)}^\oplus \left(\frac{1}{\varepsilon^2}\mathcal{A}_\chi + I \right)^{-1} d\chi.
\end{equation}
In other words, by applying Gelfand transform, we have decomposed the resolvent $\left(\mathcal{A}_\varepsilon + I \right)^{-1}$ into a continuum of resolvents $\left(\varepsilon^{-2}\mathcal{A}_\chi + I \right)^{-1}$ indexed by $\chi\in [-\pi,\pi)$. As we will see, unlike the resolvent of ${\mathcal A}_\varepsilon$, the  operators  $\left(\varepsilon^{-2}\mathcal{A}_\chi + I \right)^{-1}$ are compact and hence for each $\chi$ the operator ${\mathcal A}_\chi$ has purely discrete spectrum. 

Closely related to the scaled Gelfand transform is the scaled Floquet transform defined by 
\begin{equation}
(\mathfrak{F}_\varepsilon\vect u)(\widehat{x},y,\chi):= \sqrt{\frac{\varepsilon}{2\pi}} \sum_{n\in \Z}{\rm e}^{-{\rm i}\chi n}\vect u(\widehat{x},\varepsilon(y+n)), \qquad (\widehat{x},y) \in \omega\times Y, \quad  \chi\in [-\pi,\pi).
\end{equation}
For every $\chi \in [-\pi,\pi)$,  the function $(\mathfrak{F}_\varepsilon\vect u)(\widehat{x},y,\chi),$ $(\widehat{x},y)\in\omega\times Y$ belongs to the space $H_{\chi}^1(Y;H^1(\omega;\C^3))$ of quasiperiodic functions. It is a straightforward observation that the Floquet and Gelfand transforms are linked by the  following simple  identity:
\begin{equation}
(\mathfrak{F}_\varepsilon \vect u)(\widehat{x}, y) = {\rm e}^{{\rm i} \chi y}(\mathfrak{G}_\varepsilon \vect u)(\widehat{x},y),\quad (\widehat{x},y)\in\omega\times Y \qquad  \forall \vect u \in L^2(\omega \times \R;\R^3).
\end{equation}
The scaled Floquet transform is an isometry:
\begin{equation}
\left\langle \vect u,\vect v \right\rangle_{L^2(\omega\times \R;\R^3)} = \int_{-\pi}^\pi\left\langle \mathfrak{F}_\varepsilon\vect u,\mathfrak{F}_\varepsilon\vect v \right\rangle_{L^2(\omega \times Y; \C^3)}d\chi,\qquad \vect u, \vect v\in L^2(\omega \times \R;\R^3)
\end{equation}
Similarly, we have
\begin{equation}
{\mathfrak a}_\varepsilon(\vect u, \vect v) =\varepsilon^{-2}{\mathfrak a}(\mathfrak{F}_\varepsilon\vect u,\mathfrak{F}_\varepsilon\vect v),
\end{equation}
where the sesquilinear form ${\mathfrak a}$ is defined by
\begin{equation}
{\mathfrak a}(\vect u,\vect v) := 	\int_{\omega\times Y} {\mathbb A}(y)\simgrad \vect u(\widehat{x},y):\overline{\simgrad \vect v(\widehat{x},y)}\,d\widehat{x}dy, \qquad \vect u,\vect v\in H^1_{\chi}\bigl(Y;H^1(\omega;\C^3)\bigr).
\end{equation}
This is due to  the formula  
\begin{equation}
\mathfrak{F}_\varepsilon (\partial_{x_\alpha}) \vect u = \partial_{x_\alpha}(\mathfrak{F}_\varepsilon \vect u), \quad \mathfrak{F}_\varepsilon (\partial_{x_3} \vect u ) =\varepsilon^{-1} \partial_{y}(\mathfrak{F}_\varepsilon \vect u).
\end{equation}
 Note that for  quasiperiodic functions $\vect w \in H^1_{\chi}(Y;H^1(\omega;\C^3))$, namely 
\[
\vect w(\widehat{x}, y)= {\rm e}^{{\rm i}\chi y}\vect u(\widehat{x},y),\quad (\widehat{x},y)\in\omega\times Y,
\]
where $\vect u \in H_{\#}^1(Y; H^1(\omega;\C^3)),$ one has
\begin{equation}
\simgrad \vect w(\widehat{x},y) = {\rm e}^{{\rm i}\chi y} \bigl(\simgrad \vect u(\widehat{x},y) + {\rm i} X_\chi \vect u(\widehat{x},y)\bigr),\qquad (\widehat{x},y)\in\omega\times Y.
\end{equation}
 
\subsubsection{Formulation in terms of scaled quasimomenta}
\label{scaled_sec}

Here we note an alternative definition of Gelfand transform in terms of the ``scaled quasimomentum" $\theta:= \chi/\varepsilon$, for which we retain the same notation $\mathfrak{G}_\varepsilon:$
\begin{equation}
(\mathfrak{G}_\varepsilon \vect u)(\widehat{x},y,\theta):= \sqrt{\frac{\varepsilon}{2\pi}} \sum_{n\in \Z}{\rm e}^{-{\rm i}\theta \varepsilon(y+n)}\vect u\bigl(\widehat{x},\varepsilon(y+n)\bigr), \qquad (\widehat{x},y) \in \omega\times \R, \quad \theta \in [-\pi/\varepsilon, \pi/\varepsilon),
\end{equation}
where the inverse is given  by  (cf. \eqref{gotovo1}):
\begin{equation}
\left(\mathfrak{G}_\varepsilon^{-1}\vect  U\right)(\widehat{x},x_3) = \sqrt{\frac{\varepsilon}{2\pi}} \int_{-\pi/\varepsilon}^{\pi/\varepsilon} {\rm e}^{{\rm i}\theta x_3} \vect U(\widehat{x},x_3/\varepsilon,\theta)d\theta, \qquad (\widehat{x}, x_3)\in\omega\times{\mathbb R}.
\end{equation}
It is straightforward to link Gelfand transform with Fourier transform:
\begin{align*}
\int_Y (\mathfrak{G}_\varepsilon \vect u)(\widehat{x},y,\theta) dy = & \int_Y \sqrt{\frac{\varepsilon}{2\pi}} \sum_{n\in \Z}{\rm e}^{-{\rm i}\theta \varepsilon(y+n)}\vect u(\widehat{x},\varepsilon(y+n)\bigr)dy \\[0.4em]
=& \sqrt{\frac{\varepsilon}{2\pi}} \sum_{n\in \Z}\int_Y{\rm e}^{-{\rm i}\theta \varepsilon(y+n)}\vect u\bigl(\widehat{x},\varepsilon(y+n)\bigr)dy
=\frac{1}{\sqrt{2\pi \varepsilon}} \sum_{n\in \Z}\int_{\varepsilon(Y+ n)} {\rm e}^{-{\rm i}\theta y}\vect u(\widehat{x}, y)dy \\[0.4em]
= & \frac{1}{\sqrt{2\pi \varepsilon}} \int_{\R}{\rm e}^{-{\rm i}\theta y}\vect u(\widehat{x}, y)dy = \frac{1}{\sqrt{2 \pi \varepsilon}} \RRR \mathcal{F}[\vect u] (\widehat{x},\theta/2 \pi),\qquad (\widehat{x}, y)\in\omega\times Y,\quad \theta\in[-\pi/\varepsilon, \pi/\varepsilon).
\end{align*}

\begin{remark}
	The versions of the Gelfand transform introduced in Sections \ref{unscaled_sec}, \ref{scaled_sec} as well as their key properties are naturally extended to functions with values in ${\mathbb R}$ and ${\mathbb R}^2,$ with the functions $\vect u,$ ${\mathfrak G}_\varepsilon{\vect u}$ replaced by their one- and two-dimensional analogues.
\end{remark}


\subsubsection{Relation with smoothing operator}
\label{smoothing_op_sec}
The smoothing operator  $\Xi_\varepsilon$ defined in \eqref{smootheningoperator} satisfies
\begin{equation}
\label{smootheningoperator1}
\Xi_\varepsilon f= \mathfrak{G}_\varepsilon^{-1} \int_{Y}  (\mathfrak{G}_\varepsilon f)(y) dy,\qquad f\in L^2(\omega \times\R).
\end{equation}
Indeed, one has
\begin{equation}
\begin{split}
 \biggl(\mathfrak{G}_\varepsilon^{-1} \int_{Y}  (\mathfrak{G}_\varepsilon f)(\widehat{x},y,\theta) dy\biggr) & = \left(\mathfrak{G}_\varepsilon^{-1} \left\{\frac{1}{\sqrt{2 \pi \varepsilon}}{\mathcal F}[f](\hat{x},\theta/2 \pi)\right\}\right) \\
&= \frac{1}{2\pi}\int_{-\frac{\pi}{\varepsilon}}^{\frac{\pi}{\varepsilon}} {\rm e}^{i \theta x_3}{\mathcal F}[f](\widehat{x},\theta/2 \pi) d \theta = \int_{-\frac{1}{2\varepsilon}}^{\frac{1}{2\varepsilon}} {\rm e}^{2\pi{\rm i} \theta x_3}{\mathcal F}[f](\hat{x},\theta) d \theta
\\ &=\Bigl({\mathcal F}^{-1}\bigl[\mathbbm{1}_{\left[
	\frac{-1}{2\varepsilon},\frac{1}{2\varepsilon}\right]}\bigr] * f\Bigr)(x),\qquad x\in\omega\times{\mathbb R}.
\end{split} 
\end{equation}

\section{Auxiliary results and apriori estimates}
\label{section3}

In order to  render the  proofs  within  this section more elegant, we have decided to present the results with respect to quasiperiodic functions, which are natural for the  Floquet transform. This is obviously equivalent to the approach with periodic functions via the Grelafnd transform, in view of the correspondences
\begin{equation}
	H^1_{\chi}\bigl(Y;H^1(\omega;\C^3)\bigr) \longleftrightarrow H^1_{\#}\bigl(Y;H^1(\omega;\C^3)\bigr), \quad \mathfrak{F}_\varepsilon \longleftrightarrow \mathfrak{G}_\varepsilon, \quad \simgrad \longleftrightarrow \simgrad + {\rm i}X_\chi. 
\end{equation}

 The basic tool for apriori estimates in linear elasticity is Korn's inequality, see e.g. \cite{Hor} and Section \ref{appKorn}. Before formulating an appropriate quasiperiodic version of Korn's inequality (which is done in Section \ref{sec31}), we prove an auxiliary Korn-type estimate.

The following lemma, which is based on Proposition \ref{quasiperiodickorn}, provides estimates for the approximating infinitesimal rigid motions of quasiperiodic functions  that appear in Korn's inequality.
\begin{lemma} \label{lemstartKorn} 
There is a constant $C>0$ such that for all $\chi \in  [-\pi,\pi)\setminus\left\{0\right\}$ and $\vect u \in H^1_{\chi}(Y;H^1(\omega;\C^3))$ one has  
\begin{equation}
    \Vert \vect u - \widetilde{\vect w}\Vert _{H^1(\omega\times Y;\C^3)}\leq C \bigl\Vert  \simgrad \vect u \bigr\Vert _{L^2(\omega\times Y;\C^{3\times 3})},
\end{equation}
where  $\widetilde{\vect w}=Ax+{\vect c}$ is the infinitesimal rigid-body motion provided by Proposition \ref{quasiperiodickorn},
with the  entries  of $A, \vect c$ satisfying the estimates
\begin{eqnarray*}
\max\bigl\{|a|, |b|, |d|, |c_3| \bigr\} \leq|\chi|^{-1}C \bigl\Vert  \simgrad \vect u \bigr\Vert _{L^2(\omega\times Y;\C^{3\times 3})} \\[0.3em]
\bigl|({\rm e}^{{\rm i}\chi} -1 )c_2 - b\bigr| \leq C \bigl\Vert  \simgrad \vect u \bigr\Vert _{L^2(\omega\times Y;\C^{3\times 3})} \\[0.3em]
\bigl|({\rm e}^{{\rm i}\chi} -1 )c_1 - a\bigr| \leq C \bigl\Vert  \simgrad \vect u \bigr\Vert _{L^2(\omega\times Y;\C^{3\times 3})} \\[0.3em]
\max\bigl\{|c_1|, |c_2|\bigr\} \leq \chi^{-2}  C \bigl\Vert  \simgrad \vect u \bigr\Vert _{L^2(\omega\times Y;\C^{3\times 3})}.
\end{eqnarray*}

\end{lemma}
\begin{proof}
Using the trace theorem for functions in $H^1_{\chi}(Y;H^1(\omega;\C^3))$ as well as the Korn's inequality \eqref{second_Korn}, we infer that
\begin{equation}
\begin{split}
        \Vert \vect u- \widetilde{\vect w}\Vert _{L^2(\omega \times \left\{ y = 1\right\})}\leq C \bigl\Vert  \simgrad \vect u \bigr\Vert _{L^2(\omega\times Y;\C^{3\times 3})}, \\[0.3em] \Vert \vect u- \widetilde{\vect w}\Vert _{L^2(\omega \times \left\{ y = 0\right\})} \leq C \bigl\Vert  \simgrad \vect u \bigr\Vert _{L^2(\omega\times Y;\C^{3\times 3})}.
\end{split}
\end{equation}
Furthermore, for smooth quasiperiodic $\vect u$ one has
\begin{equation}
    \vect u(\widehat{x},1)={\rm e}^{{\rm i}\chi} \vect u(\widehat{x},0)\qquad \forall \widehat{x}\in\omega,
\end{equation}
and hence 
\begin{equation}
    \bigl|\widetilde{\vect w}(\widehat{x},1) - {\rm e}^{{\rm i}\chi} \widetilde{\vect w}(\widehat{x},0)\bigr|  \leq   \bigl| \vect u(\widehat{x}, 1) - \widetilde{\vect w}(\widehat{x},1)\bigr| + \bigl|{\rm e}^{{\rm i}\chi}\bigl(\vect u(\widehat{x}, 0) -\widetilde{ \vect w}(\widehat{x},0)\bigr)\bigr|\quad \forall \widehat{x}:=(x_1,x_2)\in \omega. 
\end{equation}
Extending this estimate to all $\vect u\in H^1_{\chi}(Y;H^1(\omega;\C^3))$ yields
\begin{equation}
    \bigl\Vert  \widetilde{\vect w}(\widehat{x},1) - {\rm e}^{{\rm i}\chi}\widetilde{\vect w}(\widehat{x},0)\bigr\Vert _{L^2(\omega;\C^3)} \leq \Vert \vect u-\widetilde{\vect w}\Vert _{L^2(\omega \times \left\{y=1\right\};\C^3)} + \Vert \vect u- \widetilde{\vect w}\Vert _{L^2(\omega \times \left\{y=0\right\};\C^3)}
\end{equation}
and therefore
$$
\bigl\Vert \widetilde{\vect w} (\widehat{x},1) - {\rm e}^{{\rm i}\chi} \widetilde{\vect w}(\widehat{x},0)\bigr\Vert _{L^2(\omega;\C^3)} \leq C \bigl\Vert  \simgrad \vect u \bigr\Vert _{L^2(\omega \times Y;\C^3)}.
$$
Component-wise, the last estimate reads 
\begin{align*}
\int_{\omega}\bigl| ({\rm e}^{{\rm i}\chi} - 1)(c_1 + d x_2) - a\bigr|^2 d\widehat{x}& \leq C \bigl\Vert \simgrad \vect u \bigr\Vert_{L^2(\omega\times Y;\C^{3\times 3})}^2, \\[0.3em]
\int_{\omega} \bigl| ({\rm e}^{{\rm i}\chi} - 1)(c_2 - d x_1) - b\bigr|^2 d\widehat{x} & \leq C \bigl\Vert \simgrad \vect u \bigr\Vert_{L^2(\omega\times Y;\C^{3\times 3})}^2, \\[0.3em]
\int_{\omega} \bigl| ({\rm e}^{{\rm i}\chi} - 1)(c_3 - a x_1 - b x_2)\bigr|^2 d\widehat{x} & \leq C \bigl\Vert \simgrad \vect u\bigr\Vert_{L^2(\omega\times Y;\C^{3\times 3})}^2. 
\end{align*}
\CCC By using  the Taylor expansion, we note that there exist constants $C_1, C_2>0$ such that for all $\chi \in [- \pi, \pi)$ one has
$$C_1 |\chi | \leq |{\rm e}^{{\rm i}\chi} - 1| \leq C_2 |\chi|.  $$
These observations allow us to infer that \CCC (recall \eqref{coordinatesymmetries}) 
\begin{align*}
\bigl|({\rm e}^{{\rm i}\chi} -1)c_1 - a\bigr|^2 + C {\chi^2} |d|^2&\leq C \int_{\omega}\bigl| ({\rm e}^{{\rm i}\chi} - 1)(c_1 + d x_2) - a\bigr|^2 d\widehat{x},\\[0.3em]
\bigl|({\rm e}^{{\rm i}\chi} -1)c_2 - b\bigr|^2 + C {\chi^2} |d|^2&\leq C \int_{\omega} \bigl| ({\rm e}^{{\rm i}\chi} - 1)(c_2 - d x_1) - b\bigr|^2 d\widehat{x},   \\[0.3em]
{\chi^2}\bigl(|c_3|^2 + |a|^2 + |b|^2\bigr)&\leq C \int_{\omega}\bigl|({\rm e}^{{\rm i}\chi} - 1)(c_3 - a x_1 - b x_2)\bigr|^2 d\widehat{x}, 
\end{align*}
which yields the required estimates
\end{proof}

\subsection{Leading-order term of expansion in $\chi$}
\label{sec31}

The next result is important for the analysis of the asymptotic behaviour of the rod deformations.  It is a version of Korn's inequality for quasiperiodic functions on $\omega \times Y$ and relies on Lemma \ref{lemstartKorn}. 
\begin{proposition}
\label{leadingordertermproposition}
There is a constant $C>0$ such that for every $\chi\in [-\pi,\pi)\setminus\left\{0\right\}$, $\vect u \in H^1_{\chi}(Y;H^1(\omega;\C^3))$ there exist $c_1,c_2,c_3,d \in \C$ such that  for the function 
\begin{equation}
    \label{leadingorderterm}
   \vect w(\widehat{x}, y) = {\rm e}^{{\rm i}\chi y} \left(\begin{bmatrix}dx_2 \\[0.3em] -dx_1 \\[0.3em] c_3 \end{bmatrix} + \begin{bmatrix} c_1 \\[0.3em]
c_2 \\[0.3em] -{\rm i}\chi(c_1 x_1 + c_2 x_2 )\end{bmatrix} \right), \qquad (\widehat{x},y)\in\omega\times Y 
\end{equation}
one has 
\begin{align}
        &\Vert \vect u- \vect w\Vert _{H^1(\omega\times Y;\C^3)} \leq C \bigl\Vert  \simgrad \vect u \bigr\Vert _{L^2(\omega\times Y;\C^{3\times 3})}, \label{uw_est} 
        \\[0.5em]
        &\max\bigl\{|d|,|c_3|\bigr\} \leq |\chi|^{-1}C\bigl\Vert  \simgrad \vect u \bigr\Vert _{L^2(\omega\times Y;\C^{3\times 3})},\noindent 
        \\[0.5em]
          &\max\bigl\{|c_1|,|c_2|\bigr\} \leq \chi^{-2}C\bigl\Vert  \simgrad \vect u \bigr\Vert _{L^2(\omega\times Y;\C^{3\times 3})}.\noindent
\end{align}
\end{proposition}
\begin{proof}
 We take $\widetilde{\vect w}$ and $a\in \R$, $b\in \R$, $d \in \R$ and $c \in \R^3$ from Lemma \ref{lemstartKorn}. 
It follows that
\begin{align}
\bigl\Vert \vect u(\widehat{x},y) -  \widetilde{\vect w}(\widehat{x},y)\bigr\Vert_{H^1(\omega\times Y;\C^3)} & = 
\left\Vert \vect u(\widehat{x},y) - \left(\begin{bmatrix}
0 & d & a \\[0.3em]
-d & 0 & b \\[0.3em]
-a & -b & 0 \\[0.3em]
\end{bmatrix} 
\begin{bmatrix}
x_1 \\[0.3em]
x_2 \\[0.3em]
y
\end{bmatrix}
+
\begin{bmatrix}
c_1 \\[0.3em]
c_2 \\[0.3em]
c_3
\end{bmatrix}
\right)\right\Vert_{H^1(\omega\times Y;\C^3)}\label{ab_coef} \\[0.5em]
&= 
\left\Vert \vect u -\begin{bmatrix}
 d x_2  \\[0.35em]
 - d x_1 \\[0.35em]
 c_3
\end{bmatrix}-\begin{bmatrix}
ay+c_1 \\[0.35em]
by+c_2 \\[0.35em]
 -a x_1 -b x_2
\end{bmatrix} \right\Vert_{H^1(\omega\times Y;\C^3)} \leq C \bigl\Vert  \simgrad \vect u \bigr\Vert _{L^2(\omega\times Y;\C^{3\times 3})}.\nonumber
\end{align} 
  
Using the obvious estimates 
\begin{equation} \label{nak100} 
\bigl|{\rm e}^{{\rm i}\chi}- (1 + {\rm i}\chi)\bigr| = {O}(\chi^2), \quad |{\rm e}^{{\rm i}\chi}- 1| = {O}(\RRR |\chi| ),
\end{equation} 
    we  infer that 
\begin{equation}
\label{ab}
\begin{aligned}
    |{\rm i}\chi c_1 - a| &\leq  \left| ({\rm e}^{{\rm i}\chi}- 1)c_1 - a\right| + \left|\bigl({\rm e}^{{\rm i}\chi}- (1 + {\rm i}\chi)\bigr)c_1 \right| \\[0.4em]
				&\leq  C_1 \bigl\Vert  \simgrad \vect u \bigr\Vert _{L^2(\omega\times Y;\C^{3\times 3})} + c_2 {\chi^2}\chi^{-2}\bigl\Vert  \simgrad \vect u \bigr\Vert _{L^2(\omega \times Y;\C^{3\times 3})}\\[0.4em]
			    &\leq C\bigl\Vert  \simgrad \vect u \bigr\Vert _{L^2(\omega \times Y;\C^{3\times 3})}, \\[0.4em]
|{\rm i}\chi c_2 - b| &\leq  C\bigl\Vert  \simgrad \vect u \bigr\Vert _{L^2(\omega \times Y;\C^{3\times 3})}
\end{aligned}
\end{equation}

The estimates \eqref{ab} allow us to eliminate the coefficients $a$ and $b$ in \eqref{ab_coef}, by estimating them in terms of $c_1$, $c_2$ in the $H^1(\omega\times Y;\C^3)$ norm. 
The following calculation provides the details of such estimates for the coefficient $a:$
\begin{align*}
\left\Vert ay + c_1 - {\rm e}^{{\rm i}\chi y} c_1\right\Vert_{L^2(\omega\times Y;\C^3)} &= \left\Vert ay  - \left({\rm e}^{{\rm i}\chi y} -1\right) c_1 \right\Vert_{L^2(\omega\times Y;\C^3)} \\[0.3em]
&\leq \left\Vert ay - {\rm i}\chi y c_1 \right\Vert_{L^2(\omega\times Y;\C^3)} + \left\Vert {\rm i}\chi y c_1 - \left({\rm e}^{{\rm i}\chi y}-1\right) c_1 \right\Vert_{L^2(\omega\times Y;\C^3)} \\[0.3em]
&= \bigl\Vert y\left(a - {\rm i}\chi  c_1\right) \bigr\Vert_{L^2(\omega\times Y;\C^3)} + \left\Vert \left({\rm i}\chi y - \left({\rm e}^{{\rm i}\chi y}-1\right)\right) c_1 \right\Vert_{L^2(\omega\times Y;\C^3)} \\[0.3em]
&\leq C_1 \bigl\Vert  \simgrad \vect u \bigr\Vert _{L^2(\omega\times Y;\C^{3\times 3})} + c_2 {\chi^2}\chi^{-2} \bigl\Vert  \simgrad \vect u \bigr\Vert _{L^2(\omega\times Y;\C^{3\times 3})} \\[0.3em]
 &\leq C \bigl\Vert  \simgrad \vect u \bigr\Vert _{L^2(\omega\times Y;\C^{3\times 3})},
\end{align*}
\begin{align*}
\left\Vert \partial_y\left(ay+ c_1 - {\rm e}^{{\rm i}\chi y} c_1 \right) \right\Vert_{L^2(\omega\times Y;\C^3)} &= \left\Vert a - {\rm i}\chi {\rm e}^{{\rm i}\chi y} c_1 \right\Vert_{L^2(\omega\times Y;\C^3)} \\[0.3em]
 &\leq \left\Vert a - {\rm i}\chi c_1 \right\Vert_{L^2(\omega\times Y;\C^3)} + \left\Vert {\rm i}\chi\left( 1 - {\rm e}^{{\rm i}\chi y} c_1\right) \right\Vert_{L^2(\omega\times Y;\C^3)} \\ 
&\leq C_1 \bigl\Vert  \simgrad \vect u \bigr\Vert _{L^2(\omega\times Y;\C^{3\times 3})} + c_2 {\chi^2}\chi^{-2}\bigl\Vert  \simgrad \vect u \bigr\Vert _{L^2(\omega\times Y;\C^{3\times 3})} \\ 
&\leq C \bigl\Vert  \simgrad \vect u \bigr\Vert _{L^2(\omega\times Y;\C^{3\times 3})}.
\end{align*}
Therefore, one has
\begin{equation}
\left\Vert ay + c_1 - {\rm e}^{{\rm i}\chi y} c_1 \right\Vert_{H^1(\omega\times Y;\C^3)}\leq C \bigl\Vert  \simgrad \vect u \bigr\Vert _{L^2(\omega\times Y;\C^{3\times 3})}.
\end{equation}
Similarly, for the coefficient $b$ one has
\begin{equation}
\left\Vert by + c_2 - {\rm e}^{{\rm i}\chi y} c_2 \right\Vert_{H^1(\omega\times Y;\C^3)}\leq C \bigl\Vert  \simgrad \vect u \bigr\Vert _{L^2(\omega\times Y;\C^{3\times 3})}.
\end{equation}
Combining the above estimates for $a$ and $b,$ one obtains,  using \eqref{ab} and Lemma \ref{lemstartKorn},
\begin{align*}
\left\Vert ax_1 + bx_2 - {\rm e}^{{\rm i}\chi y}{\rm i}\chi(c_1x_1+c_2x_2) 
 \right\Vert_{L^2(\omega\times Y;\C^3)} 
&\leq\left\Vert x_1\left(a - {\rm e}^{{\rm i}\chi y}{\rm i}\chi c_1 \right) \right\Vert_{L^2(\omega\times Y;\C^3)} + \left\Vert x_2\left(b - {\rm e}^{{\rm i}\chi y} {\rm i}\chi c_2 \right)  \right\Vert_{L^2(\omega\times Y;\C^3)} \\[0.4em]
&\leq C_1 \bigl\Vert  \simgrad \vect u \bigr\Vert _{L^2(\omega\times Y;\C^{3\times 3})} + c_2 \bigl\Vert  \simgrad \vect u \bigr\Vert _{L^2(\omega\times Y;\C^3) } \\[0.4em]
& \leq C\bigl\Vert  \simgrad \vect u \bigr\Vert _{L^2(\omega\times Y;\C^{3\times 3})},
\end{align*}
\begin{align*}
\left\Vert \partial_{x_1}\left( ax_1 + bx_2 - {\rm e}^{{\rm i}\chi y} {\rm i}\chi (x_1 c_1 + x_2 c_2) \right) \right\Vert_{L^2(\omega\times Y;\C^3)}  &\leq  
\left\Vert a - {\rm e}^{{\rm i}\chi y}{\rm i}\chi  c_1 \right\Vert_{L^2(\omega\times Y;\C^3)} \leq C\bigl\Vert  \simgrad \vect u \bigr\Vert _{L^2(\omega\times Y;\C^{3\times 3})}, \\[0.4em]
\left\Vert \partial_{x_2}\left( ax_1 + bx_2 - {\rm e}^{{\rm i}\chi y} {\rm i}\chi (x_1 c_1 + x_2 c_2) \right) \right\Vert_{L^2(\omega\times Y;\C^3)} & \leq 
\left\Vert b - {\rm e}^{{\rm i}\chi y}{\rm i}\chi  c_2 \right\Vert_{L^2(\omega\times Y;\C^3)} \leq C\bigl\Vert  \simgrad \vect u \bigr\Vert _{L^2(\omega\times Y;\C^{3\times 3})}, \\[0.4em]
\left\Vert \partial_{y}\left( ax_1 + bx_2 - {\rm e}^{{\rm i}\chi y}{\rm i}\chi (x_1 c_1 + x_2 c_2) \right) \right\Vert_{L^2(\omega\times Y;\C^3)}  &\leq
\left\Vert \chi^2 {\rm e}^{{\rm i}\chi y} \left(  x_1c_1 + x_2c_2 \right) \right\Vert_{L^2(\omega\times Y;\C^3)} \\[0.4em]
 &\leq C{\chi^2}\chi^{-2}\bigl\Vert  \simgrad \vect u \bigr\Vert _{L^2(\omega\times Y;\C^{3\times 3})} =   C\bigl\Vert  \simgrad \vect u \bigr\Vert _{L^2(\omega\times Y;\C^{3\times 3})}.
\end{align*}
Summarising the above estimates, one has
\begin{equation*}
\left\Vert ax_1 + bx_2 - {\rm e}^{{\rm i}\chi y}{\rm i}\chi (x_1 c_1 + x_2 c_2) \right\Vert_{H^1(\omega\times Y;\C^3)} \leq C\bigl\Vert  \simgrad \vect u \bigr\Vert _{L^2(\omega\times Y;\C^{3\times 3})},
\end{equation*}
as required. 

In the similar way, one can replace 
$(dx_2, -dx_1, c_3)^\top$ in \eqref{ab_coef} with 
$e^{i\chi y}(dx_2, -dx_1, c_3)^\top,$ by
using \eqref{nak100} and the estimates on $c_3$, $d$ provided by Lemma \ref{lemstartKorn}. Therefore, $\widetilde{\vect w}$ defined in Lemma \ref{lemstartKorn} can be replaced by $\vect w$,  which finishes the proof. 
\end{proof}

\begin{remark}
For $\vect u \in H^1_{\chi}(Y;H^1(\omega;\C^3)),$ we denote

\begin{equation} \label{nak60} 
{\rm rod}(\vect u):= {\rm e}^{{\rm i}\chi y} ( d x_2, -d x_1, c_3)^\top + {\rm e}^{{\rm i}\chi y}\bigl(c_1, c_2,- {\rm i}\chi (c_1x_1 + c_2 x_2)\bigr)^\top.
\end{equation} 
The estimate \eqref{uw_est} can be written as follows:
\begin{equation}
\vect u ={\rm  rod}( \vect u) + O\bigl(\Vert \simgrad \vect u \Vert_{L^2(\omega\times Y;\C^{3\times 3})}\bigr).
\end{equation}
\CCC The expression ${\rm  rod}( \vect u)$ has the classical  Bernoulli-Navier  form for a rod displacement, restricted to a finite-dimensional space. 
\end{remark}

\subsection{\CCC Invariant subspaces and Korn-type inequalities  }
\label{sub_invar}
 Under the Assumption \ref{matsym}, one can reduce the problem of estimating the resolvent of the operator ${\mathcal A}_\chi$ to two simpler problems. More precisely, under these assumptions, we are able to identify two orthogonal subspaces  $L^2_{\rm bend}$ and $L^2_{\rm stretch}$ of the space $L^2(\omega\times Y;\C^3)$ that are invariant under the operator of elasticity\footnote{\RRR  We keep the same notation as in \eqref{invariant_subspaces_real} despite the slight difference in the function space. It will always be clear from the context to which definition we refer to. }. These subspaces consist of ``out-of-line" displacements, which we refer to as bending displacements, and ``in-line" displacements, which we refer to as stretching displacements:
\begin{equation}
    \begin{aligned}
        L^2_{\rm bend}&\equiv L^2_{\rm bend}(\omega\times Y;\C^3):=\left\{\vect u \in L^2(\omega\times Y;\C^3):\ \widehat{u}(-\widehat{x})=  \widehat{u} (\widehat{x}), \quad {u_3}(-\widehat{x}) = -{u_3}(\widehat{x})\right\}, \\[0.4em]
        L^2_{\rm stretch}&\equiv L^2_{\rm stretch}(\omega\times Y;\C^3):=\left\{\vect u \in L^2(\omega\times Y;\C^3):\ \widehat{u}(-\widehat{x}) = \vect -\widehat{u}(\widehat{x}), \quad {u_3}(-\widehat{x}) = {u_3}(\widehat{x})\right\}.
    \end{aligned} 
\label{invar_spaces_def_comp}
\end{equation}
Note that these spaces are mutually orthogonal also in $ H^1_{\chi}(Y;H^1(\omega;\C^3))$:
$$
\bigl[H^1_{\chi}(Y;H^1(\omega;\C^3))\cap L^2_{\rm bend}\bigr]^\perp = H^1_{\chi}\bigl(Y;H^1(\omega;\C^3)\bigr)\cap L^2_{\rm stretch}, 
$$
 which will be used in the proof of Propositon \ref{propnak10}. 
Furthermore, for $\vect u\in H^1_{\chi}(Y;H^1(\omega;\C^3))$, we have $\rm rod(\vect u) = \vect u_{\rm bend} + \vect u_{\rm stretch},$ where $\vect u_{\rm stretch}\in L^2_{\rm stretch}$ and $\vect u_{\rm bend}\in L^2_{\rm bend}$ are defined by  (cf. \eqref{nak60}): 
\begin{equation}
\begin{aligned}
\vect u_{\rm stretch}(\widehat{x},y)&:= {\rm e}^{{\rm i}\chi y} ( d x_2,-d x_1, c_3)^\top,\qquad(\widehat{x},y)\in\omega\times Y,\\[0.3em]
\vect u_{\rm bend}(\widehat{x},y)&:= {\rm e}^{{\rm i}\chi y}\bigl(c_1, c_2, -{\rm i}\chi (c_1x_1 + c_2 x_2)\bigr)^\top,\qquad(\widehat{x},y)\in\omega\times Y.
\end{aligned}
\label{decomp_vect}
\end{equation}
In view of the functions (\ref{decomp_vect}) being independent of the variable $y,$ in what follows we write $\vect u_{\rm stretch}=\vect u_{\rm stretch}(\widehat{x}),$ $\vect u_{\rm bend}=\vect u_{\rm
bend}(\widehat{x}).$ 
We next observe that $\simgrad \vect u_{\rm stretch} = {\rm i}X_\chi \vect u_{\rm stretch}.$   Using Proposition \ref{leadingordertermproposition}, we obtain the following estimates: 
\begin{equation}
    \begin{split}
        \Vert \vect u_{\rm bend}\Vert _{L^2(\omega\times Y,\C^3)} \leq C \max\left\{|c_1|, |c_2|\right\} \leq C  \chi^{-2}\bigl\Vert  \simgrad \vect u \bigr\Vert _{L^2(\omega\times Y;\C^{3\times 3})}, \\[0.4em]
\Vert \vect u_{\rm stretch}\Vert _{L^2(\omega\times Y,\C^3)} \leq C \max\left\{|d|, |c_3|\right\} \leq C|\chi|^{-1}\bigl\Vert  \simgrad \vect u \bigr\Vert _{L^2(\omega\times Y;\C^{3\times 3})}, \\[0.4em]
\bigl\Vert\simgrad \vect u_{\rm bend}\bigr\Vert _{L^2(\omega\times Y,\C^{3 \times 3})} \leq {\chi^2} \max\left\{|c_1|, |c_2|\right\} \leq C  \bigl\Vert  \simgrad \vect u \bigr\Vert _{L^2(\omega\times Y;\C^{3\times 3})} \\[0.4em]
\bigl\Vert\simgrad \vect u_{\rm stretch}\bigr\Vert _{L^2(\omega\times Y,\C^{3 \times 3})} \leq |\chi| \max\left\{|d|, |c_3|\right\} \leq C \bigl\Vert  \simgrad \vect u \bigr\Vert _{L^2(\omega\times Y;\C^{3\times 3})}.
    \end{split}
\end{equation}

The form of the vectors \eqref{decomp_vect} motivates introducing the following subspaces of $H^1_{\chi}(Y;H^1(\omega;\C^3)):$
\begin{align*}
	V_\chi^{\rm stretch}:=&\left\{(d x_2, -d x_1, c_3)^\top {\rm e}^{{\rm i}\chi y}, (\widehat{x},y)\in\omega\times Y: d, c_3 \in \C \right\}, 
	\\[0.4em]
V_\chi^{\rm bend}:=&\left\{(c_1,c_2,-{\rm i}\chi(c_1x_1+c_2 x_2))^\top {\rm e}^{{\rm i}\chi y}, (\widehat{x},y)\in\omega\times Y: c_1,c_2 \in \C \right\}. 
\end{align*}
 for which one has  
\begin{equation}
\dim V_\chi^{\rm stretch}=\dim V_\chi^{\rm bend} =2, \qquad  V_\chi^{\rm stretch}\perp V_\chi^{\rm bend},\qquad
	V_\chi^{\rm stretch} < L^2_{\rm stretch},\qquad V_\chi^{\rm bend} < L^2_{\rm bend}.
\end{equation}

The following estimates are crucial for the spectral analysis carried out in Section \ref{spec_analysis_sec}.

\begin{proposition}
\label{L2estim}
There exists a constant $C>0$ such that for all $\chi\in[-\pi,\pi)\setminus\{0\}$ the following estimates hold, depending on the space of deformations chosen:
\begin{itemize}
\item For all $\vect u  \in H^1_{\chi}(Y;H^1(\omega;\C^3)),$ one has
	\begin{equation}
	\label{estim1}
	\Vert \vect u\Vert _{L^2(\omega\times Y,\C^3)} \leq
	C\chi^{-2} \bigl\Vert  \simgrad \vect u \bigr\Vert _{L^2(\omega\times Y;\C^{3\times 3})}, 
	\end{equation}
\item For all $\vect u \in (V_\chi^{\rm bend})^\perp,$ one has
	\begin{equation}
	\label{estim2}
	\Vert \vect u\Vert _{L^2(\omega\times Y,\C^3)} \leq C|\chi|^{-1}\bigl\Vert  \simgrad \vect u \bigr\Vert _{L^2(\omega\times Y;\C^{3\times 3})},
	\end{equation}
\item For all $\vect u \in (V_\chi^{\rm bend} \cup V_\chi^{\rm stretch})^\perp,$ one has
	\begin{equation}
	\label{estim3}
	\Vert \vect u\Vert _{L^2(\omega\times Y,\C^3)} \leq C \bigl\Vert  \simgrad \vect u \bigr\Vert _{L^2(\omega\times Y;\C^{3\times 3})},
	\end{equation}
\end{itemize} 
\end{proposition}
\begin{proof}
The proof of all three estimates is based on the existence of $C>0$ such that for all vector functions $\vect u \in H^1_{\chi}(Y;H^1(\omega;\C^3))$ one has, \CCC as a consequence of Proposition \ref{leadingordertermproposition}, 
$$\Vert \vect u - \vect u_{\rm stretch} - \vect u_{\rm bend}\Vert _{L^2(\omega\times Y,\C^3)} \leq C\bigl\Vert  \simgrad \vect u \bigr\Vert _{L^2(\omega\times Y;\C^{3\times 3})}.$$

\begin{itemize}
	\item
For the bound \eqref{estim1}, we note that
\begin{align*}
\Vert \vect u\Vert _{L^2(\omega\times Y,\C^3)}&\leq \Vert \vect u - \vect u_{\rm stretch} - \vect u_{\rm bend}\Vert _{L^2(\omega\times Y,\C^3)} + \Vert \vect u_{\rm bend}\Vert _{L^2(\omega\times Y,\C^3)} + \Vert \vect u_{\rm stretch}\Vert _{L^2(\omega\times Y,\C^3)}\\[0.35em] &\leq  C\chi^{-2}\bigl\Vert  \simgrad \vect u \bigr\Vert _{L^2(\omega\times Y;\C^{3\times 3})}.
\end{align*}

\item
Next, for \eqref{estim2} is obtained by noting that for $\vect u \in (V_\chi^{\rm bend})^\perp$ one has
$$
\Vert \vect u - \vect u_{\rm stretch}\Vert _{L^2}^2 + \Vert \vect u_{\rm bend}\Vert _{L^2}^2 = \Vert \vect u - \vect u_{\rm stretch} - \vect u_{\rm bend}\Vert _{L^2(\omega\times Y,\C^3)}^2 \leq C\bigl\Vert  \simgrad \vect u \bigr\Vert _{L^2(\omega\times Y;\C^{3\times 3})}^2.
$$
Therefore
$$
\Vert \vect u - \vect u_{\rm stretch}\Vert _{L^2(\omega\times Y,\C^3)} \leq C\bigl\Vert  \simgrad \vect u \bigr\Vert _{L^2(\omega\times Y;\C^{3\times 3})}, $$
and hence
$$\Vert \vect u\Vert _{L^2(\omega\times Y,\C^3)} \leq \Vert \vect u - \vect u_{\rm stretch}\Vert _{L^2(\omega\times Y,\C^3)}  + \Vert \vect u_{\rm stretch}\Vert _{L^2(\omega\times Y,\C^3)} \leq  C|\chi|^{-1}\bigl\Vert  \simgrad \vect u \bigr\Vert _{L^2(\omega\times Y;\C^{3\times 3})}. $$

\item
Finally, for \eqref{estim3} we note that if $\vect u \in (V_\chi^{\rm bend} \cup V_\chi^{\rm stretch})^\perp$ then
\begin{align*}
\Vert \vect u\Vert _{L^2(\omega\times Y,\C^3)}^2 + \Vert \vect u_{\rm stretch}\Vert _{L^2(\omega\times Y,\C^3)}^2 + \Vert \vect u_{\rm bend}\Vert _{L^2(\omega\times Y,\C^3)}^2& = \Vert \vect u - \vect u_{\rm stretch} - \vect u_{\rm bend}\Vert _{L^2(\omega\times Y,\C^3)}^2\\[0.3em]
&\leq C\bigl\Vert  \simgrad \vect u \bigr\Vert _{L^2(\omega\times Y;\C^{3\times 3})}^2,
\end{align*}
which concludes the proof.
\end{itemize}
\end{proof}

\subsection{Spectral estimates}
\label{spec_analysis_sec}
By the Rellich-Kondrashov theorem, the space $H^1_{\chi}(Y;H^1(\omega;\C^3))$ is compactly embedded into $L^2_{\chi}( Y;L^2(\omega;\C^3))$. 
It follows that the spectrum of $\mathcal{A}_\chi$ \RRR is a sequence  of eigenvalues $(\lambda_n^\chi)_{n=1}^\infty$ that tends to infinity  and  can be assumed non-decreasing by a suitable rearrangement, if necessary.  

Here we state some results on the structure of the spectrum  of $\mathcal{A}_{\chi}$   and its scaling.
Recall the definition of the Rayleigh quotient associated with the bilinear form ${\mathfrak a}_\chi$, namely:
\begin{equation}
\mathcal{R}_\chi(\vect u) =  \frac{{\mathfrak a}( \vect u, \vect u)}{\Vert \vect u\Vert _{L^2(\omega\times Y,\C^3)}^2}, \quad \vect u \in H^1_{\chi}\bigl(Y;H^1(\omega;\C^3)\bigr).
\end{equation}
The Rayleigh quotient is closely related with the spectrum via the following characterizations:
\begin{equation}
\lambda_n^\chi = \min_{V \in L^n} \max_{ \vect v \in V} \mathcal{R}_\chi(\vect v),\qquad n=1,2,\dots.
\end{equation}
where $L^n$ denotes the family of $n$-dimensional subspaces of $H^1_{\chi}(Y;H^1(\omega;\C^3))$. Alternatively,
\begin{equation}
\lambda_n^\chi = \min \left\{\mathcal{R}_\chi(\vect v); \quad {\vect v}\perp{\vect v}_i \mbox{ in } L^2(\omega\times Y;\C^3),\quad  1\leq i \leq n-1\right\},\qquad n=1,2,\dots,
\end{equation}
where $\{{\vect v}_i,$ $1\le i\le n-1\}$ is any set spanning the same subspace of $H^1_{\chi}(Y;H^1(\omega;\C^3))$ as the first $n-1$ eigenvectors of $\mathcal{A}_\chi,$ $n=1,2,\dots,$ in the order of increasing $\lambda_n^\chi.$ 
\begin{remark}
Here we note that the function $\chi \to \lambda_n^{\chi}$ is continuous for all $n \in \N$. Therefore, the spectrum of the operator $\mathcal{A}_\varepsilon$ is a union of intervals $$
\bigl[\underline{\lambda}_n, \overline{\lambda}_n\bigr] := \bigcup_{\chi \in [-\pi,\pi)} \bigl\{\lambda_n^{\chi}\bigr\}, \quad n \in \N.$$
\end{remark}

The lowest eigenvalues of ${\mathcal A}_\chi$ in the invariant subspaces $V_\chi^{\rm bend}$, $V_\chi^{\rm stretch}$ have different orders of smallness with respect to the quasimomentum $\chi,$ as stated in the following proposition.

\begin{proposition}
\label{Rayleighestim}
There exist constants \CCC $C_1 > C_2 > 0$  such that
\begin{itemize}
\item $\mathcal{R}_\chi(\vect u) \geq \CCC C_2{\chi^4}  \qquad\forall \vect u \in H^1_{\chi}\bigl(Y;H^1(\omega;\C^3)\bigr);$
\item  $0\leq\mathcal{R}_\chi(\vect u) \leq C_1{\chi^4}\qquad \forall \vect u \in V_\chi^{\rm bend};$
\item $\mathcal{R}_\chi(\vect u) \geq \CCC C_2{\chi^2} \qquad\forall \vect u \in (V_\chi^{\rm bend})^\perp;$ 
\item $0\leq\mathcal{R}_\chi(\vect u) \leq C_1{\chi^2}\qquad\forall \vect u \in V_\chi^{\rm stretch};$
\item $\mathcal{R}_\chi(\vect u) \geq \CCC C_2 \qquad \forall \vect u \in (V_\chi^{\rm bend} \cup V_\chi^{\rm stretch})^\perp.$
\end{itemize}
\end{proposition}
\begin{proof}
Using the uniform positive definiteness of the tensor $\A$ together with the estimates \eqref{estim1}, we obtain
\begin{equation}
\mathcal{R}_{\chi}(\vect u) = \frac{{ \mathfrak a}(\vect u,\vect u)}{\Vert \vect u\Vert _{L^2(\omega\times Y ,\C^3)}^2} \geq \nu\frac{\Vert \simgrad  \vect u\Vert _{L^2(\omega\times Y ,\C^{3 \times 3})}^2}{\Vert \vect u\Vert _{L^2(\omega\times Y ,\C^3)}^2} \geq C\nu{\chi^4}.
\end{equation} 
For arbitrary $\vect u= (c_1,c_2,-{\rm i}\chi (c_1 x_1 + c_2 x_2))^\top {\rm e}^{{\rm i}\chi y} \in V_\chi^{\rm bend},$ we estimate:
\begin{align*}
\bigl\Vert \simgrad \vect u \bigr\Vert_{L^2(\omega\times Y;\C^{3\times 3})}&\leq  \max\Bigl\{\sqrt{{\mathfrak c}_1(\omega)},\sqrt{{\mathfrak c}_2(\omega)}\Bigr\}{\chi^2}\sqrt{c_1^2 + c_2^2},\\[0.25em]
\Vert \vect u\Vert _{L^2(\omega\times Y ,\C^3)}&\geq  \sqrt{c_1^2 + c_2^2}, 
\end{align*}
where ${\mathfrak c}_1(\omega),$ ${\mathfrak c}_2(\omega)$ are defined by \eqref{crossectionconstants}. By combining the above, it follows that 
\begin{equation}
\mathcal{R}_{\chi}(\vect u) = \frac{{ \mathfrak a }(\vect u,\vect u)}{\Vert \vect u\Vert _{L^2(\omega\times Y ,\C^3)}^2} \leq \frac{1}{\nu}\frac{\bigl\Vert \simgrad\vect u \bigr\Vert_{L^2(\omega\times Y;\C^{3\times 3})}^2}{\Vert \vect u\Vert _{L^2(\omega\times Y ,\C^3)}^2} \leq \CCC \frac{C}{\nu}{\chi^4}. 
\end{equation}
Next, for an arbitrary $\vect u = (d x_2, -d x_1, c_3)^\top {\rm e}^{{\rm i}\chi y}\in V_\chi^{\rm stretch},$  one has 
\begin{equation}
\bigl\Vert \simgrad \vect u \bigr\Vert_{L^2(\omega\times Y;\C^{3\times 3})} \leq  |\chi|\max\Bigl\{\sqrt{{\mathfrak c}_1(\omega)+{\mathfrak c}_2(\omega)},1\Bigr\}\sqrt{d^2 + c_3^2}  ,
\end{equation}
while 
\begin{equation}
\Vert \vect u \Vert _{L^2(\omega\times Y ,\C^3)} \geq \min\Bigl\{\sqrt{{\mathfrak c}_1(\omega)+{\mathfrak c}_2(\omega)},1\Bigr\}\sqrt{d^2 + c_3^2} .
\end{equation}
It follows that 
\begin{equation}
\mathcal{R}_\chi (\vect v) \leq  \CCC \frac{C}{\nu}{\chi^2} ,
\end{equation}
for some  \CCC $C>0.$  Combining this with the estimates of Proposition \ref{L2estim} concludes the proof.
\end{proof}

The previous proposition allows us to deduce the following result on the structure and the scaling of the spectrum: 
\begin{theorem}
 The spectrum $\sigma(\mathcal{A}_\chi)$ contains two eigenvalues of order ${O}({\chi^4})$, two eigenvalues of order ${O}({\chi^2}),$ and the remaining eigenvalues are of order ${O}(1).$ 
 
Under Assumption \ref{matsym} on the rod cross-section and  material symmetries, the spectrum $\sigma(\mathcal{A}_\chi)$ is a disjoint union of $\sigma\bigl(\mathcal{A}_\chi |_{L^2_{\rm bend}}\bigr)$ and $\sigma\bigl(\mathcal{A}_\chi |_{L^2_{\rm stretch}}\bigr)$. The spectrum $\sigma\bigl(\mathcal{A}_\chi |_{L^2_{\rm bend}}\bigr)$ contains two eigenvalues of order ${O}({\chi^4}),$ while the remaining eigenvalues are of order ${O}(1)$. In contrast, the spectrum $\sigma\bigl(\mathcal{A}_\chi |_{L^2_{\rm stretch}}\bigr)$ contains two eigenvalues of order ${O}({\chi^2}),$ while the remaining eigenvalues are of order ${O}(1).$
\end{theorem}
\begin{proof}
The proof is based on on characterising of eigenvalues via the min-max principle and estimating the relevant Rayleigh quotients.
\end{proof}
Under the symmetry assumptions, the spectrum of the operator $\mathcal{A}_\chi$ can be split into the sets $\sigma(\mathcal{A}_\chi |_{L^2_{\rm bend}})$ and $\sigma(\mathcal{A}_\chi |_{L^2_{\rm stretch}}),$ so it is of interest to perform the asymptotic analysis of the resolvent problems for the scaled operators $\chi^{-4}\mathcal{A}_\chi |_{L^2_{\rm bend}}$ and
$\chi^{-2}\mathcal{A}_\chi |_{L^2_{\rm stretch}}$. 
The following proposition provides us with Korn type inequalities which are crucial for calculating apriori estimates for the resolvent problems. 
\begin{proposition}\label{propnak10} 
\label{aprioriestimates}There exists $C>0$ such that for every $\chi \in [-\pi,\pi)\setminus\left\{0\right\}:$
\begin{itemize}
	\item
For every $\vect u\in H^1_{\chi}(Y;H^1(\omega;\C^3))\cap L^2_{\rm stretch},$ one has
\begin{equation*}
\Vert \vect u\Vert _{H^1(\omega\times Y, \C^3)} \leq C|\chi|^{-1}\bigl\Vert \simgrad \vect u \bigr\Vert_{L^2(\omega\times Y;\C^{3\times 3})}.
\end{equation*}
\item
For every $\vect u\in H^1_{\chi}(Y;H^1(\omega;\C^3))\cap L^2_{\rm bend},$ one has
\begin{align*}
    \Vert {u_1}\Vert _{H^1(\omega\times Y, \C)}&\leq C\chi^{-2}\bigl\Vert  \simgrad \vect u \bigr\Vert _{L^2(\omega\times Y;\C^{3\times 3})},\\[0.35em]
    \quad \Vert {u_2}\Vert _{H^1(\omega\times Y, \C)}&\leq C\chi^{-2}\bigl\Vert  \simgrad \vect u \bigr\Vert _{L^2(\omega\times Y;\C^{3\times 3})}, \\[0.35em] 
    \Vert {u_3}\Vert _{H^1(\omega\times Y, \C)}&\leq C|\chi|^{-1}\bigl\Vert  \simgrad \vect u \bigr\Vert _{L^2(\omega\times Y;\C^{3\times 3})}.
\end{align*}
\end{itemize}
\end{proposition}
\begin{proof}
The proof is based on the fact that the spaces   $H^1_{\chi}(Y;H^1(\omega;\C^3))\cap L^2_{\rm stretch}$ and $H^1_{\chi}(Y;H^1(\omega;\C^3))\cap L^2_{\rm bend}$ are orthogonal with respect to the $H^1(\omega \times Y;\C^3)$ inner product.  
\end{proof}
\begin{remark}
This heterogeneity in  component-wise  estimates allows the scaling of the third component of the force terms in the bending case. 
\end{remark}

\section{Asymptotic analysis of  scaled resolvents of $\mathcal{A}_{\chi}$ }
\label{section4}

The purpose of this section is to establish estimates on the difference between the solutions to the resolvent problems for ${\mathcal A}_\chi$ and the resolvent problem for
an appropriate 
 effective (``homogenised") operator introduced in Section \ref{hom_op_sec}.  
The norm of the difference is estimated explicitly in terms of $|\chi|$ and the norm of the force density. 
In order to achieve this, we perform an asymptotic expansion of the solution to the resolvent problem for ${\mathcal A}_\chi,$ with its leading-order term provided by the solution to the effective resolvent problem. 
Before doing so, we introduce some auxiliary objects  describing the homogenised material properties. 

\subsection{Auxiliary notation and definitions}
\label{aux_obj_sec}

First, for each $\chi\in[-\pi,\pi),$ 
consider the following embedding operators:
\begin{equation}
    \begin{aligned}
        \widetilde{\mathcal{I}}^{\rm bend}_\chi&:\C^2 \to L^2_{\rm bend}(\omega\times Y;\C^3), \quad \widetilde{\mathcal{I}}^{\rm bend}_\chi \begin{bmatrix} m_1 \\[0.15em] m_2\end{bmatrix} = \begin{bmatrix}m_1 \\[0.3em] m_2 \\[0.15em] -{\rm i}\chi (x_1 m_1 + x_2 m_2)\end{bmatrix}, \\[0.7em]
        \widetilde{\mathcal{I}}^{\rm stretch}&: \C^2\to L^2_{\rm stretch}(\omega\times Y;\C^3), \quad   \widetilde{\mathcal{I}}^{\rm stretch}\begin{bmatrix}m_3 \\[0.15em] m_4\end{bmatrix} = \begin{bmatrix}x_2 m_3 \\[0.15em] -x_1 m_3 \\[0.15em]   m_4\end{bmatrix}, \\[0.7em]
        \widetilde{\mathcal{I}}^{\rm rod}_\chi&:\C^4 \to L^2(\omega \times Y;\C^3), \quad 
        \widetilde{\mathcal{I}}^{\rm rod}_\chi{\vect m} = \widetilde{\mathcal{I}}^{\rm bend}_\chi \begin{bmatrix} m_1 \\[0.15em] m_2\end{bmatrix} + \widetilde{\mathcal{I}}^{\rm stretch}\begin{bmatrix}m_3 \\[0.15em] m_4\end{bmatrix},\quad {\vect m}= \CCC (m_1, m_2, m_3, m_4)^\top .
    \end{aligned}
\end{equation}
These operators serve as a link between the appropriate Euclidean spaces and the finite-dimensional subspaces of rod displacement approximations. 

Second, we define the complex force-and-momentum operators, as follows: for each $\chi \in [-\pi,\pi),$ $\vect f \in L^2(\omega\times Y;\C^3)$ we set
\begin{equation}
\label{forcemomentumcomplex}
\begin{aligned}
        \widetilde{\mathcal{M}}_\chi^{\rm bend}{\vect f} 
        &:= \int_{\omega\times Y}(\,\widehat{\!\vect f}+{\rm i}\chi f_3\widehat{x})
        \in \C^2, 
        \qquad 
        \widetilde{\mathcal{M}}^{\rm stretch}{\vect f} 
         := \int_{\omega\times Y}\begin{bmatrix} x_2 {f_1} - x_1 {f_2} \\[0.25em] {f_3}  \end{bmatrix}\in \C^2, \\[0.8em]
        \widetilde{\mathcal{M}}_\chi^{\rm rod}{\vect f}
         &:= \begin{bmatrix}
           \widetilde{\mathcal{M}}_\chi^{\rm bend} \vect f \\[0.3em] \widetilde{\mathcal{M}}^{\rm stretch} \vect f
        \end{bmatrix} = \int_{\omega\times Y}\begin{bmatrix}
           \,\widehat{\!\vect f}+{\rm i}\chi f_3\widehat{x}
           \\[0.25em]  x_2 {f_1} - x_1 {f_2} \\[0.1em] {f_3}
        \end{bmatrix}\in \C^4.
\end{aligned}
\end{equation}
These momentum operators satisfy the following estimates: 
\begin{align*}
     \left\Vert \widetilde{\mathcal{M}}_\chi^{\rm bend} \vect f   \right\Vert
     \leq \bigl\Vert\,\widehat{\!\vect f} \bigr\Vert_{L^2(\omega \times Y; {\mathbb C}^2)}
     + |\chi|\left\Vert {f_3} \right\Vert_{L^2(\omega \times Y)},\qquad
     \left\Vert \widetilde{\mathcal{M}}^{\rm stretch} \vect f\right\Vert
     \leq \left\Vert \vect f \right\Vert_{L^2(\omega \times Y;\C^3)}, \qquad \left\Vert \widetilde{\mathcal{M}}_\chi^{\rm rod} \vect f   \right\Vert \leq \left\Vert \vect f \right\Vert_{L^2(\omega \times Y;\C^3)}.
\end{align*}
Furthermore, 
\begin{equation}
    \begin{split}
        \mathfrak{G}_\varepsilon^{-1} \widetilde{\mathcal{M}}_\chi^{\rm bend} \mathfrak{G}_\varepsilon \vect f& = \mathfrak{G}_\varepsilon^{-1} \int_{\omega\times Y}\bigl(
        \mathfrak{G}_\varepsilon\,\widehat{\!\vect f}+{\rm i}\chi \mathfrak{G}_\varepsilon({f_3}\widehat{x})\bigr)
        = \int_\omega \left\{\mathfrak{G}_\varepsilon^{-1} \int_Y \mathfrak{G}_\varepsilon\,\widehat{\!\vect f}+
        \mathfrak{G}_\varepsilon^{-1}
        \biggl({\rm i}\chi \int_Y \mathfrak{G}_\varepsilon{f_3}\biggr)\widehat{x}
         \right\}\\[0.8em]
        &= \int_\omega\left\{\mathfrak{G}_\varepsilon^{-1} \int_Y \mathfrak{G}_\varepsilon\,\widehat{\!\vect f}- \varepsilon\biggl(\dfrac{d}{dx_3}\mathfrak{G}_\varepsilon^{-1} \int_Y \mathfrak{G}_\varepsilon{f_3}\biggr)\widehat{x} 
        \right\}=\mathcal{M}_\varepsilon^{\rm bend} \Xi_\varepsilon \vect f.
    \end{split}
\end{equation}
Similarly, one has 
\begin{equation}
    \mathfrak{G}_\varepsilon^{-1} \widetilde{\mathcal{M}}^{\rm stretch} \mathfrak{G}_\varepsilon \vect f=\mathcal{M}^{\rm stretch} \Xi_\varepsilon \vect f, \qquad \mathfrak{G}_\varepsilon^{-1} \widetilde{\mathcal{M}}_\chi^{\rm rod} \mathfrak{G}_\varepsilon \vect f = \mathcal{M}_\varepsilon^{\rm rod} \Xi_\varepsilon \vect f.
\end{equation}

In addition to ${\mathfrak C}^{\rm stretch}$ and ${\mathfrak C}^{\rm rod}$ defined in Section \ref{main_results_sec}, for each $\chi\in[-\pi,\pi)$ we introduce the following matrices, which contain further information about the rod  cross-section and appear in calculations:
\begin{equation}
	\label{cstretchrodbend}
	\begin{aligned}
		\mathfrak{C}^{\rm bend}_{\chi}(\omega)&:= 
		\begin{bmatrix}
			1 + {\chi^2}{\mathfrak c}_1(\omega) & 0 \\[0.3em]
			0 & 1 + {\chi^2}{\mathfrak c}_2(\omega)
		\end{bmatrix}, \\[0.9em]
		\mathfrak{C}^{\rm rod}_{\chi}(\omega)&:= \begin{bmatrix}
			\mathfrak{C}^{\rm bend}_{\chi}(\omega) & 0 \\[0.3em]
			0 & \mathfrak{C}^{\rm stretch}(\omega)
		\end{bmatrix} = 
		\begin{bmatrix}
			1 + {\chi^2}{\mathfrak c}_1(\omega) & 0 & 0 & 0 \\[0.3em]
			0 & 1 + {\chi^2}{\mathfrak c}_2(\omega) & 0 & 0 \\[0.3em]
			0 & 0 & {\mathfrak c}_1(\omega)+{\mathfrak c}_2(\omega) & 0 \\[0.3em]
			0 & 0 & 0 & 1
		\end{bmatrix}.
	\end{aligned}
\end{equation}
As was mentioned just above Theorem \ref{THML2L2}, we shall normally drop `$(\omega)$' in the notation.

Note that for every ${\vect m} = (m_1,m_2,m_3,m_4)^\top$, ${\vect d} = (d_1,d_2,d_3,d_4)^\top \in \C^4$, one has
\begin{equation}
    \int_{\omega\times Y}
    \widetilde{\mathcal{I}}_\chi^{\rm bend}
\begin{bmatrix}
    m_1 \\ m_2 
\end{bmatrix}
\cdot 
\overline{\widetilde{\mathcal{I}}_\chi^{\rm bend}\begin{bmatrix}
    d_1 \\ d_2 
\end{bmatrix}} = \mathfrak{C}^{\rm bend}_{\chi} \begin{bmatrix}
    m_1 \\ m_2 
\end{bmatrix} \cdot \overline{\begin{bmatrix}
    d_1 \\ d_2 
\end{bmatrix} },
\end{equation}
\begin{equation}
\int_{\omega\times Y} \widetilde{\mathcal{I}}^{\rm stretch}
 \begin{bmatrix}
    m_3   \\
    m_4 
\end{bmatrix} \cdot 
\overline{\widetilde{\mathcal{I}}^{\rm stretch}\begin{bmatrix}
    d_3   \\
    d_4 
\end{bmatrix} } =\mathfrak{C}^{\rm stretch}\begin{bmatrix}
    m_3   \\
    m_4 
\end{bmatrix} \cdot 
\overline{\begin{bmatrix}
    d_3   \\
    d_4 
\end{bmatrix}},
\end{equation}
\begin{equation}
  \int_{\omega\times Y}\widetilde{\mathcal{I}}_\chi^{\rm rod}{\vect m} \cdot \overline{\widetilde{\mathcal{I}}_\chi^{\rm rod}{\vect d}} = \mathfrak{C}^{\rm rod}_{\chi}{\vect m}\cdot \overline{\vect d}.
\end{equation}
Furthermore, for all $\vect f \in L^2(\omega \times Y;\C^3)$, $\vect d = (d_1,d_2,d_3,d_4)^\top \in \C^4$, one has
\begin{equation}
\begin{aligned}
    \int_{\omega\times Y}
\vect f 
\cdot 
\overline{\widetilde{\mathcal{I}}_\chi^{\rm bend} \begin{bmatrix}
    d_1   \\[0.2em]
    d_2 
\end{bmatrix}}   = \widetilde{\mathcal{M}}_\chi^{\rm bend} \vect f \cdot \overline{\begin{bmatrix}
    d_1   \\[0.2em]
    d_2 
\end{bmatrix}},&\quad\quad
\int_{\omega\times Y} 
\vect f \cdot 
\overline{\widetilde{\mathcal{I}}^{\rm stretch}\begin{bmatrix}
    d_3   \\[0.2em]
    d_4 
\end{bmatrix}} =  \widetilde{\mathcal{M}}^{\rm stretch} \vect f \cdot \overline{\begin{bmatrix}
    d_3   \\[0.2em]
    d_4 
\end{bmatrix}},\\[0.9em]
  \int_{\omega\times Y}\vect f \cdot \overline{\widetilde{\mathcal{I}}_\chi^{\rm rod}{\vect d}} &= \widetilde{\mathcal{M}}_\chi^{\rm rod} \vect f\cdot \overline{\vect d}.
\end{aligned}
\label{f_id}
\end{equation}
%
The identities \eqref{f_id} can be equivalently expressed as follows:
\begin{equation}
    \begin{split}
        \widetilde{\mathcal{I}}^{\rm stretch}= (\widetilde{\mathcal{M}}^{\rm stretch})^*, \qquad & \widetilde{\mathcal{I}}^{\rm bend}_\chi= (\widetilde{\mathcal{M}}^{\rm bend}_\chi)^*, \qquad
        \widetilde{\mathcal{I}}^{\rm rod}_\chi= (\widetilde{\mathcal{M}}^{\rm rod}_\chi)^*.
    \end{split}
\end{equation}

Finally, we introduce the following matrices defined for each $\chi\in[-\pi,\pi),$ ${\vect m}= (m_1,m_2,m_3,m_4) \in \R^4,$ $\widehat{x}=(x_1, x_2)\in \omega$ and contain information on the symmetrized gradient of  Bernoulli-Navier  deformations (cf. \eqref{I_bend_stretch}):
\begin{align}
    \Lambda_{\chi,m_3,m_4}^{\rm stretch}(\widehat{x})&:={\rm i}\chi\mathcal{J}_{m_3,m_4}^{\rm stretch}(\widehat{x})={\rm i}\chi\begin{bmatrix}
     0 & 0 & \dfrac{x_2 m_3}{2} \\[0.8em]
     0 & 0 & \dfrac{-x_1 m_3}{2} \\[0.8em]
     \dfrac{x_2 m_3}{2} & \dfrac{-x_1 m_3}{2} &  m_4
\end{bmatrix}=\simgrad \begin{bmatrix}
m_3 x_2   \\
-m_3 x_1 \\
m_4
\end{bmatrix} + {\rm i}X_{\chi}\begin{bmatrix}
m_3 x_2   \\
-m_3 x_1 \\
m_4
\end{bmatrix},\label{lam_bend}
\\[0.9em]
\Lambda_{\chi,m_1,m_2}^{\rm bend}(\widehat{x})&:= ({\rm i}\chi)^2 \mathcal{J}_{m_1,m_2}^{\rm bend}(\widehat{x})= ({\rm i}\chi)^2 \begin{bmatrix}
     0 & 0 & 0 \\
     0 & 0 & 0 \\
     0 & 0 & -m_1 x_1 - m_2 x_2
\end{bmatrix}\\[0.9em]
&\,=\simgrad  \begin{bmatrix}
m_1   \\
m_2 \\
-{\rm i}\chi (m_1 x_1 + m_2 x_2)
\end{bmatrix}
+ {\rm i}X_{\chi}\begin{bmatrix}
m_1   \\
m_2 \\
-{\rm i}\chi (m_1 x_1 + m_2 x_2)
\end{bmatrix},\label{lam_stretch}
\\[0.9em]
    \Lambda_{\chi,{\vect m}}^{\rm rod}(\widehat{x})&:= \Lambda_{\chi,m_1,m_2}^{\rm bend}(\widehat{x}) + \Lambda_{\chi,m_3,m_4}^{\rm stretch}(\widehat{x}).\nonumber
\end{align}
In what follows, we make use of the following estimates,  which can be obtained easily with some $C_1,$ $C_2>0$: 
\begin{eqnarray*}
C_1{\chi^2}\bigl|(m_1,m_2)^\top\bigr| \leq \bigl\Vert\Lambda_{\chi,m_1,m_2}^{\rm bend}\bigr\Vert_{L^2(\omega \times Y;\C^{3\times 3})} \leq C_2{\chi^2} \bigl|(m_1,m_2)^\top\bigr|, \\[0.4em]
C_1|\chi|\bigl|(m_3,m_4)^\top\bigr| \leq \bigl\Vert\Lambda_{\chi,m_3,m_4}^{\rm stretch}\bigr\Vert_{L^2(\omega \times Y;\C^{3 \times 3})} \leq C_2|\chi| \bigl|(m_3,m_4)^\top\bigr|,
\end{eqnarray*}

\subsection{Homogenised tensors in $\chi$-representation}

Now, consider the problem of finding $\vect u\in H_\#^1(Y;H^1(\omega;\C^3))$ such that
\begin{equation}
\label{perprob}
\int_{\omega\times Y}{\mathbb A}(y)\simgrad \vect u(\widehat{x},y): \overline{\simgrad \vect v(\widehat{x},y)}d\widehat{x}dy  = \int_{\omega\times Y} \vect f(\widehat{x},y)\cdot\overline{\vect v(\widehat{x},y)}\,d\widehat{x}dy\qquad \forall \vect v 
\in H_\#^1\bigl(Y;H^1(\omega;\C^3)\bigr).
\end{equation}
In order for this problem to be well posed, the right-hand side $\vect f$ must be orthogonal to the kernel of the operator $\simgrad$. The kernel of this operator consists of infinitesimal rigid motions, but since we are dealing with functions periodic in $y$ we have
\begin{equation}
\label{Hker}
H_\#^1\bigl(Y;H^1(\omega;\C^3)\bigr) \cap \ker(\simgrad) = \left\{ \begin{bmatrix} dx_2 + c_1\\[0.25em] -d x_1 + c_2 \\[0.25em] c_3 \end{bmatrix}, c_1,c_2,c_3,d \in \C \right\}< H_\#^1\bigl(Y;H^1(\omega;\C^3)\bigr). 
\end{equation}
In order to address the solvability of \eqref{perprob}, we consider the closed subspace of $H_\#^1\bigl(Y;H^1(\omega;\C^3)$ defined as the orthogonal complement of the kernel \eqref{Hker}  with the respect to the $L^2(\omega\times Y)$ inner product:
\begin{align*}
H:=
\left[H_\#^1\bigl(Y;H^1(\omega;\C^3)\bigr) \cap \ker(\simgrad) \right]^\perp  
=\left\{ \vect u \in H_\#^1\bigl(Y;H^1(\omega;\C^3)\bigr): \int_{\omega \times Y}\vect u = 0,\ \  \int_{\omega \times Y} (x_2u_1 - x_1u_2) = 0 \right\}, 
\end{align*}
where the the orthogonal complement is understood in the sense of the $L^2$ inner product. Note 
that the Korn and Poincar\'{e} inequalities hold on $H,$ and for $\vect f\in H$ the Lax-Milgram theorem yields the existence of a unique solution $\vect u\in H$ to the problem \eqref{perprob}, where test functions taken in the space $H$. Since $\vect f\in H$, this is equivalent to allowing arbitrary test functions $\vect v\in H_\#^1(Y;H^1(\omega;\C^3))$. 

The above argument allows us to consider the following 
well-posed problem: find $\vect u_{\chi,{\vect m}} \in H$ such that
\begin{equation*}
\int_{\omega \times Y} {\mathbb A}(y) \left(\simgrad \vect u_{\chi,{\vect m}}(\widehat{x},y) + \Lambda_{\chi,{\vect m}}^{\rm rod}(\widehat{x}) \right):\overline{\simgrad \vect v(\widehat{x},y)}\,d\widehat{x}dy = 0
\qquad \forall \vect v\in H.
\end{equation*}
It has a unique solution and can be equivalently rewritten by using the adjoint of the operator $\simgrad$, namely 
$$
(\simgrad)^*: L_\#^2\bigl(Y;H^1(\omega;\C^3)\bigr) \to H, $$ as follows:
\begin{equation}
\label{achirodcorrector}
(\simgrad)^* {\mathbb A}(y) \simgrad \vect u_{\chi,{\vect m}}(\widehat{x},y)= -(\simgrad)^* {\mathbb A}(y) \Lambda_{\chi,{\vect m}}^{\rm rod}(\widehat{x}).
\end{equation}
Next, we define the \CCC Hermitian  matrix ${\mathbb A}_\chi^{\rm rod}  \in \C^{4\times4}$ via
\begin{equation}
{\mathbb A}_\chi^{\rm rod}{\vect m}\cdot \overline{\vect d} = \int_{\omega \times Y} {\mathbb A}(y)\left( \simgrad \vect u_{\chi,{\vect m}}(\widehat{x},y) + \Lambda_{\chi,{\vect m}}^{\rm rod}(\widehat{x}) \right) :\overline{\Lambda_{\chi,{\vect d}}^{\rm rod}(\widehat{x})}\,d\widehat{x}dy,\qquad {\vect m}, {\vect d}\in{\mathbb C}^4.
\label{arodchi}
\end{equation} 
\CCC 
Using the linearity in $\chi$ and the uniqueness of the solution to \eqref{achirodcorrector}, we infer that
\begin{eqnarray*}
 {\mathbb A}_\chi^{\rm rod}{\vect m}\cdot \overline{\vect d}&=& {\mathbb A}^{\rm rod}\left(({\rm i}\chi)^2(m_1,m_2,0,0)^\top+{\rm i}\chi(0,0,m_3,m_4)^\top\right) \cdot \overline{\left(({\rm i}\chi)^2(d_1,d_2,0,0)^\top+{\rm i}\chi(0,0,d_3,d_4)^\top\right)}, \\ & & \qquad \qquad \vect m,\vect d \in \C^4,        
\end{eqnarray*} 
where $\mathbb{A}^{\rm rod}$ is defined in Proposition \ref{proposition21}.


\CCC In the case when Assumption \ref{matsym} is satisfied  we analyse the problem \eqref{perprob} separately on two invariant subspaces, namely (cf. Section \ref{sub_invar})
\begin{align*}
H^{\rm stretch} &:= \left[H_\#^1\bigl(Y;H^1(\omega;\C^3)\bigr) \cap \ker(\simgrad)\cap L^2_{\rm stretch} \right]^\perp  \\[0.3em] 
&\hspace{8em}=  \left\{ \vect u \in H_\#^1\bigl(Y;H^1(\omega;\C^3)\bigr)\cap L^2_{\rm stretch}: \int_{\omega \times Y} {u_3}  = 0, \int_{\omega \times Y}(x_1u_2 - x_2u_1)= 0 \right\}, \\[0.3em]
H^{\rm bend} &:=\left[H_\#^1\bigl(Y;H^1(\omega;\C^3)\bigr) \cap \ker(\simgrad) \cap L^2_{\rm bend}\right]^\perp  \\[0.3em]
&\hspace{8em}=\left\{\vect u \in H_\#^1\bigl(Y;H^1(\omega;\C^3)\bigr)\cap L^2_{\rm bend}: \int_{\omega \times Y} u_1 = 0, \int_{\omega \times Y}u_2 = 0\right\},
\end{align*}
where, as in the case of the definition of the space $H,$ the orthogonal complements are understood in the sense of the corresponding $L^2$ inner products. Note that
\begin{align*}
	H_\#^1\bigl(Y;H^1(\omega;\C^3)\bigr) \cap \ker(\simgrad) \cap L^2_{\rm stretch}&= \left\{ \begin{bmatrix} dx_2 \\ -d x_1  \\ c_3 \end{bmatrix}, c_3,d \in \C \right\} < H_\#^1(Y;H^1(\omega;\C^3)),\\[0.35em]
H_\#^1\bigl(Y;H^1(\omega;\C^3)\bigr) \cap \ker(\simgrad)\cap L^2_{\rm bend}&= \left\{ \begin{bmatrix}  c_1\\  c_2 \\ 0 \end{bmatrix}, c_1,c_2 \in \C \right\}< H_\#^1(Y;H^1(\omega;\C^3)).
\end{align*}
We consider the solutions $\vect u_{\chi,m_3, m_4}^{\rm stretch}\in H^{\rm stretch},$ $\vect u_{\chi,m_1, m_2}^{\rm bend}\in H^{\rm bend}$  to the problems
\begin{align} \label{nak61} 
	(\simgrad)^* {\mathbb A}(y) \simgrad \vect u_{\chi,m_3, m_4}^{\rm stretch}(\widehat{x}, y)&= -(\simgrad)^* {\mathbb A}(y) \Lambda_{\chi,m_3,m_4}^{\rm stretch}(\widehat{x}), \\[0.45em]
(\simgrad)^* {\mathbb A}(y) \simgrad \vect u_{\chi,m_1, m_2}^{\rm bend}(\widehat{x},y)&= -(\simgrad)^* {\mathbb A}(y) \Lambda_{\chi,m_1,m_2}^{\rm bend}(\widehat{x}),
\end{align}
 which are clearly well posed, as a consequence of the fact that the range of $(\simgrad)^*$ is orthogonal to the kernel of $\simgrad$. 
Note that under Assumption \ref{matsym} on the cross-section geometry and material symmetries, each of the the above problems has a unique solution.  
Furthermore, \CCC we have 
\begin{align*}
	\chi^2{\mathbb A}_{\rm stretch}  (m_3,m_4)^{\top}\cdot \overline{(d_3,d_4)^{\top}}&= \int_{\omega \times Y} {\mathbb A}(y)\left( \simgrad \vect u_{\chi,m_3, m_4}^{\rm stretch}(\widehat{x},y) + \Lambda_{\chi,m_3,m_4}^{\rm stretch}(\widehat{x}) \right) :\overline{\Lambda_{\chi,d_3,d_4}^{\rm stretch}(\widehat{x})} d\widehat{x}dy,\\[0.3em]
	&\hspace{7em}\quad (m_1, m_2)^\top, (d_1, d_2)^\top\in{\mathbb C}^2,\\[0.4em]
\chi^4{\mathbb A}^{\rm bend}  (m_1,m_2)^{\top}\cdot \overline{(d_1,d_2)^{\top}}&= \int_{\omega \times Y} {\mathbb A}(y)\left( \simgrad \vect u_{\chi,m_1, m_2}^{\rm bend} (\widehat{x},y)+ \Lambda_{\chi,m_1,m_2}^{\rm bend}(\widehat{x}) \right) :\overline{\Lambda_{\chi,d_1,d_2}^{\rm bend}(\widehat{x})} d\widehat{x}dy,\\[0.3em]
&\hspace{7em}\quad (m_3, m_4)^\top, (d_3, d_4)^\top\in{\mathbb C}^2,
\end{align*}
where ${\mathbb A}^{\rm bend}$, ${\mathbb A}^{\rm stretch}$ are defined by \eqref{gotovo2}. \CCC Also, under Assumption \ref{matsym}   the following decomposition holds \CCC (cf. Section \ref{hom_op_sec}) :
\begin{align*}
{\mathbb A}_\chi^{\rm rod} {\vect m}\cdot \overline{\vect d} = \chi^4{\mathbb A}^{\rm bend}  (m_1,m_2)^{\top}\cdot \overline{(d_1,d_2)^{\top}}&+\chi^2{\mathbb A}^{\rm stretch}  (m_3,m_4)^{\top}\cdot \overline{(d_3,d_4)^{\top}},\\[0.3em]
\quad &\forall {\vect m}=(m_1, m_2, m_3, m_4)^\top,\ \ {\vect d}=(d_1, d_2, d_3, d_4)^\top\in{\mathbb C}^4.
\end{align*} 
\CCC The following statement is a direct consequence of Proposition \ref{proposition21}, see also Corollary \ref{coerc_corrol}.
\begin{proposition}
For the constant $\eta$ of Proposition \ref{proposition21} one has, for all ${\vect m}= (m_1,m_2,m_3,m_4)^\top \in \C^4,$
\begin{equation}
    \begin{aligned}
       & \eta \left({\chi^4}\bigl|(m_1,m_2)^\top\bigr|^2 + {\chi^2} \bigl|(m_3,m_4)^\top\bigr|^2\right) \leq  {\mathbb A}_\chi^{\rm rod}{\vect m}\cdot \overline{\vect m} \leq \eta^{-1}\left({\chi^4}\bigl|(m_1,m_2)^\top\bigr|^2 + {\chi^2} \bigl|(m_3,m_4)^\top\bigr|^2\right), \\[0.3em]
       &
        \eta\bigl|(m_3,m_4)^\top\bigr|^2 \leq {\mathbb A}^{\rm stretch}  (m_3,m_4)\cdot \overline{(m_3,m_4)^{\top}}
       \leq \eta^{-1}\bigl|(m_3,m_4)^\top\bigr|^2,\\[0.3em] 
       &
       \eta\bigl|(m_1,m_2)^\top\bigr|^2\leq {\mathbb A}^{\rm bend}  (m_1,m_2)\cdot \overline{(m_1,m_2)^{\top}} \leq \eta^{-1} \bigl|(m_1,m_2)^\top\bigr|^2.
\end{aligned}
\end{equation}
\end{proposition}

Note that by applying the scaled Gelfand transform to the homogenised operators we get the following formulae:
\begin{equation}
    \mathfrak{G}_\varepsilon\mathcal{A}^{\rm bend} \Xi_\varepsilon\mathfrak{G}_\varepsilon^{-1} \vect u = \frac{\chi^4}{\varepsilon^4}{\mathbb A}^{\rm bend} \int_Y \vect u,\qquad
    \mathfrak{G}_\varepsilon\mathcal{A}^{\rm stretch} \Xi_\varepsilon\mathfrak{G}_\varepsilon^{-1} \vect u = \frac{\chi^2}{\varepsilon^2}{\mathbb A}^{\rm stretch} \int_Y \vect u.
\end{equation}


In what follows, we fist establish the asymptotics of the solutions in the stretching invariant subspace (Section \ref{stretching_sec}), as the related analysis is more straightforward. We then carry out the analysis in the bending invariant subspace (Section \ref{bending_sec}). 

Henceforth, we shall normally omit the independent variables of functions under integrals and in differential equations. In each instance these can be easily recovered from the context and the definitions of the auxiliary objects involved (such as those defined in Section \ref{aux_obj_sec}).  

\subsection{Asymptotics in the ``stretching" space}
\label{stretching_sec}


Here we provide estimates for the error in approximating the solution to the following resolvent equation posed in the space of stretching deformations: 
find $\vect u\in H^1_{\#}(Y;H^1(\omega;\C^3))\cap L^2_{\rm stretch}$ such that
\begin{equation}
\label{stretchingproblem}
\begin{aligned}
\frac{1}{{\chi^2}}\int_{\omega \times Y} \mathbb{A}  (\simgrad \vect u + {\rm i}X_\chi \vect u):\overline{(\simgrad \vect v+ {\rm i}X_\chi \vect v)}&+ \int_{\omega \times Y} \vect u\cdot\overline{\vect v}\\[0.35em]
&= \int_{\omega \times Y} \vect f\cdot\overline{\vect v}\qquad \forall\vect v\in H^1_{\#}\bigl(Y;H^1(\omega;\C^3)\bigr)\cap L^2_{\rm stretch}.
\end{aligned}
\end{equation}
Using the adjoint of the expression $\simgrad + {\rm i}X_\chi$, the identity \eqref{stretchingproblem} can be formally written as follows:
\begin{equation}
\chi^{-2}\left((\simgrad)^* +  \left({\rm i}X_\chi\right)^*\right){\mathbb A}\left(\simgrad + {\rm i}X_\chi\right) \vect u+ \vect u = \vect f.
\end{equation}
%
As discussed in the next remark, the problem \eqref{stretchingproblem} is well posed. 

\begin{remark}
We test the above equation with the solution $\vect u$ and employ Proposition \ref{aprioriestimates} to obtain
\begin{align*}
    \chi^{-2}\bigl\Vert \left(\simgrad + {\rm i}X_\chi \right) \vect u \bigr\Vert_{L^2(\omega\times Y;\C^{3\times 3})}^2 + \Vert \vect u\Vert _{L^2(\omega\times Y, \C^3)}^2  &\leq  C\bigl\Vert  \vect f \bigr\Vert _{L^2(\omega\times Y, \C^3)}\Vert \vect u\Vert _{L^2(\omega\times Y, \C^3)} \\[0.35em]
& \leq C|\chi|^{-1}\bigl\Vert  \vect f \bigr\Vert _{L^2(\omega\times Y, \C^3)}\bigl\Vert \left(\simgrad + {\rm i}X_\chi \right)\vect u \bigr\Vert_{L^2(\omega\times Y;\C^{3\times 3})},
\end{align*}
so that
\begin{equation}
\bigl\Vert\left(\simgrad + {\rm i}X_\chi \right) \vect u \bigr\Vert_{L^2(\omega\times Y;\C^{3\times 3})} \leq C|\chi| \bigl\Vert  \vect f \bigr\Vert _{L^2(\omega\times Y, \C^3)}.
\end{equation}
Again, by Proposition \ref{aprioriestimates} we deduce the following apriori estimate:
\begin{equation}
\label{aprioriestimatestetchingh1}
\begin{split}
\Vert \vect u\Vert _{H^1(\omega\times Y, \C^3)} \leq C\bigl\Vert  \vect f \bigr\Vert _{L^2(\omega\times Y, \C^3)} 
\end{split}
\end{equation}
\end{remark}

 The main result of this section is Proposition \ref{propbending}. Its proof follows by performing an asymptotic expansion in $\chi,$ which we will now present.

We want to construct a function that approximates the solution to \eqref{stretchingproblem}, up to an error of order ${\chi^2}$. This is achieved in several steps, by introducing appropriate corrector terms. The leading-order term of the approximation is given by a suitable infinitesimal stretching rigid motion, as described below. 
\subsubsection*{1) Leading-order term and first-order corrector}
Consider the solution $(m_3,m_4)^\top$ to the equation
\begin{equation}
\label{stretchlimitequation}
\bigl({\mathbb A}^{\rm stretch}  + \mathfrak{C}^{\rm stretch} \bigr) 
\begin{bmatrix}
    m_3 \\ m_4
\end{bmatrix}  = \widetilde{\mathcal{M}}^{\rm stretch} \vect f,
\end{equation}
 where $\mathfrak{C}^{\rm stretch}$ is given by \eqref{cstretchrodbend1}. 
By testing this equation against $(m_3,m_4)^\top$  and using the fact that $\mathfrak{C}^{\rm stretch}$ is positive-definite and  ${\mathbb A}_\chi^{\rm stretch}$ in non-negative,  it is clear that 
\begin{equation}
\bigl|(m_3,m_4)^\top\bigr| \leq C \bigl\Vert  \vect f \bigr\Vert _{L^2(\omega \times Y;\C^3)},
\end{equation}
with $C>0$ independent of $\chi$ and ${\vect f}.$ Consider the function $\vect u_0:= \widetilde{\mathcal{I}}^{\rm stretch} \begin{bmatrix}
    m_3 \\ m_4
\end{bmatrix}$ and note that it can also be bounded, as follows:
\begin{equation}
\Vert \vect u_0\Vert _{H^1(\omega\times Y ; \C^3)} \leq C\bigl\Vert  \vect f \bigr\Vert _{L^2(\omega\times Y;\C^3)}.
\end{equation}  
Furthermore, define a "first-order corrector" ${\vect u_1} \in H^{\rm stretch}$ as the solution to 
\begin{equation}
\label{stretchfirstcorrectoreqation}
(\simgrad)^* {\mathbb A}\simgrad{\vect u_1} = -(\simgrad)^* {\mathbb A} \Lambda_{\chi,m_3,m_4}^{\rm stretch}.
\end{equation}
 Notice that $\vect u_1$ was denoted by $\vect u^{\rm stretch}_{\chi,m_3,m_4} $ in \eqref{nak61} and was used to define $\mathbb{A}_{\chi}^{\rm stretch}$.  
Also, it is clear from the definition of ${\mathbb A}_\chi^{\rm stretch} $  and \eqref{stretchlimitequation}  that 
\begin{equation}
\label{firstorderterm}
\begin{aligned}
\frac{1}{{\chi^2}}\int_{\omega \times Y} {\mathbb A}\left( \simgrad {\vect u_1} + \Lambda_{\chi,m_3,m_4}^{\rm stretch} \right) :\overline{\Lambda_{\chi,d_3,d_4}^{\rm stretch}} 
&+ \int_{\omega\times Y} 
 \begin{bmatrix}
    m_3 x_2   \\
    -m_3 x_1 \\
     m_4
\end{bmatrix} \cdot 
\overline{\begin{bmatrix}
    d_3 x_2   \\
    -d_3 x_1 \\
     d_4
\end{bmatrix} } \\[0.6em]
&=\int_{\omega\times Y} 
{\vect f}\cdot 
\overline{\begin{bmatrix}
    d_3 x_2   \\
    -d_3 x_1 \\
     d_4
\end{bmatrix}}\qquad \forall (d_3, d_4)^\top\in{\mathbb R}^2.
\end{aligned}
\end{equation}
The corrector term ${\vect u_1}$ belongs to the space $L^2_{\rm stretch}$ due to the structure of the elasticity tensor ${\mathbb A}$. Finally,  by testing \eqref{stretchfirstcorrectoreqation} with $\vect u_1$ we obtain   the existence of $C>0$ such that for all ${\vect f}\in L^2(\omega\times Y;\C^3),$ $\chi\in[-\pi,\pi)\setminus\{0\}$ one has
\begin{equation} \label{nakn1}
\Vert {\vect u_1}\Vert _{H^1(\omega\times Y ; \C^3)} \leq C|\chi|\bigl\Vert  \vect f \bigr\Vert _{L^2(\omega\times Y;\C^3)}.
\end{equation}

\subsubsection*{2) Second-order corrector}

Define a functional $\widetilde{\vect f}_1:=L^2(\omega\times Y;\C^3)\to\C$ by the formula
	\begin{equation}
		\widetilde{\vect f}_1 = -\chi^{-2}\left({\rm i}X_\chi\right)^* {\mathbb A}\left(\simgrad {\vect u_1} + \Lambda_{\chi,m_3,m_4}^{\rm stretch} \right) -\vect u_0 + \vect f.
	\end{equation}
	It  follows from \eqref{firstorderterm}  that $\widetilde{\vect f}_1$ vanishes when tested  against  infinitesimal stretching rigid-body motions, which follows directly from \eqref{firstorderterm}.
In view of the equation \eqref{stretchingproblem}, we define a second-order corrector term ${\vect u_2}\in H^{\rm stretch}$ as the solution to the equation
\begin{equation}
\label{stretchsecondcorrectorequation}
\begin{aligned}
\chi^{-2}(\simgrad)^* {\mathbb A}\simgrad {\vect u_2}
&=-\chi^{-2}\left(\left({\rm i}X_\chi\right)^*{\mathbb A}\simgrad {\vect u_1} + (\simgrad)^* {\mathbb A}{\rm i}X_{\chi}{\vect u_1} + \left({\rm i}X_\chi\right)^*{\mathbb A} \Lambda_{\chi,m_3,m_4}^{\rm stretch}\right)
 - \vect u_0 +\vect f\\[0.3em]
&= \widetilde{\vect f}_1-\chi^{-2}(\simgrad)^* {\mathbb A}{\rm i}X_{\chi}{\vect u_1},
\end{aligned}
\end{equation}
whose right-hand side clearly vanishes when tested against functions in $H^{\rm stretch}$. Therefore, \eqref{stretchsecondcorrectorequation} is well posed, and
\begin{equation}\label{nakn2} 
\Vert {\vect u_2}\Vert _{H^1(\omega\times Y ; \C^3)} \leq C{\chi^2}\bigl\Vert  \vect f \bigr\Vert _{L^2(\omega\times Y;\C^3)}.
\end{equation}
\begin{remark}
   The problem \eqref{stretchlimitequation}
   is well motivated by the following rationale. Thinking about $\vect u_0$ as the leading-order term in a perturbation expansion for the solution $\vect u$ to \eqref{stretchingproblem}, one is lead to the equation  
   \begin{equation}
       (\simgrad)^*{\mathbb A}\simgrad \vect u_0 = 0.
   \end{equation}
 In other words, $\vect u_0$ is an infinitesimal 
$y$-periodic rigid motion, hence $\vect u_0 = (m_1 + m_3 x_2,m_2 - m_3 x_1, m_4)^\top.$  Next, since $\vect u_0\in  L_{\rm stretch}^2(\omega \times Y;\C^3)$, we infer  that $\vect u_0 = (m_3 x_2, - m_3 x_1, m_4)^\top$. The equation \eqref{firstorderterm}  
 could be inferred as a consequence of well-posedness of \eqref{stretchsecondcorrectorequation} and thus does not need to be guessed apriori. 
\end{remark}

The total approximation constructed so far, namely $\vect u_{\rm approx}:= \vect u_0 + {\vect u_1} + {\vect u_2}$ satisfies the equation
\begin{equation}
\chi^{-2}\left((\simgrad)^* + \left({\rm i}X_\chi\right)^*\right){\mathbb A}\left(\simgrad + {\rm i}X_\chi\right) \vect u_{\rm approx} + \vect u_{\rm approx} - \vect f = {\vect R}_{\chi},
\end{equation}
where
\begin{equation}
{\vect R}_{\chi}:=\chi^{-2}\left(\left({\rm i}X_\chi\right)^*{\mathbb A}\simgrad {\vect u_2} + (\simgrad)^* {\mathbb A}{\rm i}X_{\chi}{\vect u_2} + \left({\rm i}X_\chi\right)^* {\mathbb A}(y){\rm i}X_{\chi} {\vect u_1} + \left({\rm i}X_\chi\right)^* {\mathbb A}{\rm i}X_{\chi}{\vect u_2} \right) +{\vect u_1} +{\vect u_2}.  
\end{equation}
The remainder ${\vect R}_{\chi}$ satisfies  the bound 
\begin{equation}
\Vert {\vect R}_{\chi}\Vert _{[H^1_\#(Y;H^1(\omega;\C^3))]^*} \leq C |\chi| \bigl\Vert  \vect f \bigr\Vert _{L^2(\omega\times Y;\C^3)},
\label{Rchiest}
\end{equation}
where $[H^1_\#(Y;H^1(\omega;\C^3))]^*$ stands for the dual of $H^1_\#(Y;H^1(\omega;\C^3)).$

As far as the proof of Theorem \ref{THML2L2} is concerned, the bound \eqref{Rchiest}, combined with a standard ellipticity argument (shown below for the case of a refined approximation $\widetilde{\vect u}_{\rm approx},$ see \eqref{u_approx_tilde}),
provides the required resolvent estimate for ${\mathcal A}^\varepsilon,$ see Section \ref{L2toL2} for details.
However, for higher precision norm-resolvent estimates, one needs to construct suitable ``correctors" to ${\vect u}_1,$ ${\vect u}_2,$ ${\vect u}_3.$ Indeed, on the one hand, as discussed above in the context of  constructing  ${\vect u}_0,$ and consequently ${\vect u}_1$ and ${\vect u}_2,$ we have used up all the degrees of freedom available. In particular, choosing the terms ${\vect u}_1$ and ${\vect u}_2$ to be elements of $H^{\rm stretch}$ is too restrictive. On the other hand, the problem (${\vect u_3} \in H^{\rm stretch}$)
\begin{equation*}
\chi^{-2}(\simgrad)^* {\mathbb A} \simgrad {\vect u_3} = 
 -\chi^{-2}\left(\left({\rm i}X_\chi\right)^*{\mathbb A}\simgrad {\vect u_2} + (\simgrad)^* {\mathbb A}{\rm i}X_{\chi}{\vect u_2} + \left({\rm i}X_\chi\right)^* {\mathbb A} {\rm i}X_{\chi}{\vect u_1} \right) - {\vect u_1}
\end{equation*}
is not well posed. In what follows we construct the mentioned correctors.

\subsubsection*{3) Refining the approximation}
We proceed with the correction of the leading-order term as follows. Set 
\[
\vect u_0^{(1)}:= \widetilde{\mathcal{I}}^{\rm stretch}\begin{bmatrix}
    ({\vect m}^{(1)})_3 \\[0.3em] 
    ({\vect m}^{(1)})_4
\end{bmatrix}, 
\]
where the vector
    $\bigl(({\vect m}^{(1)})_3, ({\vect m}^{(1)})_4\bigr)^\top$
is the solution to 
\begin{equation}
	\label{u01_stretch}
\begin{aligned}
\left(\frac{1}{{\chi^2}}{\mathbb A}_\chi^{\rm stretch}  + \mathfrak{C}^{\rm stretch}  \right) \begin{bmatrix}
    ({\vect m}^{(1)})_3 \\[0.3em] 
    ({\vect m}^{(1)})_4
\end{bmatrix}\cdot \overline{\begin{bmatrix}
    d_3 \\ d_4
\end{bmatrix}}&= -\frac{1}{{\chi^2}}\int_{\omega \times Y} {\mathbb A}\bigl( \simgrad {\vect u_2} + {\rm i}X_{\chi} {\vect u_1} \bigr) :\overline{\Lambda_{\chi,d_3,d_4}^{\rm stretch}}\qquad
\\[0.3em] 
&\hspace{7em}
\forall (d_3,d_4)^\top\in \C^2.
\end{aligned}
\end{equation}
It is easy to see that $\vect u_0^{(1)}$ defined in this way satisfies the estimate
\begin{equation}\label{u0^1_stretch}
\bigl\Vert \vect u_0^{(1)}\bigr\Vert _{H^1(\omega\times Y;\C^3)} \leq C|\chi|\bigl\Vert  \vect f \bigr\Vert _{L^2(\omega\times Y;\C^3)}.     
\end{equation}
The next corrector ${\vect u_1^{(1)}}$  is set to solve 
\begin{equation}
(\simgrad)^*{\mathbb A} \simgrad {\vect u_1^{(1)}} = -(\simgrad)^* {\mathbb A} \Lambda_{\chi,({\vect m}^{(1)})_3,({\vect m}^{(1)})_4}^{\rm stretch},
\label{u11_stretch}
\end{equation}
and hence  it  satisfies the bound
\begin{equation}\label{nakn3} 
\bigl\Vert {\vect u_1^{(1)}} \bigr\Vert _{H^1(\omega\times Y;\C^3)} \leq C{\chi^2}\bigl\Vert  \vect f \bigr\Vert _{L^2(\omega\times Y;\C^3)},
\end{equation}
with $C>0$ independent of $\vect f.$ 

It is the result of a straightforward calculation that \eqref{u01_stretch} and \eqref{u11_stretch} imply
\begin{align*}
&\frac{1}{{\chi^2}}\int_{\omega \times Y} {\mathbb A}\left( \simgrad {\vect u_1^{(1)}} + \Lambda_{\chi,({\vect m}^{(1)})_3,({\vect m}^{(1)})_4}^{\rm stretch}\right) :\overline{\Lambda_{\chi,d_3,d_4}^{\rm stretch}} + \int_{\omega\times Y} 
 \begin{bmatrix}
    ({\vect m}^{(1)})_3 x_2   \\[0.3em]
    -({\vect m}^{(1)})_3 x_1 \\[0.3em]
     ({\vect m}^{(1)})_4
\end{bmatrix} \cdot 
\overline{\begin{bmatrix}
    d_3 x_2   \\[0.2em]
    -d_3 x_1 \\[0.2em]
     d_4
\end{bmatrix}} \\[0.6em]
&\quad=-\frac{1}{{\chi^2}}\int_{\omega \times Y} {\mathbb A}\left( \simgrad {\vect u_2} + {\rm i}X_{\chi}{\vect u_1} \right) :\overline{\Lambda_{\chi,d_3,d_4}^{\rm stretch}} \qquad \forall (d_3,d_4)^\top\in \C^2,
\end{align*}
and therefore the functional  $\widetilde{\vect f}_2:L^2(\omega\times Y;\C^3)\to \C$ defined by
\begin{eqnarray*}
& &\widetilde{\vect f}_2 := -\chi^{-2}\Bigl(\left({\rm i}X_\chi\right)^*{\mathbb A}\simgrad ({\vect u_2} + {\vect u_1^{(1)}} ) + \left({\rm i}X_\chi\right)^* {\mathbb A} \Lambda_{\chi,({\vect m}^{(1)})_3,({\vect m}^{(1)})_4}^{\rm stretch} + \left({\rm i}X_\chi\right)^*{\mathbb A}{\rm i}X_{\chi}{\vect u_1}\Bigr)-\vect u_0^{(1)} -{\vect u_1}
\end{eqnarray*}
vanishes on infinitesimal stretching rigid-body motions. This allows us to pose the following problem for ${\vect u_2^{(1)}} \in H^{\rm stretch}:$
\begin{align*}
\chi^{-2}(\simgrad)^* {\mathbb A} \simgrad {\vect u_2^{(1)}} &= \widetilde{\vect f}_2 
 -\chi^{-2}(\simgrad)^* {\mathbb A}{\rm i}X_{\chi}\bigl({\vect u_2} + {\vect u_1^{(1)}}\bigr),
\end{align*}
whose solution satisfies the bound
\begin{equation}\label{nakn4} 
\bigl\Vert{\vect u_2^{(1)}} \bigr\Vert _{H^1(\omega\times Y;\C^3)} \leq C|\chi|^3 \bigl\Vert  \vect f \bigr\Vert _{L^2(\omega\times Y;\C^3)}.
\end{equation}
Having defined the three correctors above, we write 
\begin{equation}
\widetilde{\vect u}_{\rm approx}:=\vect u_0 + \vect u_0^{(1)} + {\vect u_1} + {\vect u_1^{(1)}} +{\vect u_2} + {\vect u_2^{(1)}}.
\label{u_approx_tilde}
\end{equation}
Note that
\begin{equation*}
\chi^{-2}\left((\simgrad)^* + \left({\rm i}X_\chi\right)^*\right){\mathbb A}\left(\simgrad + {\rm i}X_\chi\right) \widetilde{\vect u}_{\rm approx} + \widetilde{\vect u}_{\rm approx} - \vect f = \widetilde{\vect R}_{\chi},
\end{equation*}
where
\begin{align*}
\widetilde{\vect R}_{\chi}&:=\chi^{-2}\Bigl(\left({\rm i}X_\chi\right)^*{\mathbb A}\simgrad{\vect u_2^{(1)}} + (\simgrad)^* {\mathbb A}{\rm i}X_{\chi}{\vect u_2^{(1)}} + \left({\rm i}X_\chi\right)^* {\mathbb A}{\rm i}X_{\chi} {\vect u_2^{(1)}} + \left({\rm i}X_\chi\right)^* {\mathbb A}{\rm i}X_{\chi}{\vect u_2} \Bigr)  \\[0.3em]
&\quad+\chi^{-2}\left(\left({\rm i}X_\chi\right)^* {\mathbb A}{\rm i}X_{\chi}{\vect u_2} + \left({\rm i}X_\chi\right)^* {\mathbb A}{\rm i}X_{\chi} {\vect u_1^{(1)}} \right) + {\vect u_1^{(1)}} +{\vect u_2} +{\vect u_2^{(1)}} 
\end{align*}
satisfies the estimate 
\begin{equation}
\bigl\Vert \widetilde{\vect R}_{\chi}\bigr\Vert _{[H^1_\#(Y;H^1(\omega;\C^3))]^*} \leq C {\chi^2} \bigl\Vert  \vect f \bigr\Vert _{L^2(\omega\times Y;\C^3)}.
\label{refined_est}
\end{equation}
Furthermore, define the approximation error 
$$\vect u_{\rm error}:={\vect u} - \widetilde{\vect u}_{\rm approx}. $$
The function $\vect u_{\rm error}$ is the solution to the problem
\begin{equation}
\chi^{-2}\left((\simgrad)^* + \left({\rm i}X_\chi\right)^*\right){\mathbb A}\left(\simgrad + {\rm i}X_\chi\right) \vect u_{\rm error} + \vect u_{\rm error} = -\widetilde{\vect R}_{\chi}.
\label{uerror} 
\end{equation}
Testing \eqref{uerror} with $\vect u_{\rm error}$, we obtain 
\begin{equation}
\begin{split}
        \bigl\Vert \left(\simgrad + {\rm i}X_\chi \right)\vect u_{\rm error} \bigr\Vert _{L^2(\omega \times Y;\C^{3\times 3})}^2 & + \bigl\Vert  \vect u_{\rm error} \bigr\Vert _{L^2(\omega\times Y;\C^3)}^2 \\[0.3em]
        & \leq \chi^{-2}\bigl\Vert \left(\simgrad + {\rm i}X_\chi \right)\vect u_{\rm error} \bigr\Vert _{L^2(\omega \times Y;\C^{3\times 3})}^2 + \bigl\Vert  \vect u_{\rm error} \bigr\Vert _{L^2(\omega\times Y;\C^3)}^2 \\[0.3em]
        & \leq C\bigl\Vert \widetilde{\vect R}_{\chi}\bigr\Vert _{[H^1_\#(Y;H^1(\omega;\C^3))]^*} \bigl\Vert  \vect u_{\rm error} \bigr\Vert _{H^1(\omega\times Y;\C^3)}.
\end{split}
\end{equation}
On the other hand,
\begin{equation}
    \bigl\Vert  \simgrad \vect u_{\rm error} \bigr\Vert _{L^2(\omega \times Y;\C^{3\times 3})}^2 \leq C \left(\bigl\Vert\left(\simgrad + {\rm i}X_\chi \right)\vect u_{\rm error}\bigr\Vert_{L^2(\omega \times Y;\C^{3\times 3})}^2 + {\chi^2}\bigl\Vert  \vect u_{\rm error} \bigr\Vert _{L^2(\omega\times Y;\C^3)}^2  \right),
\end{equation}
which yields
\begin{equation}
    \bigl\Vert  \simgrad \vect u_{\rm error} \bigr\Vert _{L^2(\omega \times Y;\C^{3\times 3})}^2 - C{\chi^2}\bigl\Vert  \vect u_{\rm error} \bigr\Vert _{L^2(\omega\times Y;\C^3)}^2 \leq C \left(\bigl\Vert\left(\simgrad + {\rm i}X_\chi \right)\vect u_{\rm error}\bigr\Vert_{L^2(\omega \times Y;\C^{3\times 3})}^2\right).
\end{equation}
Therefore, one has
\begin{equation*}
        \bigl\Vert \simgrad\vect u_{\rm error} \bigr\Vert _{L^2(\omega \times Y;\C^{3\times 3})}^2 + \left( 1 - C{\chi^2}\right)\bigl\Vert  \vect u_{\rm error} \bigr\Vert _{L^2(\omega\times Y;\C^3)}^2 
        \leq C\bigl\Vert \widetilde{\vect R}_{\chi}\bigr\Vert _{[H^1_\#(Y;H^1(\omega;\C^3))]^*} \bigl\Vert  \vect u_{\rm error} \bigr\Vert _{H^1(\omega\times Y;\C^3)}.
\end{equation*}
By combining this with the Korn's inequality
\begin{equation}
    \bigl\Vert  \vect u_{\rm error} \bigr\Vert _{H^1(\omega\times Y;\C^3)}^2 \leq C \left(\bigl\Vert \simgrad\vect u_{\rm error} \bigr\Vert _{L^2(\omega \times Y;\C^{3\times 3})}^2  + \bigl\Vert  \vect u_{\rm error} \bigr\Vert _{L^2(\omega\times Y;\C^3)}^2 \right),
\end{equation}
it is clear that for $|\chi|<\eta$, where $\eta$ is a fixed small constant, one has
\begin{equation}
    \label{aprioriestimateh-1}
    \bigl\Vert  \vect u_{\rm error} \bigr\Vert _{H^1(\omega\times Y;\C^3)} \leq C\bigl\Vert \widetilde{\vect R}_{\chi}\bigr\Vert _{[H^1_\#(Y;H^1(\omega;\C^3))]^*}. 
\end{equation}
where the constant $C>0$ depends only on $\eta>0$.

Finally, combining \eqref{aprioriestimateh-1} with \eqref{refined_est}, we infer that
\begin{equation}\label{nakn5} 
\bigl\Vert \vect u_{\rm error}\bigr\Vert _{H^1(\omega\times Y;\C^3)} \leq C{\chi^2} \bigl\Vert  \vect f \bigr\Vert _{L^2(\omega\times Y;\C^3)}.
\end{equation}
In particular,  by Steps 1)-5) we have proved the following statement.
\begin{proposition}\label{propbending} 
Let $\vect u\in H^1_{\#}(Y;H^1(\omega;\C^3))$ be the solution of \eqref{stretchingproblem}. Then there exists $C>0$ such that for all $\chi\in [-\pi, \pi)\setminus\{0\}$ and $\vect f\in L^2(\omega\times Y;\C^3)$ one has
\begin{equation}
\begin{split}\label{stretch_estimates}
 \Vert \vect u -{\vect u}_0 
 \Vert_{H^1(\omega\times Y,\C^3)} \leq C|\chi|\bigl\Vert  \vect f \bigr\Vert _{L^2(\omega\times Y;\C^3)}, \\
 \bigl\Vert \vect u - {\vect u}_0-{\vect u}_0^{(1)}
- {\vect u_1}\bigr\Vert_{H^1(\omega\times Y,\C^3)} \leq C{\chi^2}\bigl\Vert  \vect f \bigr\Vert _{L^2(\omega\times Y;\C^3)}, 
\end{split}
\end{equation}
where ${\vect u}_0,$ ${\vect u}_0^{(1)},$ ${\vect u}_1$ 
are defined by the above approximation procedure.
\end{proposition}

\begin{proof} 
The proof follows from \eqref{nakn5} by using \eqref{nakn1}, \eqref{nakn2}, \eqref{u0^1_stretch}, \eqref{nakn3}, and \eqref{nakn4}. 
\end{proof} 	

\begin{remark}
   The first estimate in \eqref{stretch_estimates} can be rewritten as
   \begin{equation}
   \label{stretchingestimate1}
   \left\Vert \left(\frac{1}{{\chi^2}}\mathcal{A}_\chi + I \right)^{-1} \bigg|_{L^2_{\rm stretch}} - \left(\widetilde{\mathcal{M}}^{\rm stretch}\right)^*\left(\frac{1}{{\chi^2}}{\mathbb A}_\chi^{\rm stretch } + \mathfrak{C}^{\rm stretch} \right)^{-1}\widetilde{\mathcal{M}}^{\rm stretch}  \right\Vert_{L^2 \to H^1} \leq C|\chi|, 
   \end{equation}
while the second estimate takes the form
   \begin{equation}
   \label{stretchingestimate2}
       \left\Vert \left(\frac{1}{{\chi^2}}\mathcal{A}_\chi + I \right)^{-1} \bigg|_{L^2_{\rm stretch}} - \left(\widetilde{\mathcal{M}}^{\rm stretch}\right)^*\left(\frac{1}{{\chi^2}}{\mathbb A}_\chi^{\rm stretch } + \mathfrak{C}^{\rm stretch} \right)^{-1}\widetilde{\mathcal{M}}^{\rm stretch} - \mathcal{A}_{\chi,\rm corr}^{\rm stretch} - \widetilde{\mathcal{A}}_{\chi,\rm corr}^{\rm stretch} \right\Vert_{L^2 \to H^1} \leq C{\chi^2}, 
   \end{equation}
   where the bounded operators $\mathcal{A}_{\chi,\rm corr}^{\rm stretch}$ and $\widetilde{\mathcal{A}}_{\chi,\rm corr}^{\rm stretch}$ are defined by 
   \begin{equation}
       \mathcal{A}_{\chi,\rm corr}^{\rm stretch} \vect f := {\vect u_1}, \qquad \widetilde{\mathcal{A}}_{\chi,\rm corr}^{\rm stretch} \vect f := \vect u_0^{(1)}.
   \end{equation}

\end{remark}

\subsection{Asymptotics in the ``bending" space}
\label{bending_sec}
  In order to obtain an approximation for scaled loads  for $\chi\in[-\pi,\pi)\setminus\{0\},$ we consider the scaled force density
 \[
S_{\!|\chi|}\vect f:=\bigl(\,\widehat{\!\vect f},
|\chi|^{-1}{f_3}\bigr)^\top.
\] 
For each $\chi\in[-\pi,\pi)\setminus\{0\},$ we study the following problem: 
for $\vect f\in L^2_{\rm bend},$ find $\vect u\in H^1_{\#}(Y;H^1(\omega;\C^3))$ such that
\begin{equation}
\label{quasiperiodicbend}
\frac{1}{{\chi^4}}\int_{\omega \times Y} {\mathbb A} (\simgrad \vect u + {\rm i}X_\chi \vect u) :\overline{(\simgrad \vect v + {\rm i}X_\chi \vect v)} + \int_{\omega \times Y} \vect u\,\cdot\,\overline{\vect v} = \int_{\omega \times Y} S_{\!|\chi|}\vect f\,\cdot\,\overline{\vect v}\qquad \forall\vect v\in H^1_{\#}(Y;H^1(\omega;\C^3)).  
\end{equation} 
which can be formally written as
\begin{equation}
	\label{bendingproblem}
\chi^{-4}\left((\simgrad)^* + \left({\rm i}X_\chi\right)^*\right){\mathbb A}\left(\simgrad + {\rm i}X_\chi\right) \vect u + \vect u = S_{\!|\chi|}\vect f.
\end{equation}
As we are working under the assumption of material symmetries, one has $\vect u\in L^2_{\rm bend}.$ 

As in the case of the stretching subspace,  the above problem is well posed, see the remark below. Note that related apriori bound is used in deriving the convergence estimate \eqref{noforcescaling}.
\begin{remark}
	\label{bending_apriori_rem}
	Testing \eqref{bendingproblem} against the solution $\vect u$ and applying Proposition \ref{aprioriestimates} yields
	\begin{align*}
		& \chi^{-4}\bigl\Vert \left(\simgrad + {\rm i}X_\chi \right)\vect u\bigr\Vert _{L^2(\omega\times Y, \C^3)}^2 + \Vert \vect u\Vert _{L^2(\omega\times Y, \C^3)}^2  \\[0.4em] 
		&\leq   C\left(\Vert {f_1}\Vert _{L^2(\omega\times Y, \C)}\Vert u_1\Vert _{L^2(\omega\times Y, \C)}+ \Vert {f_2}\Vert _{L^2(\omega\times Y, \C)}\Vert u_2\Vert _{L^2(\omega\times Y, \C)} +|\chi|^{-1}\Vert {f_3}\Vert _{L^2(\omega\times Y, \C)}\Vert {u_3}\Vert _{L^2(\omega\times Y, \C)}\right) \\[0.4em]
		&\leq C\chi^{-2}\bigl\Vert  \vect f \bigr\Vert _{L^2(\omega\times Y, \C^3)}\bigl\Vert \left(\simgrad + {\rm i}X_\chi \right) \vect u \bigr\Vert_{L^2(\omega\times Y;\C^{3\times 3})}. 
	\end{align*}
	Therefore, one has 
	\begin{equation}
		\bigl\Vert \left(\simgrad + {\rm i}X_\chi \right) \vect u \bigr\Vert_{L^2(\omega\times Y;\C^{3\times 3})} \leq C{\chi^2} \bigl\Vert  \vect f \bigr\Vert _{L^2(\omega\times Y, \C^3)}.
	\end{equation}
	Finally, again by Proposition \ref{aprioriestimates}, we have
	\begin{equation}
		\label{aprioriestimatebendingh1}
		\begin{split}\Vert {\widehat{\vect u}}\Vert _{H^1(\omega\times Y, \C)} \leq C\bigl\Vert  \vect f \bigr\Vert _{L^2(\omega\times Y, \C^3)},
			\quad \Vert {u_3}\Vert _{H^1(\omega\times Y, \C)} \leq C|\chi|\bigl\Vert  \vect f \bigr\Vert _{L^2(\omega\times Y, \C^3)}. 
		\end{split}
	\end{equation}
\end{remark}

We would like to construct a function in $L^2_{\rm bend}$ describing the leading-order behaviour of ${\vect u}$ with respect to the parameter $\chi$ with suitable correctors added, so that the overall error is of order  $O({\chi^4})$ in the $H^1$ norm. In doing so, we shall repeat the approach of Section \ref{stretching_sec}.
The mentioned leading-order term is given with the general expression for an infinitesimal bending rigid motion.  The main result of this section is Proposition \ref{propbending}. The asymptotics here is more complex than the one performed in Section \ref{stretching_sec}.

\subsubsection*{1) Leading-order term and second-order corrector}
The leading-order term is defined as the solution to the following problem: find $(m_1,m_2)^\top\in\C^2$ that satisfies
\begin{equation}
    \bigl({\mathbb A}^{\rm bend}+\mathfrak{C}^{\rm bend}_{\chi}\bigr)
    \begin{bmatrix} m_1 \\ m_2 \end{bmatrix}  = 
    \widetilde{\mathcal{M}}_{\chi}^{\rm bend} S_{\!|\chi|} \vect f,
    \label{hom1}
\end{equation}
 where $\mathfrak{C}^{\rm bend}_{\chi}$ is defined in \eqref{cstretchrodbend}. 
 Similarly to the case of \eqref{stretchlimitequation},  the solution satisfies the estimate 
\begin{equation}
\bigl|(m_1,m_2)^\top\bigr| \leq C \bigl\Vert  \vect f \bigr\Vert _{L^2(\omega\times Y,\C^3)}.
\end{equation}
Setting $\vect u_0 = \widetilde{\mathcal{I}}_\chi^{\rm bend} \begin{bmatrix} m_1 \\ m_2 \end{bmatrix}$, it is then clear that there exists $C>0$ such that for all $\vect f\in L^2(\omega\times Y;\C^3)$ one has 
\begin{equation}
\bigl\Vert ({\vect u}_0)_\alpha\bigr\Vert _{H^1(\omega\times Y ; \C)} \leq C\bigl\Vert  \vect f \bigr\Vert _{L^2(\omega\times Y;\C^3)}, \quad \alpha=1,2,\qquad  \bigl\Vert ({\vect u}_0)_3\bigr\Vert _{H^1(\omega\times Y ; \C)} \leq C|\chi|\bigl\Vert  \vect f \bigr\Vert _{L^2(\omega\times Y;\C^3)}.
\end{equation}  
We define the first-order corrector \CCC ${\vect u_1} \in H^{\rm bend}$  as the solution to
\begin{equation}\label{bendfirstcorrectoreqation}
(\simgrad)^* {\mathbb A} \simgrad {\vect u_1} = -(\simgrad)^* {\mathbb A} \Lambda_{\chi,m_1,m_2}^{\rm bend}. 
\end{equation}
 Note that this corrector is already defined in \eqref{nak61}, where it was denoted by $\vect u^{\rm bend}_{\chi,m_1,m_2}$ and was used in the definition of $\mathbb{A}^{\rm bend}_{\chi}$.  
This problem  \eqref{bendfirstcorrectoreqation} is well posed, and one has
\begin{equation}\label{nak6} 
\Vert {\vect u_1}\Vert _{H^1(\omega\times Y ; \C^3)} \leq C{\chi^2}\bigl\Vert  \vect f \bigr\Vert _{L^2(\omega\times Y;\C^3)},
\end{equation}
with a constant $C>0$ independent of $\vect f.$ It is clear from the definition of ${\mathbb A}_\chi^{\rm bend}$ that
	\begin{equation}
\begin{aligned}
\frac{1}{{\chi^4}}\int_{\omega \times Y} {\mathbb A}\left( \simgrad {\vect u_1} + \Lambda_{\chi,m_1,m_2}^{\rm bend}\right)&:\overline{\Lambda_{\chi,d_1,d_2}^{\rm bend}} 
+\int_{\omega\times Y}
\begin{bmatrix}
    m_1 \\ m_2 \\-{\rm i}\chi(m_1 x_1 + m_2 x_2)
\end{bmatrix}
\cdot 
\overline{\begin{bmatrix}
    d_1 \\ d_2 \\-{\rm i}\chi(d_1 x_1 + d_2 x_2)
\end{bmatrix}}
\\[0.5em] 
&=\int_{\omega\times Y} 
\begin{bmatrix}
   \,\widehat{\!\vect f}
   \\[0.25em]   |\chi|^{-1}{f_3}
\end{bmatrix} \cdot\overline{\begin{bmatrix}
    d_1 \\ d_2 \\-{\rm i}\chi(d_1 x_1 + d_2 x_2)
\end{bmatrix}}\qquad\forall (d_1,d_2)^\top \in \C^2.
\end{aligned}
\label{first_step_bend}
\end{equation}

\subsubsection*{2) Higher-order correctors}
In view of the equation \eqref{bendingproblem}, we introduce a third-order corrector ${\vect u_2}$ as follows. Consider the following functional on $H_{\#}^1( Y;H^1(\omega;\C^3)):$
\begin{align*}
\widetilde{\vect f}_2 :=&-\chi^{-4}\left(\left({\rm i}X_\chi\right)^*{\mathbb A}\simgrad {\vect u_1} + (\simgrad)^* {\mathbb A}{\rm i}X_{\chi}{\vect u_1} + \left({\rm i}X_\chi\right)^* {\mathbb A} \Lambda_{\chi,m_1,m_2}^{\rm bend}\right)  \\[0.3em]
 &-\bigl( 0, 0, -{\rm i}\chi(m_1 x_1 + m_2 x_2) \bigr)^\top + \bigl(0, 0, |\chi|^{-1}{f_3}\bigr)^\top.
\end{align*}
This functional vanishes when tested against constant functions of form $(c_1,c_2,0)^\top,$ $c_1, c_2\in{\mathbb C}.$ To see this, we note first that 
\begin{equation}
	{\rm i}X_\chi(c_1, c_2, 0)^\top = 
	\begin{bmatrix}
		0 & 0 & \dfrac{1}{2}{\rm i}\chi c_1 \\[0.7em]
		0 & 0 & \dfrac{1}{2}{\rm i}\chi c_2 \\[0.7em]
		\dfrac{1}{2}{\rm i}\chi c_1 & \dfrac{1}{2}{\rm i}\chi c_1 & 0
	\end{bmatrix} = -\simgrad\bigl(0, 0, -{\rm i}\chi(c_1 x_1 + c_2 x_2)\bigr)^\top\qquad \forall c_1, c_2\in{\mathbb C},
\end{equation}
from which it follows that (cf. \eqref{lam_bend}, \eqref{lam_stretch})
\begin{equation}
	\Lambda^{\rm bend}_{\chi,c_1,c_2}(x)=  {\rm i}X_\chi\bigl(0, 0, -{\rm i}\chi(c_1 x_1 + c_2 x_2)\bigr)^\top\qquad \forall c_1, c_2\in{\mathbb C}.
\end{equation}
Next, we have 
\begin{align*}
\big(\left({\rm i}X_\chi\right)^*{\mathbb A}\simgrad {\vect u_1}&+ \left({\rm i}X_\chi\right)^* {\mathbb A} \Lambda_{\chi,m_1,m_2}^{\rm bend}\big)(c_1,c_2,0)^\top \\[0.3em]
&=\int_{\omega\times Y}{\mathbb A}\left(\simgrad {\vect u_1} + \Lambda_{\chi,m_1,m_2}^{\rm bend}\right):\overline{{\rm i}X_{\chi}  (c_1,c_2,0)^\top} \\[0.3em]
&=-\int_{\omega\times Y}{\mathbb A}\left(\simgrad {\vect u_1} + \Lambda_{\chi,m_1,m_2}^{\rm bend}\right):\overline{\simgrad\bigl( 0, 0, -{\rm i}\chi(c_1 x_1 + c_2 x_2)\bigr)^\top}\\[0.3em]
&=
-\left((\simgrad)^*{\mathbb A}\simgrad {\vect u_1} + (\simgrad)^* {\mathbb A} \Lambda_{\chi,m_1,m_2}^{\rm bend}\right)\bigl( 0, 0, -{\rm i}\chi(c_1 x_1 + c_2x_2)\bigr)^\top= 0.
\end{align*}
The last equality follows from the definition of the corrector term ${\vect u_1}$ and the fact that $$
\bigl(0,0,{\rm i}\chi(c_1 x_1 + c_2 x_2)\bigr)^\top\in H^{\rm bend}.
$$
Next we define the corrector ${\vect u_2}\in H^{\rm bend}$ as the solution to the following well-posed problem:
\begin{equation*}
\chi^{-4}(\simgrad)^* {\mathbb A} \simgrad {\vect u_2} = \widetilde{\vect f}_2.
\end{equation*}
It satisfies the bound
\begin{equation}\label{nak9} 
\Vert {\vect u_2}\Vert _{H^1(\omega\times Y ; \C^3)} \leq C|\chi|^3\bigl\Vert  \vect f \bigr\Vert _{L^2(\omega\times Y;\C^3)}.
\end{equation}

In order to decrease the error further, we proceed to introducing a corrector of higher order.
To this end, define ${\vect u}_3$ as the solution to
\begin{equation}
\label{u3bendproblem}
    \chi^{-4}(\simgrad)^* {\mathbb A} \simgrad {\vect u}_3 = \widetilde{\vect f}_3, 
\end{equation}
where the functional $\widetilde{\vect f}_3$ is defined by
\begin{equation}
    \widetilde{\vect f}_3 :=-\chi^{-4}\left(\left({\rm i}X_\chi\right)^*{\mathbb A}\simgrad {\vect u_2} + (\simgrad)^* {\mathbb A}{\rm i}X_{\chi}{\vect u_2} + \left({\rm i}X_\chi\right)^* {\mathbb A} {\rm i}X_{\chi}{\vect u_1}\right)  \\
 - ( m_1, m_2, 0 )^\top + (\,\widehat{\!\vect f}, 0)^\top.
\end{equation}
The following calculation shows that $\widetilde{\vect f}_3$ vanishes when tested against vectors of the form $(c_1, c_2, 0)^\top,$ $c_1, c_2\in{\mathbb C}:$
\begin{align*}
& \frac{1}{{\chi^4}} \int_{\omega \times Y} {\mathbb A} \left(\simgrad {\vect u_2} +  {\rm i}X_{\chi}{\vect u_1}\right):\overline{{\rm i}X_{\chi}(c_1, c_2, 0)^\top}\\[0.4em]
%
&\hspace{5em}=-\frac{1}{{\chi^4}}\int_{\omega \times Y} {\mathbb A}\left(\simgrad {\vect u_2} +  {\rm i}X_\chi {\vect u_1}\right):\overline{\simgrad\bigl(0, 0, -{\rm i}\chi(c_1 x_1 + c_2 x_2)\bigr)^\top}\\[0.4em]
&\hspace{5em}=\frac{1}{{\chi^4}} \int_{\omega \times Y} {\mathbb A}\left(\simgrad {\vect u_1} +  \Lambda_{\chi,m_1,m_2}^{\rm bend}\right):\overline{{\rm i}X_\chi\bigl(0, 0, -{\rm i}\chi(c_1 x_1 + c_2 x_2)\bigr)^\top}  \\[0.4em] 
 &\hspace{10em}+\int_{\omega\times Y}\bigl(0, 0, -{\rm i}\chi(m_1 x_1 + m_2 x_2)\bigr)^\top\cdot \overline{\bigl(0, 0, -{\rm i}\chi(c_1 x_1 + c_2 x_2)\bigr)^\top}\\[0.4em]
 &\hspace{15em}-\int_{\omega\times Y}\bigl(0, 0,|\chi|^{-1} {f_3}\bigr)^\top\cdot \overline{\bigl(0, 0, -{\rm i}\chi(c_1 x_1 + c_2 x_2)\bigr)^\top} \\[0.4em]
&\hspace{5em}=\frac{1}{{\chi^4}}\int_{\omega \times Y} {\mathbb A}\left( \simgrad {\vect u_1} + \Lambda_{\chi,m_1,m_2}^{\rm bend} \right) :\overline{\Lambda_{\chi,c_1,c_2}^{\rm bend}} 
\\[0.4em]
 &\hspace{10em}+\int_{\omega\times Y}(0, 0, -{\rm i}\chi(m_1 x_1 + m_2 x_2) )^\top\cdot \overline{\bigl(0, 0, -{\rm i}\chi(c_1 x_1 + c_2 x_2)\bigr)^\top}\\[0.4em]
 &\hspace{15em}- \int_{\omega\times Y}\bigl(0, 0,|\chi|^{-1}{f_3}\bigr)^\top\cdot \overline{\bigl(0, 0, -{\rm i}\chi(c_1 x_1 + c_2 x_2)\bigr)^\top}\\
%
&\hspace{5em}=-\int_{\omega \times Y}(m_1, m_2, 0 )^\top\cdot \overline{(c_1, c_2, 0 )^\top} + 
  \int_{\omega\times Y}(\,\widehat{\!\vect f}, 0 )^\top\cdot \overline{(c_1, c_2, 0 )^\top}. 
\end{align*}
This proves that the problem \eqref{u3bendproblem} is well posed and the solution satisfies the bound
\begin{equation}\label{nak13} 
\Vert {\vect u_3}\Vert _{H^1(\omega\times Y ; \C^3)} \leq C{\chi^4}\bigl\Vert  \vect f \bigr\Vert _{L^2(\omega\times Y;\C^3)}.
\end{equation}
The total approximation built up so far $\vect u_{\rm approx}:= \vect u_0 + {\vect u_1} + {\vect u_2} + {\vect u_3}$ satisfies
\begin{equation*}
\chi^{-4}\left((\simgrad)^* + \left({\rm i}X_\chi\right)^*\right){\mathbb A}\left(\simgrad + {\rm i}X_\chi\right) \vect u_{\rm approx} + \vect u_{\rm approx} -S_{\!|\chi|}\vect f = {\vect R}_{\chi},
\end{equation*}
where the right-hand side is given by
\begin{equation*}
{\vect R}_{\chi} =\chi^{-4}\left(\left({\rm i}X_\chi\right)^*{\mathbb A}{\rm i}X_\chi {\vect u_2} +\left({\rm i}X_\chi\right)^*{\mathbb A}\simgrad {\vect u_3} +  (\simgrad)^* {\mathbb A}{\rm i}X_{\chi}{\vect u_3} + \left({\rm i}X_\chi\right)^* {\mathbb A}{\rm i}X_{\chi} {\vect u_3} \right) + {\vect u_1} + {\vect u_2} + {\vect u_3},  
\end{equation*}
and hence satisfies
\begin{equation}
\Vert {\vect R}_{\chi}\Vert _{[H^1_\#(Y;H^1(\omega;\C^3))]^*} \leq C |\chi| \bigl\Vert  \vect f \bigr\Vert _{L^2(\omega\times Y,\C^3)}.
\end{equation}
We next refine the approximation by introducing suitable correctors ${\vect u}_j^{(1)},$ $j=0,1,2,3,$ to the functions ${\vect u}_j,$ $j=0,1,2,3,$ defined above.

\subsubsection*{3) First refinement of the approximation}
In order to eliminate lower-order terms on the right-hand side of the equation for the approximation error, see \eqref{approx_error_eq} below, we update the leading-order term of the expansion. We do so by defining 
$\vect u_0^{(1)} = \widetilde{\mathcal{I}}_\chi^{\rm bend}\begin{bmatrix}
    ({\vect m}^{(1)})_1 \\ 
    ({\vect m}^{(1)})_2
\end{bmatrix}$, where the vector $\begin{bmatrix}
    ({\vect m}^{(1)})_1 \\ 
    ({\vect m}^{(1)})_2
\end{bmatrix}$ solves  the identity 
\begin{equation}
	\label{u01_bend}
\begin{aligned}
&\left(\frac{1}{{\chi^4}}{\mathbb A}_\chi^{\rm bend}  + \mathfrak{C}^{\rm bend}_\chi\right)
\begin{bmatrix}
    ({\vect m}^{(1)})_1 \\[0.3em] 
    ({\vect m}^{(1)})_2
\end{bmatrix}\cdot \overline{
\begin{bmatrix}
    d_1 \\[0.2em] d_2
\end{bmatrix}}\\[0.6em]
&\hspace{6em}=-\frac{1}{{\chi^4}}\int_{\omega \times Y}{\mathbb A}\bigl(\simgrad {\vect u_3} + {\rm i}X_\chi {\vect u_2}\bigr):\overline{{\rm i}X_\chi(d_1,d_2,0)^\top}\qquad \forall (d_1,d_2)^\top\in \C^2.
\end{aligned}
\end{equation}
Clearly, one has the following bounds:
\begin{equation}
	\label{u0^1_bend}
\begin{aligned}
\Bigl|\bigl(({\vect m}^{(1)})_1,({\vect m}^{(1)})_2\bigr)^\top\Bigr| &\leq C|\chi| \bigl\Vert  \vect f \bigr\Vert _{L^2(\omega\times Y,\C^3)},\\[0.2em]
\bigl\Vert ({\vect u}_0^{(1)})_\alpha\bigr\Vert _{H^1(\omega\times Y ; \C)} &\leq C|\chi|\bigl\Vert  \vect f \bigr\Vert _{L^2(\omega\times Y;\C^3)}, \quad \alpha = 1,2,\qquad 
\bigl\Vert ({\vect u}_0^{(1)})_3\bigr\Vert _{H^1(\omega\times Y ; \C)}\leq C{\chi^2}\bigl\Vert  \vect f \bigr\Vert _{L^2(\omega\times Y;\C^3)}.
\end{aligned}
\end{equation}
Next, we define the corrector  ${\vect u_1^{(1)}} $ as the solution to
\begin{equation}
 (\simgrad)^* {\mathbb A} \simgrad {\vect u_1^{(1)}} = - (\simgrad)^* {\mathbb A} \Lambda_{\chi,({\vect m}^{(1)})_1,({\vect m}^{(1)})_2}^{\rm bend}.
 \label{u11_bending}
\end{equation}
Similarly to a previously shown argument, the function ${\vect u_1^{(1)}}$ satisfies the estimate 
\begin{equation}\label{nak10} 
\bigl\Vert {\vect u_1^{(1)}} \bigr\Vert _{H^1(\omega\times Y ; \C^3)} \leq C|\chi|^3\bigl\Vert  \vect f \bigr\Vert _{L^2(\omega\times Y;\C^3)}.
\end{equation}
Combining \eqref{u01_bend} and \eqref{u11_bending} yields the identity
\begin{align}
\frac{1}{{\chi^4}}\int_{\omega \times Y} {\mathbb A}\Bigl( \simgrad {\vect u_1^{(1)}} &+ \Lambda_{\chi,({\vect m}^{(1)})_1,({\vect m}^{(1)})_2}^{\rm bend} \Bigr) :\overline{\Lambda_{\chi,d_1,d_2}^{\rm bend}}
\\[0.4em]
&+\int_{\omega\times Y}
\begin{bmatrix}
    ({\vect m}^{(1)})_1 \\[0.3em] 
    ({\vect m}^{(1)})_2 \\[0.3em]
    -{\rm i}\chi(({\vect m}^{(1)})_1 x_1 + ({\vect m}^{(1)})_2 x_2)
\end{bmatrix}
\cdot 
\overline{\begin{bmatrix}
    d_1 \\[0.2em] d_2 \\[0.2em]
    -{\rm i}\chi(d_1 x_1 + d_2 x_2)
\end{bmatrix}} \\[0.5em]
&=-\frac{1}{{\chi^4}}\int_{\omega \times Y} {\mathbb A}\bigl( \simgrad {\vect u_3} + {\rm i}X_{\chi}{\vect u_2} \bigr) :\overline{{\rm i}X_\chi (d_1,d_2,0)^\top} \qquad \forall (d_1,d_2)^\top\in \C^2.
\end{align}

The approximation procedure is continued similarly to the case of the stretching space discussed in Section \ref{stretching_sec}. In particular, we sequentially define two more correctors ${\vect u_2^{(1)}} , {\vect u}_3^{(1)} \in H_{\#}^1(Y; H^1(\omega;\C^3))$ as solutions to the well-posed problems
\begin{align*}
\chi^{-4}(\simgrad)^* {\mathbb A} \simgrad {\vect u_2^{(1)}} &= 
 -\chi^{-4}\Big(\left({\rm i}X_\chi\right)^*{\mathbb A}\simgrad {\vect u_1^{(1)}} + (\simgrad)^* {\mathbb A}{\rm i}X_{\chi}{\vect u_1^{(1)}} 
 \\[0.3em]
 &\quad+\left({\rm i}X_\chi\right)^* {\mathbb A} \Lambda_{\chi,({\vect m}^{(1)})_1,({\vect m}^{(1)})_2}^{\rm bend}\Big) -\Bigl(0,0,-{\rm i}\chi\bigl(({\vect m}^{(1)})_1 x_1 + ({\vect m}^{(1)})_2 x_2\bigr)\Bigr)^\top,
 \\[0.5em]
\chi^{-4}(\simgrad)^* {\mathbb A} \simgrad \vect u_3^{(1)}&= 
-\chi^{-4}\left( \left({\rm i}X_\chi\right)^* {\mathbb A} \simgrad \bigl({\vect u_2^{(1)}} + {\vect u_3}\bigr) + (\simgrad)^* {\mathbb A} {\rm i}X_{\chi}\bigl({\vect u_2^{(1)}} + {\vect u_3}\bigr) \right)
 \\[0.4em]
&\quad-\chi^{-4}\left( \left({\rm i}X_\chi\right)^* {\mathbb A}{\rm i}X_{\chi}\bigl({\vect u_1^{(1)}} + {\vect u_2}\bigr) \right) - \bigl(({\vect m}^{(1)})_1, ({\vect m}^{(1)})_2, 0\bigr)^\top.
\end{align*}
By a straightforward calculation, it is shown that they satisfy the bounds
\begin{equation}\label{nak7} 
\bigl\Vert{\vect u_2^{(1)}} \bigr\Vert_{H^1(\omega\times Y ; \C^3)} \leq C{\chi^4}\bigl\Vert  \vect f \bigr\Vert _{L^2(\omega\times Y;\C^3)},\qquad 
\bigl\Vert\vect u_3^{(1)}\bigr\Vert_{H^1(\omega\times Y ; \C^3)} \leq C|\chi|^5\bigl\Vert  \vect f \bigr\Vert _{L^2(\omega\times Y;\C^3)}.
\end{equation}
For our purposes it is necessary to further decrease the approximation error, and hence we iterate the above procedure once again.
\subsubsection*{4) Second refinement of the approximation}
Here we define correctors that eliminate the remaining lower-order terms on the right-hand side of the equation satisfied by the approximation error, thus achieving the desired order of smallness of the error itself.
 
The correctors $\vect u_0^{(2)} = \widetilde{\mathcal{I}}^{\rm bend}_\chi \begin{bmatrix}
    ({\vect m}^{(2)})_1 \\ ({\vect m}^{(2)})_2
\end{bmatrix}, \vect u_1^{(2)}, \vect u_2^{(2)}, \vect u_3^{(2)} \in H^{\rm bend}$ are constructed sequentially using the following relations:
\begin{align}
&\begin{aligned}
&\left(\frac{1}{{\chi^4}}{\mathbb A}_\chi^{\rm bend}  + \mathfrak{C}^{\rm bend}_\chi\right) \begin{bmatrix}
    ({\vect m}^{(2)})_1 \\[0.3em] ({\vect m}^{(2)})_2
\end{bmatrix}\cdot \overline{\begin{bmatrix}
    d_1 \\[0.25em] d_2
\end{bmatrix}} 
=-\int_{\omega \times  Y}{\mathbb A}\bigl(\simgrad {\vect u_3^{(1)}} + {\rm i}X_\chi {\vect u_2^{(1)}} + {\rm i}X_\chi {\vect u_3}\bigr):\overline{{\rm i}X_\chi (d_1,d_2,0)^\top} \\[0.3em] 
&\hspace{15em}-{\rm i}\chi\int_{\omega \times Y} \left(x_1{u_1},x_2{u_2} \right)^\top\cdot \overline{(d_1,d_2)^\top}\qquad
\forall (d_1,d_2)^\top\in \C^2,
\end{aligned}
\label{hom2}\\[0.5em]
&(\simgrad)^* {\mathbb A} \simgrad \vect u_1^{(2)}
 =-(\simgrad)^*{\mathbb A} \Lambda_{\chi,({\vect m}^{(2)})_1,({\vect m}^{(2)})_2}^{\rm bend},\nonumber
 \\[0.5em]
&\chi^{-4}(\simgrad)^* {\mathbb A} \simgrad {\vect u_2^{(2)}} = 
 -\chi^{-4}\Big(\left({\rm i}X_\chi\right)^*{\mathbb A}\simgrad {\vect u_1^{(2)}}  + (\simgrad)^* {\mathbb A}{\rm i}X_{\chi}{\vect u_1^{(2)}} \nonumber 
 \\[0.4em]
 &\hspace{15em}
 +\left({\rm i}X_\chi\right)^* {\mathbb A} \Lambda_{\chi,({\vect m}^{(2)})_1,({\vect m}^{(2)})_2}^{\rm bend}\Big)-\Bigl(0,0,-{\rm i}\chi\bigl(({\vect m}^{(2)})_1 x_1 + ({\vect m}^{(2)})_2 x_2\bigr)\Bigr)^\top,\nonumber
 \\[0.5em] 
&\chi^{-4}(\simgrad)^* {\mathbb A} \simgrad {\vect u_3^{(2)}} = 
-\chi^{-4}\left( \left({\rm i}X_\chi\right)^* {\mathbb A} \simgrad\bigl({\vect u_2^{(2)}}  + {\vect u_3^{(1)}}\bigr) + (\simgrad)^* {\mathbb A} {\rm i}X_{\chi}\bigl({\vect u_2^{(2)}}  + {\vect u_3^{(1)}}\bigr) \right)\nonumber
 \\[0.4em]
&\hspace{15em}-\chi^{-4}\left( \left({\rm i}X_\chi\right)^* {\mathbb A} {\rm i}X_{\chi}\bigl({\vect u_1^{(2)}}  + {\vect u_2^{(1)}} + {\vect u_3}\bigr) \right) - \bigl(({\vect m}^{(2)})_1, ({\vect m}^{(2)})_2, 0\bigr)^\top - {\vect u_1}.\nonumber
\end{align}
All of these problems are well posed, which can be seen by checking directly that their right-hand sides vanish when tested against functions in $H^{\rm bend}$. Furthermore, one has
\begin{align}\label{nak11} 
\bigl\Vert({\vect u}_0^{(2)})_\alpha\bigr\Vert_{H^1(\omega\times Y ; \C)} &\leq C{\chi^2}\bigl\Vert  \vect f \bigr\Vert _{L^2(\omega\times Y;\C^3)},\quad\alpha=1,2, \qquad \bigl\Vert({\vect u}_0^{(2)})_3\bigr\Vert_{H^1(\omega\times Y ; \C)} \leq C|\chi|^3\bigl\Vert  \vect f \bigr\Vert _{L^2(\omega\times Y;\C^3)}\\[0.3em]
\bigl\Vert{\vect u_1^{(2)}} \bigr\Vert_{H^1(\omega\times Y ; \C^3)}&\leq C{\chi^4}\bigl\Vert  \vect f \bigr\Vert _{L^2(\omega\times Y;\C^3)},\qquad
\bigl\Vert{\vect u_2^{(2)}} \bigr\Vert_{H^1(\omega\times Y ; \C^3)} \leq C|\chi|^5\bigl\Vert  \vect f \bigr\Vert _{L^2(\omega\times Y;\C^3)},\\[0.3em]
\bigl\Vert{\vect u_3^{(2)}} \bigr\Vert_{H^1(\omega\times Y ; \C^3)}&\leq C\chi^6\bigl\Vert  \vect f \bigr\Vert _{L^2(\omega\times Y;\C^3)}.
\end{align}

\subsubsection*{5) Final approximation}
Here we carry out the last step of the approximation procedure. To this end, we define the function 
$$ 
\widetilde{\vect u}_{\rm approx}:= \vect u_0 + \vect u_0^{(1)} + \vect u_0^{(2)} + {\vect u_1} + {\vect u_1^{(1)}} + {\vect u_1^{(2)}}  + {\vect u_2} + {\vect u_2^{(1)}} + {\vect u_2^{(2)}}, 
$$
 which  satisfies
\begin{equation}
\chi^{-4}\left((\simgrad)^* + \left({\rm i}X_\chi\right)^*\right){\mathbb A}\left(\simgrad + {\rm i}X_\chi\right) \widetilde{\vect u}_{\rm approx} + \widetilde{\vect u}_{\rm approx} - S_{\!|\chi|}\vect f = \widetilde{\vect R}_{\chi},
\end{equation}
where $\widetilde{\vect R}_{\chi}$ is given by
\begin{align*}
\widetilde{\vect R}_{\chi}=\chi^{-4}\Big(\left({\rm i}X_\chi\right)^*{\mathbb A}{\rm i}X_{\chi} ({\vect u_3^{(1)}} + {\vect u_2^{(2)}}  + {\vect u_3^{(2)}}  )&+(\simgrad)^* {\mathbb A}{\rm i}X_{\chi}{\vect u_3^{(2)}}  +\left({\rm i}X_\chi\right)^* {\mathbb A} \simgrad {\vect u_3^{(2)}} \Big)  \\[0.1em]
&+{\vect u_2} + {\vect u_3} + {\vect u_1^{(1)}} + {\vect u_2^{(1)}} + {\vect u_3^{(1)}} + {\vect u_1^{(2)}}  + {\vect u_2^{(2)}}  + {\vect u_3^{(2)}}
\end{align*}
and can be estimated as follows:
\begin{equation}
\bigl\Vert \widetilde{\vect R}_{\chi}\bigr\Vert _{[H^1_\#(Y;H^1(\omega;\C^3))]^*} \leq C |\chi|^3 \bigl\Vert  \vect f \bigr\Vert _{L^2(\omega\times Y,\C^3)}.
\end{equation}
The approximation error
$$
\vect u_{\rm error}:= \vect u - \widetilde{\vect u}_{\rm approx}. 
$$
satisfies
\begin{equation}
\chi^{-4}\left((\simgrad)^* + \left({\rm i}X_\chi\right)^*\right){\mathbb A}\left(\simgrad + {\rm i}X_\chi\right) \vect u_{\rm error} + \vect u_{\rm error} = -\widetilde{\vect R}_{\chi}.
\label{approx_error_eq} 
\end{equation}
A suitable version of the estimate \eqref{aprioriestimateh-1} now yields the bound
\begin{equation} \label{nak8} 
\begin{split}
    \bigl\Vert\vect u_{\rm error}\bigr\Vert_{H^1(\omega\times Y, \C^3)} \leq C|\chi|^3\bigl\Vert  \vect f \bigr\Vert _{L^2(\omega\times Y, \C^3)}.
\end{split}
\end{equation}
The following proposition provides the final estimates on the approximation.
\begin{proposition}\label{propbend} 
Let $\vect u\in H^1_{\#}(Y;H^1(\omega;\C^3))$ be the  solution of problem \eqref{bendingproblem}. Then, the following estimates  hold: 
\begin{equation}
\begin{aligned}
	\label{bend_estimates}
 \bigl\Vert P_i(\vect u -{\vect u}_0)
\bigr\Vert_{H^1(\omega\times Y,\C^2)}&\leq \left\{ \begin{array}{ll}
        C|\chi|\bigl\Vert  \vect f \bigr\Vert _{L^2(\omega\times Y;\C^3)}, & \mbox{ $i = 1,2$},\\[0.5em]
         C{\chi^2}\bigl\Vert  \vect f \bigr\Vert _{L^2(\omega\times Y;\C^3)}, & \mbox{ $i = 3,$}\end{array} \right.  \\[0.7em]
 \bigl\Vert P_i(\vect u - {\vect u}_0-{\vect u_0^{(1)}} 
-{\vect u_1})\bigr\Vert_{H^1(\omega\times Y,\C^2)}
&\leq \left\{ \begin{array}{ll}
        C{\chi^2}\bigl\Vert  \vect f \bigr\Vert _{L^2(\omega\times Y;\C^3)}, & \mbox{ $i = 1,2$},\\[0.5em]
         C|\chi|^3\bigl\Vert  \vect f \bigr\Vert _{L^2(\omega\times Y;\C^3)}, & \mbox{ $i = 3$},\end{array} \right. 
\end{aligned}
\end{equation}
where ${\vect u}_0, {\vect u}_0^{(1)}, {\vect u_1}$
are defined by the above approximation procedure in the bending space.
\end{proposition}

\begin{proof} 
The proof is a direct consequence of \eqref{nak8}, by utilising the estimates \eqref{nak6}, \eqref{nak13}, \eqref{u0^1_bend}, \eqref{nak10}, \eqref{nak7}, and \eqref{nak11}.  	
\end{proof} 	

\begin{remark}
	\label{replacement_rem}
	\begin{enumerate}[label=\textbf{(\alph*)}]
    \item Note that the estimates in above asymptotic procedure continue to hold if the vectors $(m_1, m_2)^\top$, $(({\vect m}^{(1)})_1, ({\vect m}^{(1)})_2)^\top,$ $(({\vect m}^{(2)})_1, ({\vect m}^{(2)})_2)^\top$ are replaced throughout by the solutions $(\widetilde{m}_1, \widetilde{m}_2)^\top$, $((\widetilde{\vect m}^{(1)})_1, (\widetilde{\vect m}^{(1)})_2)^\top,$ $((\widetilde{\vect m}^{(2)})_1, (\widetilde{\vect m}^{(2)})_2)^\top$ to the problems obtained from \eqref{hom1}, \eqref{u01_bend}, \eqref{hom2} respectively, by setting $\mathfrak{C}^{\rm bend}_{\chi} = I.$ We start by demonstrating how the described change affects the part pertaining to $(m_1, m_2)^\top.$
    
Using $(m_1 - \widetilde{m}_1, m_2 - \widetilde{m}_2)^\top$ as a test function in \eqref{hom1} and in the equation
\[
 \bigl({\mathbb A}^{\rm bend}+I\bigr)
\begin{bmatrix} \widetilde{m}_1 \\ \widetilde{m}_2 \end{bmatrix}  = 
\mathcal{M}_{\chi}^{\rm bend} S_{\!|\chi|} \vect f
\]
and subtracting one from the other, yields
$$
{\mathbb A}^{\rm bend}
 \begin{bmatrix}
 m_1 - \widetilde{m}_1 \\ m_2 - \widetilde{m}_2     
 \end{bmatrix} \cdot \begin{bmatrix}
 m_1 - \widetilde{m}_1 \\ m_2 - \widetilde{m}_2     
 \end{bmatrix} + {\chi^2}\begin{bmatrix}
 {\mathfrak c}_1(\omega) & 0 \\ 0 &  {\mathfrak c}_2(\omega)     
 \end{bmatrix} \begin{bmatrix}
 m_1  \\ m_2     
 \end{bmatrix} \cdot \begin{bmatrix}
 m_1 - \widetilde{m}_1 \\ m_2 - \widetilde{m}_2     
 \end{bmatrix} = 0.
 $$
Therefore, one has 
 $$
 \left\vert\begin{bmatrix}
 m_1 - \widetilde{m}_1 \\ m_2 - \widetilde{m}_2     
 \end{bmatrix}\right\vert \leq C{\chi^2} \bigl\Vert  \vect f \bigr\Vert _{L^2(\omega\times Y;\C^3)}.
$$
Similarly, one obtains appropriate estimates for the differences 
\[
\begin{bmatrix}
	({\vect m}^{(1)})_1 - (\widetilde{\vect m}^{(1)})_1 \\  ({\vect m}^{(1)})_2 - (\widetilde{\vect m}^{(1)})_2
\end{bmatrix},\quad 
\begin{bmatrix}
	({\vect m}^{(2)})_1 - (\widetilde{\vect m}^{(2)})_1 \\  ({\vect m}^{(2)})_2 - (\widetilde{\vect m}^{(2)})_2,
\end{bmatrix},
\]
which concludes the proof of the above claim.
 
\item Note that using ${\mathfrak C}^{\rm bend}_\chi$ instead of ${\mathfrak C}^{\rm bend}$ is equivalent to utilising the expression 
\begin{equation}
-{\rm i}\chi(m_1 x_1 + m_2 x_2)
\label{third_comp}
\end{equation}
 in the third component of the second term on the left-hand side of \eqref{first_step_bend}. The observation made in the present remark above shows, in particular, that the asymptotics procedure could have been carried by setting the said component to zero. However, we decided to include the term \eqref{third_comp}
in order to show that the approximate solution has the form of the Kirchhoff-Love ansatz, cf. Section \ref{sec31}, which by implication necessitated the use of the tensor ${\mathfrak C}^{\rm bend}_\chi.$ 
\end{enumerate}
\end{remark}  
\begin{remark}
   The first estimate in \eqref{bend_estimates} can be rewritten as:
   \begin{equation}
   \label{bendingestimate1}
       \left\Vert P_i\left(\left(\frac{1}{{\chi^4}}\mathcal{A}_\chi + I \right)^{-1}\bigg|_{L^2_{\rm bend}} - \left(\widetilde{\mathcal{M}}_\chi^{\rm bend}\right)^*\bigl({\mathbb A}^{\rm bend} + I \bigr)^{-1}\widetilde{\mathcal{M}}_\chi^{\rm bend}  \right)S_{\!|\chi|} \right\Vert_{L^2 \to H^1} \leq  \left\{ \begin{array}{ll}
        C|\chi|, & \mbox{ $i = 1,2$},\\[0.4em]
         C{\chi^2}, & \mbox{ $i = 3$}.\end{array} \right. 
   \end{equation}
   The second estimate can be rewritten as:
   \begin{equation}
   \label{bendingestimate2}
   \begin{split}
       \left\Vert P_i\left(\left(\dfrac{1}{\chi^4}\mathcal{A}_\chi + I \right)^{-1}\bigg|_{L^2_{\rm bend}} - \left(\widetilde{\mathcal{M}}_\chi^{\rm bend}\right)^*\bigl({\mathbb A}^{\rm bend} + I \bigr)^{-1}\widetilde{\mathcal{M}}_\chi^{\rm bend} - \mathcal{A}_{\chi,\rm corr}^{\rm bend} - \widetilde{\mathcal{A}}_{\chi,\rm corr}^{\rm bend} \right) S_{\!|\chi|}\right\Vert_{L^2 \to H^1} \\[0.7em] \leq \left\{ \begin{array}{ll}
        C{\chi^2}, & \mbox{ $i = 1,2$},\\[0.35em]
         C|\chi|^3, & \mbox{ $i = 3$},\end{array} \right.
            \end{split}
   \end{equation}
   where the bounded operators $\mathcal{A}_{\chi,\rm corr}^{\rm bend}$ and $\widetilde{\mathcal{A}}_{\chi,\rm corr}^{\rm bend}$ are defined by 
   \begin{equation}
       \mathcal{A}_{\chi,\rm corr}^{\rm bend} \vect f := {\vect u_1}, \qquad \widetilde{\mathcal{A}}_{\chi,\rm corr}^{\rm bend} \vect f := \vect u_0^{(1)}.
   \end{equation}

\end{remark}

\begin{remark}
 We note that the absence of the scaling term $S_{\!|\chi|}$ is similar to the absence of the out-of-line force term in the $H^1$ estimate for the distance between the original and homogenised resolvents. Indeed, consider the following two problems: 
 \begin{equation}
     \bigl({\mathbb A}^{\rm bend}  + I \bigr)\begin{bmatrix}
 \widehat{m}_1  \\ \widehat{m}_2     
 \end{bmatrix} = \widetilde{\mathcal{M}}_\chi^{\rm  bend}\vect f, \qquad     \bigl({\mathbb A}^{\rm bend}+I\bigr)\begin{bmatrix}
 \widetilde{m}_1  \\ \widetilde{m}_2     
 \end{bmatrix} = \widetilde{\mathcal{M}}_\chi^{\rm  bend} S_\infty \vect f = \int_{\omega \times Y}\,\widehat{\!\vect f},
 \end{equation}
 where 
 \begin{equation} \label{nak70} 
 S_\infty = \begin{bmatrix}
 1 & 0 & 0 \\ 0 & 1 & 0 \\ 0 & 0 & 0
 \end{bmatrix}.
 \end{equation} 
 The difference $(\widehat{m}_1 - \widetilde{m}_1, \widehat{m}_2 - \widetilde{m}_2)^\top$ satisfies 
 \begin{equation}
     \bigl({\mathbb A}^{\rm bend}  + I \bigr)\begin{bmatrix}
\widehat{m}_1 - \widetilde{m}_1  \\[0.25em] \widehat{m}_2 - \widetilde{m}_2     
 \end{bmatrix} =\int_{\omega \times Y}{\rm i}\chi f_3\widehat{x}
 \end{equation}
 so that one has
 \begin{equation}
     \bigl|\left(\widehat{m}_1 - \widetilde{m}_1, \widehat{m}_2 - \widetilde{m}_2  \right)^\top\bigr|\leq C |\chi| \left\Vert \vect f\right\Vert_{L^2(\omega \times Y;\C^3)}.
 \end{equation}
It follows immediately that
\begin{equation}
    \left\Vert P_i\left( \widetilde{\mathcal{I}}_\chi^{\rm bend}\begin{bmatrix}
    \widehat{m}_1 \\ \widehat{m}_2
    \end{bmatrix} - \widetilde{\mathcal{I}}_\chi^{\rm bend}\begin{bmatrix}
    \widetilde{m}_1 \\ \widetilde{m}_2
    \end{bmatrix} \right) \right\Vert_{H^1(\omega\times Y;\C^3)} \leq  \left\{ \begin{array}{ll}
        C|\chi|\bigl\Vert  \vect f \bigr\Vert _{L^2(\omega\times Y;\C^3)}, & \mbox{ $i = 1,2$},\\[0.6em]
         C{\chi^2}\bigl\Vert  \vect f \bigr\Vert _{L^2(\omega\times Y;\C^3)}, & \mbox{ $i = 3$},\end{array}\right. 
\end{equation} 
and therefore
 \begin{equation}
     \left\Vert  P_i \left( \widetilde{\mathcal{M}}_\chi^{\rm bend}\right)^* \bigl({\mathbb A}^{\rm bend}
     +I\bigr)^{-1}\widetilde{\mathcal{M}}_\chi^{\rm bend} \left( I - S_\infty \right)\right\Vert _{L^2 \to H^1} \leq  \left\{ \begin{array}{ll}
        C|\chi|, & \mbox{ $i = 1,2$},\\[0.5em]
         C{\chi^2}, & \mbox{ $i = 3$}.\end{array} \right.
 \end{equation}
 Similarly, due to \eqref{aprioriestimatebendingh1}, see Remark \ref{bending_apriori_rem}, we have 
 \begin{equation}
 \label{noforcescaling}
     \left\Vert  P_i  \left(\frac{1}{{\chi^4}}\mathcal{A}_\chi + I \right)^{-1} \left( I - S_\infty \right)\right\Vert _{L^2 \to H^1} \leq  \left\{ \begin{array}{ll}
        C|\chi|, & \mbox{ $i = 1,2$},\\[0.5em]
         C{\chi^2}, & \mbox{ $i = 3$}.\end{array} \right.
 \end{equation}
\end{remark}

 \begin{remark}
    We emphasise that the asymptotic procedure, performed in this section can be extended to an approximation of arbitrary order in $|\chi|$, by generating a series in $\chi$. One can write down explicit recurrence relations that define the relevant correctors entering this series up to any order in $|\chi|$. 
 \end{remark}

\section{Norm-resolvent estimates under geometric and material symmetries}
\label{section5}
In this section, we provide the norm-resolvent estimates for the operators on the original domain $\omega \times \R$. This is done by chosing optimal estimates with respect to $\varepsilon$ for each $\chi$, and then applying the Gelfand pullback in order to  translate  the estimates to the original physical setting. 

\subsection{$L^2 \to L^2$ norm-resolvent estimates}
\label{L2toL2}
We start with the stretching case. Recall that $P_i:\R^3 \to \R$ is the projection on the $i$-th vector component, $i=1,2,3;$ in what follows we also utilise similar projections $\pi_i :\R^2 \to \R,$ $i=1,2.$

\begin{proof}[Proof of the Theorem \ref{THML2L2} under Assumption \ref{matsym}, formula \eqref{stretch23}]\RRR
First, we notice that the norm-resolvent estimate \eqref{stretch23} is equivalent to 
\begin{equation}
\label{stretchl2l2estimatescomponents}
    \begin{aligned}
        &\left\Vert P_1\left(\frac{1}{\varepsilon^{\gamma}} \mathcal{A}_\varepsilon + I \right)^{-1}\vect f - x_2 \pi_1 \left(\frac{1}{\varepsilon^{\gamma}}\mathcal{A}^{\rm stretch} + \mathfrak{C}^{\rm stretch}\right)^{-1}\mathcal{M}^{\rm stretch} \Xi_\varepsilon\vect f \right\Vert_{L^2(\omega\times \mathbb{R})} \leq C\varepsilon^{\tfrac{\gamma+2}{2}} \bigl\Vert  \vect f \bigr\Vert _{L^2(\omega\times \mathbb{R};\R^3)},  \\[0.6em]
        &\left\Vert P_2\left( \frac{1}{\varepsilon^{\gamma}}\mathcal{A}_\varepsilon + I \right)^{-1}\vect f + x_1 \pi_1 \left(\frac{1}{\varepsilon^{\gamma}}\mathcal{A}^{\rm stretch} + \mathfrak{C}^{\rm stretch}\right)^{-1}\mathcal{M}^{\rm stretch}\Xi_\varepsilon\vect f \right\Vert_{L^2(\omega\times \mathbb{R})} \leq C\varepsilon^{\tfrac{\gamma+2}{2}} \bigl\Vert  \vect f \bigr\Vert _{L^2(\omega\times \mathbb{R};\R^3)}, \\[0.6em]
        &\left\Vert P_3 \left(\frac{1}{\varepsilon^{\gamma}} \mathcal{A}_\varepsilon + I \right)^{-1}\vect f -  \pi_2 \left(\frac{1}{\varepsilon^{\gamma}}\mathcal{A}^{\rm stretch} + \mathfrak{C}^{\rm stretch}\right)^{-1}\mathcal{M}^{\rm stretch}\Xi_\varepsilon\vect f \right\Vert_{L^2(\omega\times \mathbb{R})} \leq C\varepsilon^{\tfrac{\gamma+2}{2}} \bigl\Vert  \vect f \bigr\Vert _{L^2(\omega\times \mathbb{R};\R^3)}.
    \end{aligned}
\end{equation}
Therefore, in what follows we aim at proving \eqref{stretchl2l2estimatescomponents}.
\CCC We will only prove the first estimate, since the others are obtained in a similar way. 
We recall the estimates \eqref{stretch_estimates} for the solutions of the resolvent equations 
\begin{equation*}
 \bigl(\chi^{-2}\mathcal{A}_\chi + I\bigr)\vect u = \vect f,\qquad\quad
 \bigl({\mathbb A}^{\rm stretch}  + \mathfrak{C}^{\rm stretch}\bigr)(m_1,m_2)^\top = \widetilde{\mathcal{M}}^{\rm stretch} \vect f.
\end{equation*}

\CCC By virtue of \eqref{stretchingestimate1}, the following norm-resolvent estimate holds:
\begin{align*}
    &\left\Vert P_1\left( \frac{1}{{\chi^2}} \mathcal{A}_\chi + I \right)^{-1} \vect f - x_2 \pi_1 \bigl({\mathbb A}^{\rm stretch}  + \mathfrak{C}^{\rm stretch}\bigr)^{-1}\widetilde{\mathcal{M}}^{\rm stretch}\vect f \right\Vert_{L^2(\omega \times  Y)} \leq C |\chi| \bigl\Vert  \vect f \bigr\Vert _{L^2(\omega \times  Y;\C^3)}.
\end{align*}

Recall that the operator $\mathcal{A}_\chi$ is selfadjoint, positive, with compact resolvent, so its spectrum consists of real non-negative eigenvalues, all of which are of order ${O}(1)$ except for the  smallest two,  $\lambda_1^\chi, \lambda_2^\chi$, which are of order ${O}({\chi^2})$. (We  can estimate the  interval  in which $\lambda_1^\chi/{\chi^2}$ and $\lambda_2^\chi/{\chi^2}$ are found uniformly in $|\chi|$). 

 For every $\varepsilon>0$, $\chi \neq 0$ we define
\begin{equation}\label{def_g_analytic}
   g_{\varepsilon,\chi}(z) := \left(\frac{{\chi^2}}{\varepsilon^{\gamma+2}}z + 1 \right)^{-1},\quad  \Re(z) > 0.
\end{equation} 
 Note that for every $\eta > 0,$ the function  $g_{\varepsilon,\chi}$ is bounded on the half-plane $\left\{z \in \C, \Re(z) \geq \eta\right\}:$
 \begin{equation}
     \bigl|g_{\varepsilon,\chi}(z)\bigr| \leq C(\eta) \left(\max\left\{\frac{{\chi^2}}{\varepsilon^{\gamma+2}}, 1\right\}\right)^{-1}.
 \end{equation}
 Due to the bounds on the both eigenvalues of $\mathcal{A}_\chi$ of order ${\chi^2}$, we deduce that the two smallest eigenvalues of the operator $\chi^{-2}\mathcal{A}_\chi$ are uniformly positioned within a fixed, independent  $\chi$ nor $\varepsilon,$ interval  on the positive real axis. The same is true for the two \RRR eigenvalues $\kappa_1^\chi$, $\kappa_2^\chi$ of the matrix  $\chi^{-2}{\mathbb A}_\chi^{\rm stretch} $.  Uniform bounds on these eigenvalues allow us to deduce that there exists a closed contour $\Gamma \subset \left\{z\in \C, \Re(z)>0\right\}$ and a constant $\mu > 0$, such that for every $\chi \in [-\mu,\mu]\setminus \left\{0\right\}$ one has the following properties:
\begin{itemize}
\item $\Gamma$ encloses the two smallest eigenvalues of both the operators $\chi^{-2}\mathcal{A}_\chi$ and $\chi^{-2}{\mathbb A}_\chi^{\rm stretch} $. 
\item $\Gamma$ does not enclose any other eigenvalue (of higher order). 
\item $\exists \rho_0 > 0$,  $\inf_{z\in \Gamma}|z-\lambda_i^\chi|\geq \rho_0$,  $\inf_{z\in \Gamma}|z-\kappa_i^\chi|\geq \rho_0$, $i=1,2.$ 
\end{itemize}
Note that $g_{\varepsilon,\chi}$ is analytic in the right half-plane, the Cauchy integral formula yields
$$
g_{\varepsilon,\chi}(\mathcal{A_\chi})P_{\Gamma}:=\frac{1}{2\pi{\rm i}} \oint_\Gamma g_{\varepsilon,\chi}(z)(zI-\mathcal{A}_\chi)^{-1}dz, 
$$
where $P_\Gamma$ is the projection operator onto the eigenspace spanned  by  eigenfunctions corresponding  to the  eigenvalues enclosed  by  $\Gamma,$ and the integral over $\Gamma$ is taken in the anticlockwise direction. Note that
\begin{equation*}
    \frac{1}{\varepsilon^{\gamma+2}}\mathcal{A}_\chi + I = \frac{{\chi^2}}{\varepsilon^{\gamma+2}}\left(\frac{1}{{\chi^2}}\mathcal{A}_\chi \right) + I,
\end{equation*} 
 and therefore 
$$
\left(\frac{1}{\varepsilon^{\gamma+2}}\mathcal{A}_\chi + I\right)^{-1}P_\Gamma \vect f =\frac{1}{2\pi{\rm i}} \oint_\Gamma g_{\varepsilon,\chi}(z)\left(zI-\frac{1}{{\chi^2}}\mathcal{A}_\chi\right)^{-1} \vect  f dz.
$$
Due to the uniform estimates on the spectrum of $\mathcal{A}_\chi$, we have 
\begin{equation*}
    \begin{split}
        \left\Vert
        \left( \frac{1}{\varepsilon^{\gamma+2}}\mathcal{A}_\chi + I\right)^{-1}
        -\left( \frac{1}{\varepsilon^{\gamma+2}}\mathcal{A}_\chi + I\right)^{-1}P_\Gamma
        \right\Vert_{L^2 \to H^1} 
        \!\!=\left\Vert
        \left( \frac{1}{\varepsilon^{\gamma+2}}\mathcal{A}_\chi + I\right)^{-1}
        (I-P_\Gamma)
        \right\Vert_{L^2 \to H^1} \leq C\varepsilon^{\gamma+2},
    \end{split}
\end{equation*}
so, by applying the triangle inequality, in the estimates we can drop the projection $P_\Gamma.$  
Note that the integral formula can also be applied to the resolvent  $({\mathbb A}^{\rm stretch}  + \mathfrak{C}^{\rm stretch})^{-1}$ despite its non-standard structure. As a result, we obtain
\begin{align*}
&\left\Vert P_1\left(\frac{1}{\varepsilon^{\gamma+2}}\mathcal{A}_\chi+I\right)^{-1}\vect f - x_2\pi_1\left(\frac{\chi^2}{\varepsilon^{\gamma+2}}{\mathbb A}^{\rm stretch} +\mathfrak{C}^{\rm stretch}\right)^{-1}\widetilde{\mathcal{M}}^{\rm stretch}\vect f \right\Vert_{L^2(\omega \times Y)}  
\\[0.5em]
&\leq \frac{1}{2\pi} \oint_{\Gamma}\bigl|g_{\varepsilon,\chi}(z)\bigr|\left\Vert P_1\left(zI -\frac{1}{{\chi^2}}\mathcal{A}_\chi\right)^{-1} \vect f - x_2\pi_1\bigl(z\mathfrak{C}^{\rm stretch}-{\mathbb A}^{\rm stretch}\bigr)^{-1}\widetilde{\mathcal{M}}^{\rm stretch} \vect f \right\Vert_{L^2(\omega \times Y)} dz 
 \\[0.5em]
 &\leq C|\chi| \left(\max\left\{\frac{\chi^{2}}{\varepsilon^{\gamma+2}}, 1\right\}\right)^{-1}\|\vect f\|_{L^2(\omega \times Y;\C^3)} \leq C \varepsilon^{\tfrac{{\gamma+2}}{2}}\|\vect f\|_{L^2(\omega \times Y;\C^3)}, 
\end{align*}
where the bound is the sharpest when ${\chi^2} \approx \varepsilon^{\gamma+2}$.

Since 
\begin{equation}
	\label{trans_back_stretch}
		\begin{aligned}
 \left( \frac{1}{\varepsilon^{\gamma}}\mathcal{A}^{\rm stretch} + \mathfrak{C}^{\rm stretch} \right)^{-1} &= \mathfrak{G}_\varepsilon^{-1}\left( \frac{\chi^2}{\varepsilon^{\gamma+2}}{\mathbb A}^{\rm stretch}  + \mathfrak{C}^{\rm stretch} \right)^{-1} \mathfrak{G}_\varepsilon,\\[0.6em]    \left(\frac{1}{\varepsilon^{\gamma}} \mathcal{A}_\varepsilon + I \right)^{-1}&= \mathfrak{G}_\varepsilon^{-1}\left( \frac{1}{\varepsilon^{\gamma+2}}\mathcal{A}_\chi + I \right)^{-1} \mathfrak{G}_\varepsilon,
 \end{aligned}
\end{equation}
the \CCC first estimate in \eqref{stretchl2l2estimatescomponents}  follows from the following \CCC identity,  which uses \eqref{smootheningoperator1} :
\begin{align}
&\begin{aligned}
&\mathfrak{G}_\varepsilon^{-1}\left(P_1 \left( \frac{1}{\varepsilon^{\gamma+2}}\mathcal{A}_\chi + I \right)^{-1} - x_2\pi_1\left( \frac{\chi^2}{\varepsilon^{\gamma+2}}{\mathbb A}^{\rm stretch}  + \mathfrak{C}^{\rm stretch} \right)^{-1}\widetilde{\mathcal{M}}^{\rm stretch}\right)  \mathfrak{G}_\varepsilon  \\[0.4em]
&\hspace{10em}=P_1\left( \frac{1}{\varepsilon^{\gamma}}\mathcal{A}_\varepsilon + I \right)^{-1} - x_2\pi_1\left(\frac{1}{\varepsilon^{\gamma}} \mathcal{A}^{\rm stretch} + \mathfrak{C}^{\rm stretch} \right)^{-1}\mathcal{M}^{\rm stretch}\Xi_\varepsilon, 
\end{aligned}	\label{trans_back_stretch1}
\end{align}
 by virtue of   the fact that the Gelfand transform is an isometry.
\end{proof}
Next, we analyse the bending case.

\begin{proof}[Proof of the Theorem \ref{THML2L2} under Assumption \ref{matsym}, formula \eqref{bend23}] \RRR
Our first observation is that the estimates \eqref{bend23} can be rewritten as 
\begin{equation}
\label{bendl2l2estimatescomponents}
    \begin{aligned}
        \Bigg\Vert P_i \left( \frac{1}{\varepsilon^\gamma}\mathcal{A}_\varepsilon + I \right)^{-1} S_{\!\varepsilon^\delta} \vect f&-\pi_i \left(\frac{1}{\varepsilon^{\gamma-2}}\mathcal{A}^{\rm bend} + I\right)^{-1}\mathcal{M}^{\rm bend}_\varepsilon \Xi_\varepsilon S_{\!\varepsilon^\delta}  \vect f  \Bigg\Vert_{L^2(\omega\times \mathbb{R})} \\[0.6em] 
        &\leq C \varepsilon^{\tfrac{\gamma + 2}{4}}\max\Bigl\{\varepsilon^{\tfrac{\gamma + 2}{4}-\delta}, 1\Bigr\} \bigl\Vert  \vect f \bigr\Vert _{L^2(\omega\times \mathbb{R};\R^3)}, \quad i = 1,2, \\[0.6em]
         \Bigg\Vert P_3 \left( \frac{1}{\varepsilon^\gamma}\mathcal{A}_\varepsilon + I \right)^{-1} S_{\!\varepsilon^\delta} \vect f &+\varepsilon  
     \widehat{x}\cdot\frac{d}{d x_3}\left( \frac{1}{\varepsilon^{\gamma-2}}\mathcal{A}^{\rm bend} + I\right)^{-1}\mathcal{M}^{\rm bend}_\varepsilon \Xi_\varepsilon S_{\!\varepsilon^\delta} \vect f \Bigg\Vert_{L^2(\omega\times \mathbb{R})} \\[0.6em] 
     &\leq C \varepsilon^{\tfrac{\gamma + 2}{2}}\max\Bigl\{\varepsilon^{\tfrac{\gamma + 2}{4}-\delta}, 1\Bigr\} \bigl\Vert  \vect f \bigr\Vert _{L^2(\omega\times \mathbb{R};\R^3)}.
    \end{aligned}
\end{equation}
Our goal here is therefore to prove \eqref{bendl2l2estimatescomponents}. 

By virtue of \eqref{bendingestimate1} and Remark \ref{replacement_rem} we have the following norm-resolvent estimates:
\begin{align*}
    &\left\Vert P_i\left(\frac{1}{{\chi^4}}\mathcal{A}_\chi + I\right)^{-1} S_{\!|\chi|}\vect f- \pi_i\bigl({\mathbb A}^{\rm bend}  + I\bigr)^{-1}\widetilde{\mathcal{M}}_\chi^{\rm bend} S_{\!|\chi|}\vect f \right\Vert_{L^2\to L^2} \leq C|\chi| \left\Vert \vect f \right\Vert_{L^2(\omega\times Y;\C^3)}, \quad i=1,2, \\[0.8em]
    &\left\Vert P_3\left(\frac{1}{{\chi^4}}\mathcal{A}_\chi +I \right)^{-1}S_{\!|\chi|} \vect f +{\rm i}\chi
    \widehat{x}\cdot\bigl({\mathbb A}^{\rm bend}+I \bigr)^{-1}\widetilde{\mathcal{M}}_\chi^{\rm bend} S_{\!|\chi|}\vect f \right\Vert_{L^2\to L^2} \leq C{\chi^2} \left\Vert \vect f \right\Vert_{L^2(\omega \times Y;\C^3)}. 
\end{align*}
For each fixed $\varepsilon>0$, $\chi \neq 0$ we define the function 
\begin{equation}\label{definition_f_analytic}
    h_{\varepsilon,\chi}(z) := \left(\frac{{\chi^4}}{\varepsilon^{\gamma+2}}z + 1 \right)^{-1}, \quad  \Re(z) > 0.
\end{equation}
Similar to an earlier argument, for every $\eta > 0$ the function  $h_{\varepsilon,\chi}$ is bounded on the half-plane $\left\{z \in \C, \Re(z) \geq \eta\right\}:$ 
 \begin{equation}
     \bigl|h_{\varepsilon,\chi}(z)\bigr| \leq  C(\eta)\left(\max\left\{\frac{{\chi^4}}{\varepsilon^{\gamma+2}}, 1\right\}\right)^{-1}.
 \end{equation}
Due to the bounds on both the eigenvalues $\lambda_1^\chi$, $\lambda_2^\chi$ of $\mathcal{A}_\chi$ of the order ${\chi^4}$, and the eigenvalues $\kappa_1^\chi$, $\kappa_2^\chi$ of the matrix ${\mathbb A}_\chi^{\rm bend} $, there exists a closed contour $\Gamma \subset \left\{z\in \C, \Re(z)>0\right\}$ and a constant $\mu> 0$ such that for all  $\chi \in [-\mu,\mu]\setminus \left\{0\right\}$ the following properties hold:
\begin{itemize}
\item $\Gamma$ encloses the two smallest eigenvalues of both the  operators $\chi^{-4}\mathcal{A}_\chi$ and $\chi^{-4}{\mathbb A}_\chi^{\rm bend} $. 

\item $\Gamma$ does not enclose any other eigenvalue (of higher order). 

\item $\exists \rho_0 > 0$, $\inf_{z\in \Gamma}|z-\lambda_i^\chi|\geq \rho_0$, $\inf_{z\in \Gamma}|z-\kappa_i^\chi|\geq \rho_0$,  $i=1,2.$ 
\end{itemize}
Due to the fact that $h_{\varepsilon,\chi}$ is analytic on the right half-plane,  the Cauchy integral formula yields
$$ h_{\varepsilon,\chi}(\mathcal{A}_\chi)P_\Gamma:= \frac{1}{2\pi{\rm i}} \oint_\Gamma h_{\varepsilon,\chi}(z)(zI-\mathcal{A}_\chi)^{-1}dz, $$ 
where, similarly to the stretching case,  $P_\Gamma$ is the projection operator onto the eigenspace spanned  by  eigenfunctions corresponding  to the  eigenvalues enclosed  by the contour $\Gamma$ just defined.
 Noting that 
\begin{equation*}
    \frac{1}{\varepsilon^{\gamma+2}}\mathcal{A}_\chi + I = \frac{{\chi^4}}{\varepsilon^{\gamma+2}}\left(\frac{1}{{\chi^4}}\mathcal{A}_\chi \right) + I, \quad \quad 
    S_{\!\varepsilon^\delta} \vect f = S_{\!|\chi|} S_{\varepsilon^\delta/|\chi|} \vect f
\end{equation*} 
 and 
\begin{equation}\label{estimatef_bend}
    \bigl\|S_{\varepsilon^\delta/|\chi|} \vect f\bigr\|_{L^2(\omega\times Y;\C^3)}\leq \max \left\{1, \frac{|\chi|}{\varepsilon^{\delta}}  \right\} \|\vect f\|_{L^2(\omega \times Y;\C^3)},
\end{equation}
 we write 
\begin{equation}
    \left(\frac{1}{\varepsilon^{\gamma + 2}}\mathcal{A}_\chi + I\right)^{-1}P_\Gamma S_{\!\varepsilon^\delta} \vect f = \frac{1}{2\pi{\rm i}} \oint_\Gamma h_{\varepsilon,\chi}(z)\left(zI-\frac{1}{{\chi^4}}\mathcal{A}_\chi\right)^{-1} S_{\!|\chi|} S_{\varepsilon^\delta/|\chi|} \vect f dz.
\end{equation}
In the same fashion as before,  we can drop the operators $P_{\Gamma}$ from estimates. 
As a result, $i=1,2$ we have
\begin{align*}
& \left\Vert P_i\left(\frac{1}{\varepsilon^{\gamma + 2}}\mathcal{A}_\chi+I\right)^{-1}S_{\!\varepsilon^\delta} \vect f - \pi_i\left(\frac{\chi^4}{\varepsilon^{\gamma + 2}}{\mathbb A}^{\rm bend} +I\right)^{-1}\widetilde{\mathcal{M}}_\chi^{\rm bend}S_{\!\varepsilon^\delta} \vect f \right\Vert_{L^2(\omega \times Y)}  
\\[0.9em]
 &\leq \frac{1}{2\pi} \oint_{\Gamma} \bigl|h_{\varepsilon,\chi}(z)\bigr|\left\Vert P_i\left(zI -\frac{1}{{\chi^4}}\mathcal{A}_\chi\right)^{-1}S_{\!|\chi|} S_{\varepsilon^\delta/|\chi|} \vect f - \pi_i\bigl(zI-{\mathbb A}^{\rm bend}\bigr)^{-1}\widetilde{\mathcal{M}}_\chi^{\rm bend} S_{\!|\chi|} S_{\varepsilon^\delta/|\chi|} \vect f \right\Vert_{L^2(\omega \times Y)} dz 
 \\[0.9em]
 &\leq C|\chi| \left(\max\left\{\frac{{\chi^4}}{\varepsilon^{\gamma + 2}}, 1\right\}\right)^{-1}
 \bigl\| S_{\varepsilon^\delta/|\chi|} \vect f\bigr\|_{L^2(\omega \times  Y;\C^3)}\leq C|\chi| \left(\max\left\{\frac{{\chi^4}}{\varepsilon^{\gamma + 2}}, 1\right\}\right)^{-1}\max\left\{\frac{|\chi|}{\varepsilon^{\delta}}, 1\right\}\|\vect f\|_{L^2(\omega \times  Y;\C^3)} \\[0.9em] 
 &\leq C \varepsilon^{\tfrac{\gamma + 2}{4}}\max\Bigl\{\varepsilon^{\tfrac{\gamma + 2}{4}-\delta}, 1\Bigr\} \|\vect f\|_{L^2(\omega \times Y;\C^3)},
\end{align*}
where the bound is optimal when ${\chi^4}\approx \varepsilon^{\gamma + 2}$.
For the third component, we have
\begin{align*}
&\left\Vert P_3\left(\frac{1}{\varepsilon^{\gamma+2}}\mathcal{A}_\chi+I\right)^{-1} S_{\!\varepsilon^\delta} \vect f +  {\rm i} \chi
\widehat{x}\cdot\left(\frac{\chi^4}{\varepsilon^{\gamma+2}}{\mathbb A}^{\rm bend} +I\right)^{-1}\widetilde{\mathcal{M}}_\chi^{\rm bend} S_{\varepsilon^{\delta}} \vect f \right\Vert_{L^2(\omega \times Y)}  
\\[0.8em]
& \leq \frac{1}{2\pi} \oint_{\Gamma} \bigl|h_{\varepsilon,\chi}(z)\bigr|\left\Vert P_3\left(zI -\frac{1}{{\chi^4}}\mathcal{A}_\chi\right)^{-1} S_{\!|\chi|} S_{\varepsilon^\delta/|\chi|} \vect f +  {\rm i}\chi
\widehat{x}\cdot\bigl(zI-{\mathbb A}^{\rm bend} \bigr)^{-1}\widetilde{\mathcal{M}}_\chi^{\rm bend} S_{\!|\chi|} S_{\varepsilon^\delta/|\chi|} \vect f \right\Vert_{L^2(\omega \times  Y)} dz 
 \\[0.8em]
& \leq  C{\chi^2} \left(\max\left\{\frac{{\chi^4}}{\varepsilon^{4}}, 1\right\}\right)^{-1}\max\left\{\frac{|\chi|}{\varepsilon^{\delta}}, 1\right\}\|\vect f\|_{L^2(\omega \times Y;\C^3)} 
\leq C \varepsilon^{\tfrac{\gamma+2}{2}} \max\Bigl\{\varepsilon^{\tfrac{\gamma + 2}{4}-\delta}, 1\Bigr\}\|\vect f\|_{L^2(\omega \times Y;\C^3)}.
\end{align*}
Since
$$
\left(\frac{1}{\varepsilon^{\gamma-2}}\mathcal{A}^{\rm bend} + I \right)^{-1} = \mathfrak{G}_\varepsilon^{-1} \left( \frac{\chi^4}{\varepsilon^{\gamma+2}}{\mathbb A}^{\rm bend}  + I \right)^{-1} \mathfrak{G}_\varepsilon,\qquad 
\left( \frac{1}{\varepsilon^\gamma}\mathcal{A}_\varepsilon + I \right)^{-1} = \mathfrak{G}_\varepsilon^{-1}\left( \frac{1}{\varepsilon^{\gamma+2}}\mathcal{A}_\chi + I \right)^{-1} \mathfrak{G}_\varepsilon, $$
 it follows that 
\begin{align}
	\label{trans_back_bend12}
	&\begin{aligned}	
&\mathfrak{G}_\varepsilon^{-1}\left(P_i \left( \frac{1}{\varepsilon^{\gamma+2}}\mathcal{A}_\chi + I \right)^{-1} - \pi_i\left(\frac{\chi^4}{\varepsilon^{\gamma+2}}{\mathbb A}^{\rm bend}  + I \right)^{-1}\widetilde{\mathcal{M}}_\chi^{\rm bend}\right)  \mathfrak{G}_\varepsilon  \\[0.3em]
&\hspace{8em}=P_i\left( \frac{1}{\varepsilon^\gamma}\mathcal{A}_\varepsilon + I \right)^{-1} - \pi_i\left( \frac{1}{\varepsilon^{\gamma - 2}}\mathcal{A}^{\rm bend} + I \right)^{-1}\mathcal{M}^{\rm bend}_\varepsilon \Xi_\varepsilon,\qquad i=1,2, 
\end{aligned}
\\
	\label{transform_back_bend3}
&\begin{aligned}	
&\mathfrak{G}_\varepsilon^{-1}\left(P_3 \left( \frac{1}{\varepsilon^{\gamma + 2}}\mathcal{A}_\chi + I \right)^{-1} + {\rm i}\chi\widehat{x} 
\cdot \left(\frac{\chi^4}{\varepsilon^{\gamma + 2}}{\mathbb A}^{\rm bend}  + I \right)^{-1}\widetilde{\mathcal{M}}_\chi^{\rm bend}\right)  \mathfrak{G}_\varepsilon \\[0.3em]
&\hspace{8em}=P_3\left( \frac{1}{\varepsilon^\gamma}\mathcal{A}_\varepsilon + I \right)^{-1} + \varepsilon 
\widehat{x}\cdot\frac{d}{dx_3}\left(\frac{1}{\varepsilon^{\gamma - 2}} \mathcal{A}^{\rm bend} + I \right)^{-1}\mathcal{M}^{\rm bend}_\varepsilon \Xi_\varepsilon.
\end{aligned}
\end{align}
In order to establish the formula \eqref{bendl2l2estimatescomponents} we use the fact that the Gelfand transform is an isometry. 
\end{proof}

Using the notation introduced in Remark \ref{newnotationformomentums}, we can reformulate the estimates \eqref{stretchl2l2estimatescomponents} and \eqref{bendl2l2estimatescomponents} in a more compact fashion, namely as the operator-norm estimates \eqref{stretch23} and \eqref{bend23} of Theorem \ref{THML2L2}.

\begin{remark}
It is clear that, due to \eqref{noforcescaling}, in the case $\delta = 0$, in addition to the estimate
\begin{equation}
    \left\Vert 
	P_i \left( 
			\left( \frac{1}{\varepsilon^\gamma}\mathcal{A}_\varepsilon + I \right)^{-1}\bigg|_{L^2_{\rm bend}} 
- (\mathcal{M}^{\rm bend}_\varepsilon)^*\left(\frac{1}{\varepsilon^{\gamma-2}}\mathcal{A}^{\rm bend} + I \right)^{-1}\mathcal{M}^{\rm bend}_\varepsilon \Xi_\varepsilon 
		\right) 
 \right\Vert_{L^2 \to L^2} \leq  
 { C} \left\{ \begin{array}{ll}
         \varepsilon^{\tfrac{\gamma + 2}{4}}, & i = 1,2,\\[0.7em]
         \varepsilon^{\tfrac{\gamma + 2}{2}}, & i = 3,\end{array} \right.
\end{equation}
one also has a norm-resolvent estimate in the absence of out-of-line force terms:  
\begin{equation}
    \left\Vert 
	P_i 
		\left( 
			\left( \frac{1}{\varepsilon^\gamma}\mathcal{A}_\varepsilon + I \right)^{-1}\bigg|_{L^2_{\rm bend}} 
- (\mathcal{M}^{\rm bend}_\varepsilon)^*\left(\frac{1}{\varepsilon^{\gamma-2}}\mathcal{A}^{\rm bend} + I \right)^{-1}\mathcal{M}^{\rm bend}_\varepsilon \Xi_\varepsilon S_\infty 
		\right) 
 \right\Vert_{L^2 \to L^2} \leq  
  { C}\left\{ \begin{array}{ll}
         \varepsilon^{\tfrac{\gamma + 2}{4}}, & i = 1,2,\\[0.7em]
         \varepsilon^{\tfrac{\gamma + 2}{2}}, & i = 3.\end{array} \right.
\end{equation}
   
\end{remark}

We next show that the smoothing operator $\Xi_\varepsilon$
 appearing in the above norm-resolvent estimates can be dropped without affecting the estimates.  For the definition of the matrix $S_\infty,$ which appears in one of the two estimates below, we recall \eqref{nak70}.
\begin{corollary}
\label{absenceofforceterms}
Suppose that Assumption \ref{matsym} 
holds and that the spectral scaling parameter satisfies the inequality  $\gamma > -2.$ Then there exists $C>0$ such that for every $\varepsilon> 0$ one has 
\begin{align}
&\left\Vert \left(\frac{1}{\varepsilon^\gamma} \mathcal{A}_\varepsilon + I \right)^{-1}\bigg|_{L^2_{\rm stretch}} - (\mathcal{M}^{\rm stretch})^*\left(\frac{1}{\varepsilon^\gamma}\mathcal{A}^{\rm stretch} + \mathfrak{C}^{\rm stretch} \right)^{-1}\mathcal{M}^{\rm stretch}  \right\Vert_{L^2 \to L^2} \leq C\varepsilon^{\tfrac{\gamma + 2}{2}},\nonumber\\[0.8em]
&\left\Vert P_i\left(\left( \frac{1}{\varepsilon^\gamma}\mathcal{A}_\varepsilon + I \right)^{-1}\bigg|_{L^2_{\rm bend}} 
- (\mathcal{M}^{\rm bend}_\varepsilon)^*\left(\frac{1}{\varepsilon^{\gamma-2}}\mathcal{A}^{\rm bend} + I \right)^{-1}\mathcal{M}^{\rm bend}_\varepsilon  S_\infty 
		\right) 
 \right\Vert_{L^2 \to L^2} \leq  
  { C}\left\{ \begin{array}{ll}
         \varepsilon^{\tfrac{\gamma + 2}{4}}, & i = 1,2,\\[0.55em]
         \varepsilon^{\tfrac{\gamma + 2}{2}}, & i = 3,\end{array} \right.\quad\  \label{bendingestimatesinfinity}
\end{align}
\end{corollary}
\begin{proof}

 The application of the Fourier transform to the limit resolvent in the stretching case yields 
\begin{align*}
          &{\mathcal F}\left[(\mathcal{M}^{\rm stretch})^*\left(\frac{1}{\varepsilon^\gamma}\mathcal{A}^{\rm stretch} + \mathfrak{C}^{\rm stretch} \right)^{-1}\mathcal{M}^{\rm stretch}    \left(I -\Xi_\varepsilon \right) \vect f \right](\xi) \\[0.4em]
          &\hspace{10em}=(\mathcal{M}^{\rm stretch})^*\left(\frac{\xi^2}{\varepsilon^\gamma}{\mathbb A}^{\rm stretch} + \mathfrak{C}^{\rm stretch} \right)^{-1}\mathcal{M}^{\rm stretch} \,\mathcal{F}[\vect f](\xi)\mathbbm{1}_{\left\langle -\infty, -(2\varepsilon)^{-1}] \cup [(2\varepsilon)^{-1}, \infty \right\rangle  }(\xi),
 \end{align*}
where, as before, $\mathcal{F}$ stands for the Fourier transform. Furthermore, for $|\xi|>1/(2\varepsilon)$ and $\gamma > -2$ we have
 \begin{equation}
     \left( \frac{\xi^2}{\varepsilon^\gamma} {\mathbb A}^{\rm stretch} + \mathfrak{C}^{\rm stretch} \right){\vect m} \cdot {\vect m}\geq \frac{\xi^2}{\varepsilon^{\gamma}}C |{\vect m}|^2 \geq \frac{C}{\varepsilon^{\gamma + 2}}|{\vect m}|^2\qquad\forall{\vect m}\in{\mathbb R}^2,
 \end{equation}
 and hence
 \begin{equation}
      \left| \left( \frac{\xi^2}{\varepsilon^\gamma} {\mathbb A}^{\rm stretch} + \mathfrak{C}^{\rm stretch} \right)^{-1} \right| \leq C\varepsilon^{\gamma + 2}.
 \end{equation}
Combining the above, we obtain
\begin{equation}
    \begin{split}
        \left\Vert{\mathcal F}\left[(\mathcal{M}^{\rm stretch})^*\left(\frac{1}{\varepsilon^\gamma}\mathcal{A}^{\rm stretch} + \mathfrak{C}^{\rm stretch} \right)^{-1}\mathcal{M}^{\rm stretch}    \left(I -\Xi_\varepsilon \right) \vect f \right]\right\Vert_{L^2} \leq C\varepsilon^{\gamma + 2}\bigl\Vert \,\mathcal{F}[\vect f]\bigr\Vert_{L^2},
    \end{split}
\end{equation}
 so $\Xi_\varepsilon$ can be removed from \eqref{stretch23}. Similarly to the stretching case, one can eliminate the smoothing operator from the norm-resolvent estimate in the case of absence of out-of-line force terms. 
 For $i = 1,2$ one has 
 \begin{equation}
     	P_i 
		\left( (\mathcal{M}^{\rm bend}_\varepsilon)^*\left(\frac{1}{\varepsilon^{\gamma-2}}\mathcal{A}^{\rm bend} + I \right)^{-1}\mathcal{M}^{\rm bend}_\varepsilon ( I -\Xi_\varepsilon) S_\infty \right) = \pi_i \left(\frac{1}{\varepsilon^{\gamma-2}}\mathcal{A}^{\rm bend} + I \right)^{-1} ( I -\Xi_\varepsilon).
 \end{equation}
It follows that
  \begin{equation}
 \begin{split}
          {\mathcal F}\left[\pi_i \left(\frac{1}{\varepsilon^{\gamma-2}}\mathcal{A}^{\rm bend} + I \right)^{-1} ( I -\Xi_\varepsilon) \vect f\right](\xi)  = \left(\frac{\xi^4}{\varepsilon^{\gamma-2}}{\mathbb A}^{\rm bend} + I \right)^{-1}{\mathcal F}[\vect f](\xi)  \mathbbm{1}_{\left\langle -\infty, -\frac{1}{2\varepsilon}] \cup [ \frac{1}{2\varepsilon}, \infty \right\rangle  }(\xi).
 \end{split}
 \end{equation}
 Furthermore, for $|\xi|>1/(2\varepsilon)$ and $\gamma>-2$ we have
 \begin{equation}
     \left( \frac{\xi^4}{\varepsilon^{\gamma-2}} {\mathbb A}^{\rm bend} + I \right){\vect m}\cdot{\vect m}\geq \frac{\xi^4}{\varepsilon^{\gamma-2}}C |{\vect m}|^2 \geq \frac{C}{\varepsilon^{\gamma + 2}}|{\vect m}|^2\qquad\forall{\vect m}\in{\mathbb R}^2.
 \end{equation}
 Finally, we obtain
\begin{equation}
    \begin{split}
        \left\Vert 	P_i\,{\mathcal F}\left[(\mathcal{M}^{\rm bend}_\varepsilon)^*\left(\frac{1}{\varepsilon^{\gamma-2}}\mathcal{A}^{\rm bend} + I \right)^{-1}\mathcal{M}^{\rm bend}_\varepsilon ( I -\Xi_\varepsilon) S_\infty \vect f\right]\right\Vert_{L^2} \leq C\varepsilon^{\gamma + 2}\left\Vert{\mathcal F}[\vect f] \right\Vert_{L^2},\qquad i = 1,2.
    \end{split}
\end{equation}
Similarly, one obtains
\begin{equation}
    P_3 
		\left( (\mathcal{M}^{\rm bend}_\varepsilon)^*\left(\frac{1}{\varepsilon^{\gamma-2}}\mathcal{A}^{\rm bend} + I \right)^{-1}\mathcal{M}^{\rm bend}_\varepsilon ( I -\Xi_\varepsilon) S_\infty \right) = \varepsilon  
         \widehat{x}\cdot\frac{d}{d x_3} \left(\frac{1}{\varepsilon^{\gamma-2}}\mathcal{A}^{\rm bend} + I \right)^{-1} ( I -\Xi_\varepsilon),
\end{equation}
where 
\begin{align}
      {\mathcal F}\bigg[\varepsilon  
     \widehat{x}&\cdot\frac{d}{d x_3} \left(\frac{1}{\varepsilon^{\gamma-2}}\mathcal{A}^{\rm bend} + I \right)^{-1} ( I -\Xi_\varepsilon) \vect f \biggr](\xi)
     \\[0.5em]
     &\hspace{10em}
     = \varepsilon \xi 
     \widehat{x}\cdot\left(\frac{\xi^4}{\varepsilon^{\gamma-2}}{\mathbb A}^{\rm bend} + I \right)^{-1}{\mathcal F}[\vect f](\xi)  \mathbbm{1}_{\left\langle -\infty, -\frac{1}{2\varepsilon}] \cup [ \frac{1}{2\varepsilon}, \infty \right\rangle}(\xi).
 \end{align} 
 \end{proof}
The second bound in Corollary \ref{absenceofforceterms} can be simplified as follows. 
\begin{corollary}
\label{corr3}
Suppose that Assumption \ref{matsym} 
holds. Let $\gamma > -2$ be the spectral scaling parameter. Then there exists $C>0$ such that for all $\varepsilon>0$ one has
\begin{equation}
\label{newbendingm0estimate}
     \left\Vert 
		\left( 
			\left( \frac{1}{\varepsilon^\gamma}\mathcal{A}_\varepsilon + I \right)^{-1}\bigg|_{L^2_{\rm bend}} 
- (\mathcal{M}^{\rm bend}_0)^*\left(\frac{1}{\varepsilon^{\gamma-2}}\mathcal{A}^{\rm bend} + I \right)^{-1}\mathcal{M}^{\rm bend}_0   
		\right) 
 \right\Vert_{L^2 \to L^2} \leq  C\varepsilon^{\tfrac{\gamma + 2}{4}}.
\end{equation}
\end{corollary}
\begin{proof}
It follows from the ellipticity of the operator $\mathcal{A}^{\rm bend}$ that the solution $\vect b \in H^2(\R;\R^2)$ of 
\begin{equation}
    \left(\frac{1}{\varepsilon^{\gamma-2}}\mathcal{A}^{\rm bend} + I \right) \vect b = \vect g
\end{equation}
satisfies the estimates 
\begin{equation}
   \bigl\Vert  \vect b \bigr\Vert _{L^2(\R;\R^2)} \leq C\bigl\Vert  \vect g \bigr\Vert _{L^2(\R;\R^2)}, \qquad  
    \bigl\Vert \nabla^2 \vect b\bigr\Vert_{L^2(\R;\R^{2 \times 2 \times 2})} \leq C\varepsilon^{\tfrac{\gamma -2}{2}} \bigl\Vert  \vect g\bigr\Vert _{L^2(\R;\R^2)}.
\end{equation}
Using the interpolation inequality
\begin{equation}
    \bigl\Vert  \nabla \vect b \bigr\Vert _{L^2(\R;\R^{2 \times 2})}^2 \leq C \bigl\Vert \nabla^2\vect b \bigr\Vert_{L^2(\R;\R^{2 \times 2 \times 2})}\Vert\vect b \Vert_{L^2(\R;\R^2)},
\end{equation}
we clearly have
\begin{equation}
    \bigl\Vert \varepsilon \nabla \vect b \bigr\Vert_{L^2(\omega;\R^{2 \times 2})} \leq C\varepsilon^{\tfrac{\gamma +2}{4}} \bigl\Vert  \vect g\bigr\Vert _{L^2(\omega;\R^2)}.
\end{equation}
Hence, replacing $\mathcal{M}^{\rm bend}_\varepsilon$ by $\mathcal{M}^{\rm bend}_0$ in \eqref{bendingestimatesinfinity} maintains an order $\varepsilon^{(\gamma +2)/4}$ bound on the approximation error.
\end{proof}

\subsection{$L^2 \to H^1$ norm resolvent estimates}
\label{L2toH1}

In order to state the results we define the following operators which take the zero-order terms to the associated first-order corrector terms:
\begin{equation*}
\mathcal{B}_{1,\rm stretch}^{\chi, \rm corr} : \C^2 \to H^{\rm stretch}, \quad \mathcal{B}_{1,\rm bend}^{\chi, \rm corr} : \C^2 \to H^{\rm bend}, 
\end{equation*}
 via: 
\begin{align}
	\label{B1_str}
\mathcal{B}_{1,\rm stretch}^{\chi, \rm corr}\begin{bmatrix} m_3 \\ m_4\end{bmatrix}&:= {\vect u_1}\in H^{\rm stretch}, \qquad 
(\simgrad)^* {\mathbb A} \simgrad {\vect u_1} = -(\simgrad)^* {\mathbb A} \Lambda_{\chi,m_3,m_4}^{\rm stretch},\\[0.6em]
\label{B1_bend}
  \mathcal{B}_{1,\rm bend}^{\chi, \rm corr}\begin{bmatrix} m_1 \\ m_2\end{bmatrix}&:= {\vect u_1}\in H^{\rm bend}, \qquad
(\simgrad)^* {\mathbb A} \simgrad {\vect u_1} = -(\simgrad)^* {\mathbb A} \Lambda_{\chi,m_1,m_2}^{\rm bend}, 
\end{align}
Thus, the first-order corrector operators depending on the spectral parameter $z \in \C$ can be defined for each $\chi \in [-\pi,\pi)\setminus\{0 \}$ with the following formulae: 
\begin{equation*}
   \mathcal{A}_{\chi,\rm corr}^{\rm stretch}(z):= \mathcal{B}_{1,\rm stretch}^{\chi, \rm corr}\bigl(z\mathfrak{C}^{\rm stretch}-{\mathbb A}^{\rm stretch}\bigr)^{-1}\widetilde{\mathcal{M}}^{\rm stretch}, \quad \mathcal{A}_{\chi,\rm corr}^{\rm bend}(z):=\mathcal{B}_{1,\rm bend}^{\chi, \rm corr}\bigl(zI-{\mathbb A}^{\rm bend}\bigr)^{-1}\widetilde{\mathcal{M}}_\chi^{\rm bend}.
\end{equation*}
Next we define the rescaled versions:
\begin{equation}
    \mathcal{A}_{\chi,\varepsilon,\rm corr}^{\rm stretch}=\oint_{\Gamma^{\rm stretch}} g_{\varepsilon,\chi}(z) \mathcal{A}_{\chi, \rm corr}^{\rm stretch}(z)dz,\qquad \mathcal{A}_{\chi,\varepsilon,\rm corr}^{\rm bend}=\oint_{\Gamma^{\rm bend}} h_{\varepsilon,\chi}(z)\mathcal{A}_{\chi, \rm corr}^{\rm bend}(z)dz, \quad \varepsilon > 0,
\end{equation}
where $\Gamma^{\rm bend}$, $\Gamma^{\rm stretch}$ are contours which uniformly enclose the scaled eigenvalues of ${\mathbb A}_\chi^{\rm bend} $, ${\mathbb A}_\chi^{\rm stretch} $, respectively. 
Notice here that we have:
\begin{equation*}
   \mathcal{A}_{\chi,\varepsilon,\rm corr}^{\rm stretch}:= \mathcal{B}_{1,\rm stretch}^{\chi, \rm corr}\left( \frac{\chi^2}{\varepsilon^{\gamma+2}}{\mathbb A}^{\rm stretch} + \mathfrak{C}^{\rm stretch}\right)^{-1} \widetilde{\mathcal{M}}^{\rm stretch}, \quad \mathcal{A}_{\chi,\varepsilon,\rm corr}^{\rm bend}:=\mathcal{B}_{1,\rm bend}^{\chi, \rm corr}\left( \frac{\chi^4}{\varepsilon^{\gamma+2}}{\mathbb A}^{\rm bend} + I\right)^{-1}\widetilde{\mathcal{M}}_\chi^{\rm bend}.
\end{equation*}
Finally, we are able to define the following corrector operators: 
\begin{equation}\label{trans_back_corr1_str}
    \mathcal{A}^{\rm \rm corr}_{\rm stretch}(\varepsilon)=\mathfrak{G}_\varepsilon^{-1} \mathcal{A}_{\chi,\varepsilon,\rm corr}^{\rm stretch} \mathfrak{G}_\varepsilon,\qquad \mathcal{A}^{\rm \rm corr}_{\rm bend}(\varepsilon)=\mathfrak{G}_\varepsilon^{-1}  \mathcal{A}_{\chi,\varepsilon,\rm corr}^{\rm bend} \mathfrak{G}_\varepsilon, \qquad \varepsilon > 0.
\end{equation}
Let us start with the stretching case.

\begin{proof}[Proof of the Theorem \ref{l2h1theorem} under Assumption \ref{matsym}, formula \eqref{precision_stretch}]
Our intention is to prove the following estimates:
\begin{equation}
    \begin{aligned}
        &\left\Vert P_1\left(\frac{1}{\varepsilon^\gamma} \mathcal{A}_\varepsilon + I \right)^{-1}\vect f - x_2\pi_1 \left(\frac{1}{\varepsilon^\gamma}\mathcal{A}^{\rm stretch} + \mathfrak{C}^{\rm stretch}\right)^{-1}\mathcal{M}^{\rm stretch} \Xi_\varepsilon \vect f  -P_1 \mathcal{A}_{\rm stretch}^{\rm \rm corr}(\varepsilon) \vect f\right\Vert_{H^1(\omega\times \mathbb{R})}\\[0.6em]
        &\hspace{27em}\leq C \max\left\{\varepsilon^{\gamma + 1}, \varepsilon^{\tfrac{\gamma + 2}{2}} \right\} \bigl\Vert  \vect f \bigr\Vert _{L^2(\omega\times \mathbb{R};\R^3)}, \\[0.7em]
        &\left\Vert P_2\left(\frac{1}{\varepsilon^\gamma} \mathcal{A}_\varepsilon + I \right)^{-1}\vect f + x_1 \pi_1 \left(\frac{1}{\varepsilon^\gamma}\mathcal{A}^{\rm stretch} + \mathfrak{C}^{\rm stretch}\right)^{-1}\mathcal{M}^{\rm stretch}\Xi_\varepsilon \vect f - P_2 \mathcal{A}_{\rm stretch}^{\rm \rm corr}(\varepsilon)\vect f \right\Vert_{H^1(\omega\times \mathbb{R})}\\[0.7em] &\hspace{27em}\leq C \max\Bigl\{\varepsilon^{\gamma + 1}, \varepsilon^{\tfrac{\gamma + 2}{2}}\Bigr\} \bigl\Vert  \vect f \bigr\Vert _{L^2(\omega\times \mathbb{R};\R^3)}, \\[0.7em]
        &\left\Vert P_3\left(\frac{1}{\varepsilon^\gamma} \mathcal{A}_\varepsilon + I \right)^{-1}\vect f -  \pi_2 \left(\frac{1}{\varepsilon^\gamma}\mathcal{A}^{\rm stretch} + \mathfrak{C}^{\rm stretch}\right)^{-1}\mathcal{M}^{\rm stretch}\Xi_\varepsilon \vect f -P_3 \mathcal{A}_{\rm stretch}^{\rm \rm corr}(\varepsilon)\vect f \right\Vert_{H^1(\omega\times \mathbb{R})}\\[0.7em]
        &\hspace{27em}\leq C \max\Bigl\{\varepsilon^{\gamma + 1}, \varepsilon^{\tfrac{\gamma + 2}{2}} \Bigr\} \bigl\Vert  \vect f \bigr\Vert _{L^2(\omega\times\mathbb{R};\R^3)}.
    \end{aligned}
\label{L2H1stretch}
\end{equation}
which is equivalent to formula \eqref{precision_stretch}. \CCC Again, we will only prove the first estimate, since the others go in an analogous way. 
In order to prove the required  $H^1$ estimate, the first resolvent estimate in \eqref{stretch_estimates} does not suffice, the reason being that, the Gelfand pullback would ruin the order of the estimate in the third variable. On the other hand, we do not need the whole expression in the second estimate from \eqref{stretch_estimates}, either. This is because we can neglect the element $\vect u_0^{(1)}$ in the $H^1(\omega \times Y;\C^3)$ norm. Indeed, from \eqref{u01_stretch} we note that $\vect u_0^{(1)}$ does not depend on $y,$ and 
\begin{align*}
    \bigl\|\vect u_0^{(1)}\bigr\|_{L^2(\omega \times Y;\C^3)}&\leq C |\chi|\|\vect f\|_{L^2(\omega \times Y;\C^3)},  \quad 
    \bigl\| \partial_{x_\alpha}\vect u_0^{(1)}\bigr\|_{L^2(\omega \times Y;\C^3)}\leq C |\chi|\|\vect f\|_{L^2(\omega \times Y;\C^3)}, \quad \alpha =1,2.
\end{align*}
Hence, the following estimates are obtained from \eqref{stretch_estimates}:
\begin{align}
    \left\Vert \partial_{x_\alpha}\left(
    {\vect u}- 
\begin{bmatrix} m_3  x_2  \\[0.25em] -m_3 x_1 \\[0.25em] m_4  \end{bmatrix} -{\vect u}_1 
\right) \right\Vert_{L^2(\omega\times Y,\C^3)}&\leq C|\chi|\bigl\Vert  \vect f \bigr\Vert _{L^2(\omega\times Y;\C^3)},\nonumber\\[0.6em]
  \left\Vert \partial_y\left(\vect{u}
    -\begin{bmatrix} m_3  x_2  \\[0.25em] -m_3 x_1 \\[0.25em] m_4  \end{bmatrix} - {\vect u}_1
    \right) \right\Vert_{L^2(\omega\times Y,\C^3)}&\leq C{\chi^2}\bigl\Vert  \vect f \bigr\Vert _{L^2(\omega\times Y;\C^3)}.\label{estimywithoutu0}
\end{align}
We will focus on the norm of derivative with respect to $y$. For this,  we require a higher-order estimate in $|\chi|.$ We then obtain the desired estimate by applying the inverse of scaled Gelfand transform. Rewriting the estimate \eqref{estimywithoutu0}, we deduce that
\begin{align*}
    & \left\Vert \partial_y \left( P_1\left( \frac{1}{{\chi^2}} \mathcal{A}_\chi + I \right)^{-1} - x_2 \pi_1\bigl({\mathbb A}^{\rm stretch}  + \mathfrak{C}^{\rm stretch}\bigr)^{-1}\widetilde{\mathcal{M}}^{\rm stretch} -P_1  \mathcal{A}_{\chi,\rm corr}^{\rm stretch} \right)\vect f \right\Vert_{L^2(\omega\times Y)}\\[0.6em] 
    &\hspace{25em}\leq C{\chi^2} \bigl\Vert  \vect f \bigr\Vert _{L^2(\omega\times Y;\C^3)}.
\end{align*}
We use the approach introduced when deriving the $L^2 \to L^2$ operator-norm estimates. To this end, define the function $g_{\varepsilon,\chi}(z)$ as in \eqref{def_g_analytic} and choose a contour $\Gamma \subset \C$ so that
\begin{align*}
&\left\Vert\partial_y \left( P_1\left(\frac{1}{\varepsilon^{\gamma+2}}\mathcal{A}_\chi+I\right)^{-1} - x_2\pi_1\left(\frac{\chi^2}{\varepsilon^{\gamma+2}}{\mathbb A}^{\rm stretch} +\mathfrak{C}^{\rm stretch}\right)^{-1}\widetilde{\mathcal{M}}^{\rm stretch} - P_1\mathcal{A}_{\chi,\varepsilon,\rm corr}^{\rm stretch}\right)\vect f \right\Vert_{L^2(\omega \times Y)}
\\[0.4em]
&\leq \frac{1}{2\pi} \oint_{\Gamma} \bigl|g_{\varepsilon,\chi}(z)\bigr| \Bigg\| \partial_y \left( P_1\Big(zI -\frac{1}{{\chi^2}}\mathcal{A}_\chi\right)^{-1} - x_2\pi_1\bigl(z\mathfrak{C}^{\rm stretch}-{\mathbb A}^{\rm stretch} \bigr)^{-1}\widetilde{\mathcal{M}}^{\rm stretch} \\[0.4em] 
& \qquad-P_1 \mathcal{B}_{1,\rm stretch}^{\chi, \rm corr}\bigl(z\mathfrak{C}^{\rm stretch}-{\mathbb A}^{\rm stretch}\bigr)^{-1}\widetilde{\mathcal{M}}^{\rm stretch} \Big) \vect f \Bigg\|_{L^2(\omega \times  Y)} dz \leq  C{\chi^2} \max\left\{\frac{\chi^{2}}{\varepsilon^{\gamma+2}}, 1\right\}^{-1}\|\vect f\|_{L^2(\omega \times Y;\C^3)} \\[0.4em] 
&\leq C \varepsilon^{\gamma + 2}\|\vect f\|_{L^2(\omega \times Y;\C^3)}.
\end{align*}
Applying the Gelfand transform, we obtain 
\begin{align*}
     &\left\Vert\partial_{x_3}\left( P_1\left(\frac{1}{\varepsilon^\gamma} \mathcal{A}_\varepsilon + I \right)^{-1}\vect f - x_2 \pi_1 \left(\frac{1}{\varepsilon^\gamma}\mathcal{A}^{\rm stretch} + \mathfrak{C}^{\rm stretch}\right)^{-1}\mathcal{M}^{\rm stretch} \vect f  -P_1 \mathcal{A}_{\rm stretch}^{\rm \rm corr}(\varepsilon) \vect f\right)\right\Vert_{L^2(\omega\times \mathbb{R})}\\[0.6em] 
     &\hspace{25em}\leq C \varepsilon^{\gamma + 1}\|\vect f\|_{L^2(\omega \times \R;\R^3)},
\end{align*}
due to \eqref{trans_back_stretch}, \eqref{trans_back_corr1_str}, and the second identity in \eqref{gelfand_y}.
 The estimates for the remaining derivatives are obtained similarly as in the proof of Theorem \ref{THML2L2} under the Assumption \ref{matsym}, formula  \eqref{stretch23}, see Section \ref{L2toL2}. This establishes the \CCC first estimate in  \eqref{L2H1stretch}.
\end{proof}
Next, we prove the analogous result for the bending case.

\begin{proof}[Proof of the Theorem \ref{l2h1theorem} under Assumption \ref{matsym}, formula \eqref{precision_bend}]
The equivalent form of \eqref{precision_bend} is the following: 
\begin{equation}
\begin{aligned}
        &\left\Vert P_i \left( \frac{1}{\varepsilon^\gamma}\mathcal{A}_\varepsilon + I \right)^{-1} S_{\!\varepsilon^\delta} \vect f - \pi_i \left(\frac{1}{\varepsilon^{\gamma-2}}\mathcal{A}^{\rm bend} + I\right)^{-1}\mathcal{M}_\varepsilon^{\rm bend}\Xi_\varepsilon S_{\!\varepsilon^\delta} \vect f - P_i \mathcal{A}_{\rm bend}^{\rm \rm corr}(\varepsilon)S_{\!\varepsilon^\delta} \vect f \right\Vert_{H^1(\omega\times \mathbb{R})} \\[0.5em]
         &\hspace{15em}\leq C \max\Bigl\{\varepsilon^{\tfrac{\gamma + 2}{4}},\varepsilon^{\tfrac{\gamma }{2}}\Bigr\}\max\Bigl\{\varepsilon^{\tfrac{\gamma + 2}{4} - \delta}, 1\Bigr\} \bigl\Vert  \vect f \bigr\Vert _{L^2(\omega\times \mathbb{R};\R^3)}, \quad i =1,2,\\[0.6em]
        &\left\Vert P_3\left( \frac{1}{\varepsilon^\gamma}\mathcal{A}_\varepsilon + I \right)^{-1}S_{\!\varepsilon^\delta} \vect f +  \varepsilon\widehat{x} 
    \cdot\frac{d}{d x_3}\left(\frac{1}{\varepsilon^{\gamma-2}} \mathcal{A}^{\rm bend} + I\right)^{-1}\mathcal{M}_\varepsilon^{\rm bend}\Xi_\varepsilon S_{\!\varepsilon^\delta} \vect f - P_3 \mathcal{A}_{\rm bend}^{\rm \rm corr}(\varepsilon)S_{\!\varepsilon^\delta} \vect f \right\Vert_{H^1(\omega\times \mathbb{R})} \\[0.6em] 
    &\hspace{15em}\leq  C \max\Bigl\{\varepsilon^{\tfrac{\gamma + 2}{2}},\varepsilon^{\tfrac{3\gamma + 2}{4}}\Bigr\}\max\Bigl\{\varepsilon^{\tfrac{\gamma + 2}{4}- \delta}, 1\Bigr\} \bigl\Vert  \vect f \bigr\Vert _{L^2(\omega\times \mathbb{R};\R^3)}.
\end{aligned}
\label{L2H1bend}
\end{equation}
We start the proof by claiming that the following bounds hold:
	\begin{equation}
	\begin{aligned}
		\left\Vert\partial_y \left(\widehat{\vect u} 
		- 
		\begin{bmatrix} m_1  \\[0.25em] m_2 \end{bmatrix} - \widehat{\vect u}_1
		\right) \right\Vert_{L^2(\omega\times Y,\C^2)} &\leq C{\chi^2}\bigl\Vert  \vect f \bigr\Vert _{L^2(\omega\times Y;\C^3)}, \\[0.3em]
		\bigl\Vert \partial_y \bigl( {u_3}  + {\rm i}\chi( m_1 x_1 + m_2 x_2) - ({\vect u}_1)_3 \bigr) \bigr\Vert_{L^2(\omega\times Y)} &\leq C|\chi|^3\bigl\Vert  \vect f \bigr\Vert _{L^2(\omega\times Y;\C^3)},
	\end{aligned}
\label{u01_bounds}
\end{equation}
with $C>0$ independent of $\chi,$ ${\vect f}.$ To verify (\ref{u01_bounds}), we use the second pair of estimates in \eqref{bend_estimates} and notice that the corrector term $\vect u_0^{(1)}$ does not contribute to the sought approximation, as it is 
 independent of $y$ and satisfies the bounds
\begin{align*}
    &\bigl\|({\vect u}_0^{(1)})_\alpha\bigr\|_{L^2(\omega\times Y)}\leq C|\chi|\|\vect f\|_{L^2(\omega\times Y;\C^3)},\quad\alpha=1,2, \qquad\bigl\|({\vect u}_0^{(1)})_3\bigr\|_{L^2(\omega\times Y)}\leq C{\chi^2}\|\vect f\|_{L^2(\omega\times Y;\C^3)},\\[0.4em]    
    &\bigl\|\partial_{x_\alpha}({\vect u}_0^{(1)})_\alpha\bigr\|_{L^2(\omega\times Y)}\leq C|\chi|\|\vect f\|_{L^2(\omega\times Y;\C^3)},\quad \alpha=1,2, \qquad \bigl\|\partial_{x_3}({\vect u}_0^{(1)})_3\bigr\|_{L^2(\omega\times Y)}\leq C{\chi^2}\|\vect f\|_{L^2(\omega\times Y;\C^3)},
\end{align*}
which follow directly from \eqref{u0^1_bend}. 

It follows from \eqref{u01_bounds}, taking into account the adjustment afforded by Remark \ref{replacement_rem} that there exists $C>0$ such that the following estimates hold for all  ${\vect f}\in L^2(\omega \times Y; {\mathbb C}^3),$ $\chi\in[\pi,\pi)\setminus\{0\}:$
\begin{align*}
    &\left\Vert \partial_y \left( P_i \left( \frac{1}{{\chi^4}} \mathcal{A}_\chi + I \right)^{-1} -  \pi_i \bigl({\mathbb A}^{\rm bend}  + I\bigr)^{-1}\widetilde{\mathcal{M}}_\chi^{\rm bend} -P_i  \mathcal{A}_{\chi,\rm corr}^{\rm bend} \right) S_{\!|\chi|}\vect f \right\Vert_{L^2(\omega \times Y)} 
    \leq {\chi^2} \bigl\Vert  \vect f \bigr\Vert _{L^2(\omega \times Y;\C^3)},\qquad i=1,2, \\[0.5em]
    &\left\Vert \partial_y \left( P_3\left( \frac{1}{{\chi^4}} \mathcal{A}_\chi + I \right)^{-1} +{\rm i}\chi \widehat{x}
     \cdot \bigl({\mathbb A}^{\rm bend}
     +I\bigr)^{-1}\widetilde{\mathcal{M}}_\chi^{\rm bend} -P_3  \mathcal{A}_{\chi,\rm corr}^{\rm bend}\right)S_{\!|\chi|}\vect f \right\Vert_{L^2(\omega \times Y)}  
    \leq |\chi|^3 \bigl\Vert  \vect f \bigr\Vert _{L^2(\omega \times Y;\C^3)}.
\end{align*}
Introducing the function $h_{\varepsilon,\chi}(z)$ as in \eqref{definition_f_analytic}, 
we can provide norm resolvent estimates for the operators
\begin{equation}
    \partial_y \left( P_i\left(\frac{1}{\varepsilon^{\gamma + 2}}\mathcal{A}_\chi+I\right)^{-1} - \pi_i\left(\frac{\chi^4}{\varepsilon^{\gamma + 2}}{\mathbb A}^{\rm bend} +I\right)^{-1}\widetilde{\mathcal{M}}_\chi^{\rm bend}-P_i  \mathcal{A}_{\chi,\varepsilon,\rm corr}^{\rm bend} \right), \quad i=1,2,
\end{equation}
\begin{equation}
    \partial_y \left(P_3\left(\frac{1}{\varepsilon^{\gamma + 2}}\mathcal{A}_\chi+I\right)^{-1} +  {\rm i} \chi
    \widehat{x}\cdot\left(\frac{\chi^4}{\varepsilon^{\gamma + 2}}{\mathbb A}^{\rm bend} +I\right)^{-1}\widetilde{\mathcal{M}}_\chi^{\rm bend}-P_3  \mathcal{A}_{\chi,\varepsilon,\rm corr}^{\rm bend} \right).
\end{equation}
Indeed, for $i=1,2$ we have
\begin{align*}
\Bigg\| \partial_y&\left( P_i\left(\frac{1}{\varepsilon^{\gamma + 2}}\mathcal{A}_\chi+I\right)^{-1}- \pi_i\left(\frac{\chi^4}{\varepsilon^{\gamma + 2}}{\mathbb A}^{\rm bend} +I\right)^{-1}\widetilde{\mathcal{M}}_\chi^{\rm bend} -P_i  \mathcal{A}_{\chi,\varepsilon,\rm corr}^{\rm bend}\right)S_{\!\varepsilon^\delta} \vect f \Bigg \|_{L^2(\omega \times  Y)}  
\\[0.5em]
&\leq \frac{1}{2\pi} \oint_{\Gamma} \bigl|h_{\varepsilon,\chi}(z)\bigr|\Bigg\| \partial_y\Big( P_i 
 \left (zI -\frac{1}{{\chi^4}}\mathcal{A}_\chi\right)^{-1} - \pi_i\bigl(zI-{\mathbb A}^{\rm bend} \bigr)^{-1}\widetilde{\mathcal{M}}_\chi^{\rm bend}  \\[0.5em]
&\hspace{7em}-P_i \mathcal{B}_{1,\rm bend}^{\chi, \rm corr}\bigl(zI-{\mathbb A}^{\rm bend}\bigr)^{-1}\widetilde{\mathcal{M}}_\chi^{\rm bend} \Big) S_{\!|\chi|} S_{\varepsilon^\delta / |\chi|}\vect f \Bigg\|_{L^2(\omega \times Y)} dz \\[0.5em]
&\leq  C{\chi^2} \left(\max\left\{\frac{{\chi^4}}{\varepsilon^{\gamma + 2}}, 1\right\}\right)^{-1}\max\left\{|\chi|\varepsilon^{-\delta}, 1\right\}\|\vect f\|_{L^2(\omega \times Y;\C^3)}  
\\[0.5em]
&\leq C \varepsilon^{\tfrac{\gamma + 2}{ 2}}\max\Bigl\{\varepsilon^{\tfrac{\gamma + 2}{4}-\delta}, 1\Bigr\} \|\vect f\|_{L^2(\omega \times Y;\C^3)}.
\end{align*}
Similarly, for the third component, we have
\begin{align*}
        \Bigg\| \partial_y \bigg( P_3&\left(\frac{1}{\varepsilon^{\gamma + 2}}\mathcal{A}_\chi+I \right)^{-1} - {\rm i} \chi\widehat{x}
        \cdot \left(\frac{\chi^4}{\varepsilon^{\gamma + 2}}{\mathbb A}^{\rm bend} +I\right)^{-1}\widetilde{\mathcal{M}}_\chi^{\rm bend} -P_3  \mathcal{A}_{\chi,\varepsilon,\rm corr}^{\rm bend} \bigg) \CCC S_{\varepsilon^\delta}  \vect f \Bigg \|_{L^2(\omega \times Y)} \\[0.6em]        
        &\leq \frac{1}{2\pi} \oint_{\Gamma} \bigl|h_{\varepsilon,\chi}(z)\bigr|\Bigg\| \partial_y \Big(P_3\left(zI-\frac{1}{\chi^4}\mathcal{A}_\chi\right)^{-1} - {\rm i} \chi\widehat{x}
        \cdot \bigl(zI-{\mathbb A}^{\rm bend} \bigr)^{-1}\widetilde{\mathcal{M}}_\chi^{\rm bend} \\[0.6em] 
 &\hspace{10em}-P_3 \mathcal{B}_{1,\rm bend}^{\chi, \rm corr}\bigl(zI-{\mathbb A}^{\rm bend}\bigr)^{-1}\widetilde{\mathcal{M}}_\chi^{\rm bend} \Big) S_{\!|\chi|} S_{\varepsilon^\delta/|\chi|}\vect f \Bigg\|_{L^2(\omega \times Y)} dz   \\[0.6em]
 &\leq  C|\chi|^3 \left(\max\left\{\frac{{\chi^4}}{\varepsilon^{\gamma + 2}}  , 1\right\}\right)^{-1}\max\left\{|\chi|\varepsilon^{-\delta}, 1\right\}\|\vect f\|_{L^2(\omega \times Y;\C^3)} \\[0.6em]
&\leq C \varepsilon^{\tfrac{3(\gamma + 2)}{4}}\max\Bigl\{\varepsilon^{\tfrac{\gamma+2}{4}-\delta}, 1\Bigr\} \|\vect f\|_{L^2(\omega \times Y;\C^3)}.
\end{align*}
Now, by passing back to the original physical region, we obtain $(i=1,2)$:
\begin{align*}
        \Bigg\| \partial_{x_3} \Bigg( P_i\left(\frac{1}{\varepsilon^{\gamma}}\mathcal{A}_\varepsilon+I\right)^{-1}-\pi_i&\left(\frac{1}{\varepsilon^{\gamma - 2}}\mathcal{A}^{\rm bend}+I\right)^{-1}\mathcal{M}_\varepsilon^{\rm bend}\Xi_\varepsilon -P_i  \mathcal{A}_{\rm corr}^{\rm bend}(\varepsilon)\Bigg) S_{\!\varepsilon^\delta} \vect f \Bigg \|_{L^2(\omega \times  \R)} \\[0.35em]
        &\leq C \varepsilon^{\tfrac{\gamma }{ 2}}\max\Bigl\{\varepsilon^{\tfrac{\gamma + 2}{4}-\delta}, 1\Bigr\} \|\vect f\|_{L^2(\omega \times \R;\R^3)},
\end{align*}
and also 
\begin{align*}
        \Bigg\| \partial_{x_3}\Bigg( P_3\left(\frac{1}{\varepsilon^{\gamma}}\mathcal{A}_\varepsilon+I \right)^{-1} +\widehat{x}
        &\cdot \varepsilon\frac{d}{dx_3} \left(\frac{1}{\varepsilon^{\gamma - 2}}\mathcal{A}^{\rm bend}+I\right)^{-1}\mathcal{M}_\varepsilon^{\rm bend} \Xi_\varepsilon -P_3  \mathcal{A}_{\rm corr}^{\rm bend}(\varepsilon) \Bigg) S_{\!\varepsilon^\delta} \vect f \Bigg \|_{L^2(\omega \times \R)} 
        \\[0.35em]
        &\leq C \varepsilon^{\tfrac{3\gamma + 2}{4}}\max\Bigl\{\varepsilon^{\tfrac{\gamma+2}{4}-\delta}, 1\Bigr\} \|\vect f\|_{L^2(\omega \times \R;\R^3)},
\end{align*}
by virtue of \eqref{trans_back_bend12}, \eqref{transform_back_bend3}, \eqref{trans_back_corr1_str}, and the second formula in \eqref{gelfand_y}. The remaining derivatives are estimated similarly as in the proof of Theorem \ref{THML2L2} under the Assumption \ref{matsym}, formula  \eqref{bend23}, see Section \ref{L2toL2}. 
\end{proof}

\subsection{Higher-order $L^2 \to L^2$ norm-resolvent estimates}
\label{higher_order_sec}

We define the leading order term corrector operators as follows:
\begin{equation*}
  \mathcal{\widetilde{A}}^{\rm stretch}_{\chi,\rm corr} :L^2(\omega \times Y;\C^3)\to L^2(\omega \times Y;\C^3), \qquad \mathcal{\widetilde{A}}^{\rm bend}_{\chi,\rm corr} :L^2(\omega \times Y;\C^3)\to L^2(\omega \times Y;\C^3)
\end{equation*}
such that according to the asymptotic procedure in the last section (equations \eqref{u01_stretch} and  \eqref{u01_bend})
we have:
\begin{equation}
\label{correctoroperatorsu01}
   \mathcal{\widetilde{A}}^{\rm stretch}_{\chi,\rm corr} \vect f:=\bigl(\vect u_0^{(1)}\bigr)_{\rm stretch}, \qquad
   \mathcal{\widetilde{A}}^{\rm bend}_{\chi,\rm corr}  S_{|\chi|} \vect f :=\bigl(\vect u_0^{(1)}\bigr)_{\rm bend}.
\end{equation}

Without going into much detail, we will next briefly discuss the form these operators take under the inverse Gelfand transform, i.e., on the original physical domain.  
We reflect on the asymptotic procedure from the last section and consider now the resolvent problems depending on the spectral parameter $z \in \C$.
Our aim is to vaguely express the operators $\mathcal{\widetilde{A}}^{\rm stretch}_{\chi,\rm corr}(z)$ and $\mathcal{\widetilde{A}}^{\rm bend}_{\chi,\rm corr}(z)$ in a closed form, where $z \in \C$ is a spectral parameter. To this end, we focus first on the stretching case.

   Note that, due to the structure and linearity of the equation \eqref{stretchsecondcorrectorequation}, one can express the corrector term ${\vect u_2}(z)$ as 
   \begin{equation}
       {\vect u_2}(z) =  \chi^2  \widehat{\mathcal{B}} \left(z\mathfrak{C}^{\rm stretch}- \chi^{-2}{\mathbb A}_\chi^{\rm stretch}\right)^{-1}\widetilde{\mathcal{M}}^{\rm stretch} \vect f + \chi^2  \mathcal{B}_{1,\rm stretch}^{\rm corr}\vect f,
   \end{equation}
where $\widehat{\mathcal{B}}_\chi$,  $\mathcal{B}_{1,\rm stretch}^{\chi, \rm corr}$ are  bounded linear operators which are defined via \eqref{stretchsecondcorrectorequation} and \eqref{u01_stretch}.  Furthermore, due to \eqref{u01_stretch}, we have, for all $z$ except those from a countable set,  
\begin{equation}
\label{u01operatorform}
\begin{split}
        \vect u_0^{(1)}(z) &=\chi \left( \widetilde{\mathcal{M}}^{\rm stretch} \right)^*\left(z\mathfrak{C}^{\rm stretch}-{\mathbb A}^{\rm stretch}\right)^{-1}\widecheck{\mathcal{B}} \left(z\mathfrak{C}^{\rm stretch}-{\mathbb A}^{\rm stretch}\right)^{-1}\widetilde{\mathcal{M}}^{\rm stretch} \vect f \\&+\chi \left( \widetilde{\mathcal{M}}^{\rm stretch} \right)^*\left(z\mathfrak{C}^{\rm stretch}-{\mathbb A}^{\rm stretch}\right)^{-1}\widetilde{\mathcal{B}} \vect f.
\end{split}
\end{equation}
where the bounded operators $\widetilde{\mathcal{B}}$, $\widecheck{\mathcal{B}}$ are introduced via \eqref{u01_stretch}. Equivalently, 
\begin{equation}
\begin{split}
        \mathcal{\widetilde{A}}^{\rm stretch}_{\chi,\rm corr}(z) &= \chi \left( \widetilde{\mathcal{M}}^{\rm stretch} \right)^*\left(z\mathfrak{C}^{\rm stretch}-{\mathbb A}^{\rm stretch}\right)^{-1}\widecheck{\mathcal{B}} \left(z\mathfrak{C}^{\rm stretch}-{\mathbb A}^{\rm stretch}\right)^{-1}\widetilde{\mathcal{M}}^{\rm stretch}  \\&+ \chi\left( \widetilde{\mathcal{M}}^{\rm stretch} \right)^*\left(z\mathfrak{C}^{\rm stretch}-{\mathbb A}^{\rm stretch}\right)^{-1}\widetilde{\mathcal{B}}.
\end{split}
\end{equation}
The same structure is valid for the operator $\mathcal{\widetilde{A}}^{\rm bend}_{\chi,\rm corr}(z)$ as well. Next, since we are dealing with a finite dimensional spaces, it is clear that we have the following matrix structure: 
\begin{equation}
    \left( z\mathfrak{C}^{\rm stretch}-{\mathbb A}^{\rm stretch}\right)^{-1}\widecheck{\mathcal{B}} \left(z\mathfrak{C}^{\rm stretch}-{\mathbb A}^{\rm stretch}\right)^{-1} =  \begin{bmatrix}
    \mathcal{B}_{11}(z) & \mathcal{B}_{12}(z) \\[0.35em]
    \mathcal{B}_{21}(z) & \mathcal{B}_{22}(z)
    \end{bmatrix},
\end{equation}   
where the coordinate functions depend on the spectral parameter $z\in \C$ as follows (again, excluding an appropriate countable set of $z$):

\begin{eqnarray*}
    \mathcal{B}_{ij}(z) &=& \frac{a_{ij}}{(z-\chi^{-2}\lambda_1^\chi)^2} + \frac{b_{ij}}{(z-\chi^{-2}\lambda_1^\chi)(z-\chi^{-2}\lambda_2^\chi)} +  \frac{c_{ij}}{(z-\chi^{-2}\lambda_2^\chi)^2} \\
    & & + \frac{d_{ij}}{(z-\chi^{-2}\lambda_1^\chi)}+\frac{e_{ij}}{(z-\chi^{-2}\lambda_2^\chi)}+f_{ij},
\end{eqnarray*}

where $\lambda_1^\chi$ and $\lambda_2^\chi$ are the eigenvalues of the matrix ${\mathbb A}_\chi^{\rm stretch} $ and  $a_{ij}$, $b_{ij}$, $c_{ij}$, $d_{ij}$, $e_{ij}$ $f_{ij}\in{\rm i}{\mathbb R}.$   Recall that those two eigenvalues, when scaled with $1/{\chi^2}$ are positioned in a fixed interval uniformly in $\chi$. 
\begin{lemma}
Let $\Gamma$ be a closed contour enclosing both eigenvalues $\lambda_1^\chi$, $\lambda_2^\chi$ of the matrix ${\mathbb A}_\chi^{\rm stretch} $ uniformly in $\chi$. Then the following formulae hold:
\begin{align*}
    \frac{1}{2\pi{\rm i}} \oint_\Gamma \frac{g_{\varepsilon,\chi}(z)}{(z-\chi^{-2}\lambda_i^\chi)^2}dz&= -\frac{{\chi^2}}{\varepsilon^{\gamma+2}} \frac{1}{(\varepsilon^{-(\gamma+2)}\lambda_i^\chi + 1)^2}, \quad i=1,2, \\[0.5em]
        \frac{1}{2\pi{\rm i}} \oint_\Gamma \frac{g_{\varepsilon,\chi}(z)}{(z-\chi^{-2}\lambda_1^\chi)(z-\chi^{-2}\lambda_2^\chi)}dz&=- \frac{{\chi^2}}{\varepsilon^{\gamma+2}}\frac{1}{(\varepsilon^{-(\gamma+2)}\lambda_1^\chi + 1)(\varepsilon^{-(\gamma+2)}\lambda_2^\chi + 1)}, \\[0.5em]
          \frac{1}{2\pi{\rm i}} \oint_\Gamma \frac{h_{\varepsilon,\chi}(z)}{(z-\chi^{-4}\lambda_i^\chi)^2}dz&= -\frac{{\chi^4}}{\varepsilon^{\gamma+2}} \frac{1}{(\varepsilon^{-(\gamma+2)}\lambda_i^\chi + 1)^2}, \quad i=1,2, \\[0.5em]
        \frac{1}{2\pi{\rm i}} \oint_\Gamma \frac{h_{\varepsilon,\chi}(z)}{(z-\chi^{-4}\lambda_1^\chi)(z-\chi^{-4}\lambda_2^\chi)}dz&=- \frac{{\chi^4}}{\varepsilon^{\gamma+2}}\frac{1}{(\varepsilon^{-(\gamma+2)}\lambda_1^\chi + 1)(\varepsilon^{-(\gamma+2)}\lambda_2^\chi + 1)},
\end{align*}

A similar statement holds  for the matrix ${\mathbb A}_\chi^{\rm bend} $ if we replace $g_{\varepsilon,\chi}$ with $h_{\varepsilon,\chi}$.
\end{lemma}
The previous lemma allows us to conclude the following structure ( after redefining $\widecheck{\mathcal{B}}$ and $ \widetilde{\mathcal{B}}$):
\begin{align}
        \frac{1}{2\pi{\rm i}} \oint_\Gamma g_{\varepsilon,\chi}(z)\mathcal{\widetilde{A}}^{\rm stretch}_{\chi,\rm corr}(z)dz &=- \frac{{\chi^3}}{\varepsilon^{\gamma+2}}  \left( \widetilde{\mathcal{M}}^{\rm stretch} \right)^*\!\left( \frac{\chi^2}{\varepsilon^{\gamma+2}}{\mathbb A}^{\rm stretch} + \mathfrak{C}^{\rm stretch}\right)^{-1}\!\!\widecheck{\mathcal{B}}\!\left(\frac{\chi^2}{\varepsilon^{\gamma+2}}{\mathbb A}^{\rm stretch} + \mathfrak{C}^{\rm stretch}\right)^{-1}\!\!\widetilde{\mathcal{M}}^{\rm stretch}\\[-1em]
        &\quad+ \chi  \left( \widetilde{\mathcal{M}}^{\rm stretch} \right)^*\left( \frac{\chi^2}{\varepsilon^{\gamma+2}}{\mathbb A}^{\rm stretch} + \mathfrak{C}^{\rm stretch}\right)^{-1} \widetilde{\mathcal{B}}  =: \mathcal{\widetilde{A}}^{\rm stretch}_{\chi,\varepsilon,\rm corr},\label{naknadno1_1}  
        \\[0.4em]
        \frac{1}{2\pi{\rm i}} \oint_\Gamma h_{\varepsilon,\chi}(z)\mathcal{\widetilde{A}}^{\rm bend}_{\chi,\rm corr}(z)dz &= \frac{{\chi^5}}{\varepsilon^{\gamma+2}} \left( \widetilde{\mathcal{M}}_\chi^{\rm bend} \right)^*\left( \frac{\chi^4}{\varepsilon^{\gamma+2}}{\mathbb A}^{\rm bend} + I\right)^{-1}\widecheck{\mathcal{B}} \left( \frac{\chi^4}{\varepsilon^{\gamma+2}}{\mathbb A}^{\rm bend} + I\right)^{-1}\widetilde{\mathcal{M}}_\chi^{\rm bend}  \\[0.4em]
        &\quad+\chi\left( \widetilde{\mathcal{M}}_\chi^{\rm bend} \right)^*\left( \frac{\chi^4}{\varepsilon^{\gamma+2}}{\mathbb A}^{\rm bend} + I\right)^{-1}\widetilde{\mathcal{B}} =: \mathcal{\widetilde{A}}^{\rm bend}_{\chi,\varepsilon,\rm corr}.\label{naknadno1_2}
\end{align}
We also use the following notation for the Gelfand pullback of the operators \eqref{naknadno1_1}, \eqref{naknadno1_2}:
\begin{equation}
\label{higherorderl2l2correctors}
    \mathcal{\widetilde{A}}^{\rm \rm corr}_{\rm bend}(\varepsilon)=\mathfrak{G}_\varepsilon^{-1}  \mathcal{\widetilde{A}}^{\rm bend}_{\chi,\varepsilon,\rm corr} \mathfrak{G}_\varepsilon, \qquad 
    \mathcal{\widetilde{A}}^{\rm \rm corr}_{\rm stretch}(\varepsilon)=\mathfrak{G}_\varepsilon^{-1}  \mathcal{\widetilde{A}}^{\rm stretch}_{\chi,\varepsilon,\rm corr} \mathfrak{G}_\varepsilon.
\end{equation}

\begin{remark} 
Using the above expressions as well as \eqref{stretchsecondcorrectorequation}, \eqref{u01_stretch}, and  \eqref{u01_bend}, we infer the existence of operators $K_1, K_2:L^2(\R;\R^n ) \to L^2(\R;\R^n )$, which possess a smoothing effect, such that 
\begin{align*} 
 \mathcal{\widetilde{A}}^{\rm \rm corr}_{\rm stretch}(\varepsilon)\vect f&=
  \varepsilon^{1-\gamma}(\mathcal{M}^{\rm stretch})^* \left( \frac{1}{\varepsilon^{\gamma}}\mathcal{A}^{\rm stretch} + \mathfrak{C}^{\rm stretch}\right)^{-1} \mathbf{A}_1 \frac{d^3}{dx_3^3} \left( \frac{1}{\varepsilon^{\gamma}}\mathcal{A}^{\rm stretch} + \mathfrak{C}^{\rm stretch}\right)^{-1} \mathcal{M}^{\rm stretch}\Xi_{\eps} \vect{f}\\[0.4em]
  &\quad + 
  \varepsilon  (\mathcal{M}^{\rm stretch})^*\left( \frac{1}{\varepsilon^{\gamma}}\mathcal{A}^{\rm stretch} + \mathfrak{C}^{\rm stretch}\right)^{-1}  \mathbf{A}_2 \frac{d}{dx_3}\Xi_{\varepsilon} K_1 \vect{f}, \\[0.4em] 
  \mathcal{\widetilde{A}}^{\rm \rm corr}_{\rm bend}(\varepsilon)\vect f&=
  \varepsilon^{3-\gamma}(\mathcal{M}^{\rm bend}_{\eps})^* \left( \frac{1}{\varepsilon^{\gamma-2}}\mathcal{A}^{\rm bend} +I\right)^{-1}  \mathbf{A}_3 \frac{d^5}{dx_3^5} \left( \frac{1}{\varepsilon^{\gamma-2}}\mathcal{A}^{\rm bend} + I\right)^{-1} \mathcal{M}^{\rm bend}_{\eps}\Xi_{\eps} \vect{f}\\[0.4em]
  &\quad+ 
   \varepsilon(\mathcal{M}^{\rm bend})^*\left( \frac{1}{\varepsilon^{\gamma-2}}\mathcal{A}^{\rm bend} +I\right)^{-1} \mathbf{A}_4  \frac{d}{dx_3}  \Xi_{\varepsilon} K_2 \vect{f},
\end{align*} 
 where $\mathbf{A}_i \in \R^{2 \times 2},$ $i=1,\dots,4$ are matrices.\end{remark}
	
The proof of the second part of Theorem \ref{thm_gen_L2L2high} (i.e. the case of the original system satisfying additional symmetries listed in Assumption \ref{matsym}) now closely follows the argument of Sections \ref{L2toL2}, \ref{L2toH1}.
 
\begin{remark} \label{remnonstandcorr} 
The correctors obtained in 	Section \ref{L2toH1}, cf. \eqref{trans_back_corr1_str}, are standard homogenisation correctors, after a modification. Namely from \eqref{stretchfirstcorrectoreqation}, see also \eqref{bendfirstcorrectoreqation},  it follows that they satisfy the usual cell formula. The modification comes from mollifying the loads via the use of the operators $\widetilde{\mathcal{M}}^{\rm stretch}$ and $\widetilde{\mathcal{M}}^{\rm bend}_{\chi}$. 

We refer to the correctors introduced in this section as non-standard correctors.  They  are in the spirit of the correctors obtained in \cite{ABV} and \cite{BirmanSuslina_corrector}. The difference with respect to \cite{BirmanSuslina_corrector} comes from the complexity of the problem, i.e., due to the fact that our limit problem contains a fourth-order differential expression, while the difference with respect to the correctors of \cite{ABV} comes also from the fact that the error obtained in that paper is not uniform with the respect to the data, unlike that obtained here.
\end{remark}

\section{Analysis of the case of a general elasticity tensor}
\label{section6}
In this section, we drop Assumption \ref{matsym} on the cross-section geometry and material symmetries of the rod. Separately, we investigate and develop asymptotics for the solution of two resolvent problems with different scalings, one for each of the orders of magnitudes of the operator eigenspaces. \RRR Note that the asymptotic procedure in the case of a general elasticity tensor is more involved than the one under Assumption \ref{matsym}. 
 The main results are given in Proposition \ref{propfin1} and Proposition \ref{propfin2}.

\subsection{Asymptotics for ${\chi^2}$-scaled resolvent problem}
We begin with the asymptotics for the following resolvent problem: find $\vect u\in H_\#^1(Y;H^1(\omega;\C^3))$ such that
\begin{equation}
\label{generaltensorproblemchi2}
\frac{1}{{\chi^2}}\int_{\omega \times Y}{\mathbb A} (\simgrad \vect u + {\rm i}X_\chi \vect u) :\overline{(\simgrad \vect v + {\rm i}X_\chi \vect v)}+ \int_{\omega \times Y} \vect u\cdot \overline{\vect v}= \int_{\omega \times Y} \vect f\cdot\overline{\vect v}\qquad \forall \vect v \in H_\#^1\bigl(Y;H^1(\omega;\C^3)\bigr),
\end{equation}
which we also write as
\begin{equation}
\chi^{-2}\left((\simgrad)^* +  \left({\rm i}X_\chi\right)^*\right){\mathbb A}\left(\simgrad + {\rm i}X_\chi\right) \vect u +  \vect u = \vect f.
\end{equation}
\subsubsection*{1) Leading order term, the first and the second corrector}
Consider the solution ${\vect m}\in \C^4$ to the equation
\begin{equation}
\label{leadingordertermgeneralchi2}
\left(\chi^{-2}{\mathbb A}_\chi^{\rm rod} + \mathfrak{C}^{\rm rod}_{\chi}\right) 
{\vect m} = \widetilde{\mathcal{M}}_\chi^{\rm rod} \vect f,
\end{equation}
where ${\mathbb A}_\chi^{\rm rod}$ is defined by \eqref{arodchi}  and $\mathfrak{C}^{\rm rod}_{\chi}$ is defined by \eqref{cstretchrodbend}. Using the fact that $ \mathfrak{C}^{\rm rod}$ is positive and ${\mathbb A}_\chi^{\rm rod}$ is non-negative, we conclude that  the solution satisfies the estimate  
$$
|{\vect m}| \leq C \bigl\Vert  \vect f \bigr\Vert _{L^2(\omega\times Y;\C^3)},
$$
where $C>0$ is independent of $\vect f.$ 
The leading-order term
\begin{equation}
 \RRR   \vect u_0:=\widetilde{\mathcal{I}}_\chi^{\rm rod}{\vect m}\in H_{\#}^1\bigl(Y; H^1(\omega;\C^3)\bigr) 
\end{equation}
clearly satisfies
\begin{equation}
\bigl\Vert\vect u_0\bigr\Vert_{H^1(\omega\times Y ; \C^3)} \leq C\bigl\Vert  \vect f \bigr\Vert _{L^2(\omega\times Y;\C^3)}.
\end{equation}  
The correctors ${\vect u_1}, {\vect u_2} \in H$ are defined as the solutions to the equations 
\begin{align}
(\simgrad)^* {\mathbb A} \simgrad{\vect u_1}&= -(\simgrad)^* {\mathbb A} \Lambda_{\chi,{\vect m}}^{\rm rod},\label{firstcorrectorgeneralchi2}\\[0.5em]
\chi^{-2}(\simgrad)^* {\mathbb A} \simgrad {\vect u_2}&= 
 -\chi^{-2}\Bigl(\left({\rm i}X_{\chi}\right)^*{\mathbb A}\simgrad {\vect u_1} + (\simgrad)^* {\mathbb A}{\rm i}X_{\chi}{\vect u_1} + \left({\rm i}X_{\chi}\right)^* {\mathbb A} \Lambda_{\chi,{\vect m}}^{\rm rod}\Bigr)  \\[0.6em]
 &\quad-\begin{bmatrix}
    m_3 x_2   \\[0.25em]
    -m_3 x_1 \\[0.25em]
     m_4 - {\rm i}\chi(m_1 x_1 + m_2 x_2)
\end{bmatrix}  + \begin{bmatrix}
    \,\widehat{\!\vect f} - \int_{\omega \times Y}\,\widehat{\!\vect f}
    \\[0.5em]
    {f_3}
\end{bmatrix}.
\label{secondcorrectorgeneralchi2}
\end{align}
 Notice that $\vect u_1$ was already defined in the expression \eqref{achirodcorrector}, where it was denoted by $\vect u_{\chi,\vect m}$ and used for the definition of ${\mathbb A}_\chi^{\rm rod}$.  
It is easy to check that the right-hand side of \eqref{secondcorrectorgeneralchi2} is orthogonal to $Y$-periodic infinitesimal rigid-body motions, making that a well-posed problem. The orthogonality with respect to vectors of the form $(c_1,c_2,0)^\top,$ $c_1, c_2\in{\mathbb C},$ is due to \eqref{firstcorrectorgeneralchi2}, while the orthogonality with respect to $(c_3 x_2, -c_3 x_1, c_4)^\top,$ $c_3, c_4\in{\mathbb C},$ is due to \eqref{leadingordertermgeneralchi2} and \eqref{firstcorrectorgeneralchi2}.  We have the following estimates: 
\begin{equation} \label{nak20} 
    \bigl\Vert{\vect u_1}\bigr\Vert_{H^1(\omega\times Y ; \C^3)} \leq C|\chi|\bigl\Vert  \vect f \bigr\Vert _{L^2(\omega\times Y;\C^3)}, \quad 
    \bigl\Vert{\vect u_2}\bigr\Vert_{H^1(\omega\times Y ; \C^3)} \leq C{\chi^2}\bigl\Vert  \vect f \bigr\Vert _{L^2(\omega\times Y;\C^3)}.
\end{equation}

\subsubsection*{2) Refining the approximation}
Next, we update the leading-order term with ${\vect m}^{(1)} \in \C^4$ satisfying
\begin{equation}
\label{u0chi2restart1}
\left(\frac{1}{{\chi^2}}{\mathbb A}_\chi^{\rm rod}  + \mathfrak{C}^{\rm rod}_{\chi} \right) 
{\vect m}^{(1)}\cdot\overline{\vect d} = -\frac{1}{{\chi^2}}\int_{\omega \times Y} {\mathbb A} \left(\simgrad {\vect u_2} + {\rm i}X_{\chi}{\vect u_1} \right): \overline{\Lambda_{\chi,{\vect d}}^{\rm rod}} \qquad \forall {\vect d}\in \C^4.
\end{equation}
Setting ${\vect d}={\vect m}^{(1)}$ in the identity \eqref{u0chi2restart1} 
and using ellipticity estimates for ${\mathbb A}_\chi^{\rm rod} $ as well as suitable bounds  on the right-hand side of \eqref{u0chi2restart1}, we infer that
\begin{equation} \label{nak21} 
    \bigl|{\vect m}^{(1)}\bigr|\leq C|\chi|\bigl\Vert  \vect f \bigr\Vert _{L^2(\omega \times Y;\C^3)}.
\end{equation}
Additionally, by setting to zero the stretching components $d_3,d_4,$ we obtain a sharper estimate
\begin{equation}\label{nak22} 
    \Bigl|\bigl({\vect m}^{(1)}_1,{\vect m}^{(1)}_2\bigr)^\top\Bigr| \leq C{\chi^2}\bigl\Vert  \vect f \bigr\Vert _{L^2(\omega \times Y;\C^3)}.
\end{equation}
Clearly, for $\vect u_0^{(1)}:=\widetilde{\mathcal{I}}_\chi^{\rm rod}{\vect m}^{(1)}$ one has the bound
\begin{equation}\label{nak23}
    \bigl\Vert \vect u_0^{(1)} \bigr\Vert_{H^1(\omega\times Y ; \C^3)} \leq C |\chi|\bigl\Vert  \vect f \bigr\Vert _{L^2(\omega\times Y;\C^3)}.
\end{equation}
The next corrector, ${\vect u_1^{(1)}} \in H,$ is defined as the solution to 
\begin{equation}
\label{u1chi2restart1}
(\simgrad)^* {\mathbb A} \simgrad {\vect u_1^{(1)}} = -(\simgrad)^* {\mathbb A} \Lambda_{\chi,{\vect m}^{(1)}}^{\rm rod}
\end{equation}
and satisfies the bound
\begin{equation}\label{nak24} 
\bigl\Vert{\vect u_1^{(1)}} \bigr\Vert_{H^1(\omega\times Y;\C^3)} \leq C{\chi^2}\bigl\Vert  \vect f \bigr\Vert _{L^2(\omega\times Y;\C^3)}
\end{equation}
\RRR Furthermore, the corrector ${\vect u}_2^{(1)} \in H$ is defined  so as to cancel out the remaining terms of order $|\chi|$ in the approximation. Namely, we set ${\vect u}_2^{(1)}$ to be the solution to
\begin{equation}
\label{u2chi2restart1}
\begin{aligned}
\chi^{-2}(\simgrad)^* {\mathbb A} \simgrad {\vect u_2^{(1)}} &= 
-\chi^{-2}\left(\left({\rm i}X_{\chi}\right)^*{\mathbb A}\simgrad\bigl({\vect u_2} + {\vect u_1^{(1)}} \bigr) + (\simgrad)^* {\mathbb A}{\rm i}X_{\chi}\bigl({\vect u_2} + {\vect u_1^{(1)}} \bigr)\right) \\[0.2em] 
&\quad-\chi^{-2}\left( \left({\rm i}X_{\chi}\right)^* {\mathbb A} \Lambda_{\chi,{\vect m}^{(1)}}^{\rm rod} +\left({\rm i}X_{\chi}\right)^*{\mathbb A}{\rm i}X_{\chi}{\vect u_1}   \right) \\[0.6em] 
&\quad-\begin{bmatrix}
    ({\vect m}^{(1)})_3 x_2   \\[0.3em]
    -({\vect m}^{(1)})_3 x_1 \\[0.3em]
     ({\vect m}^{(1)})_4 - {\rm i}\chi\bigl(({\vect m}^{(1)})_1 x_1 + ({\vect m}^{(1)})_2 x_2\bigr)
\end{bmatrix}  - \begin{bmatrix}
      m_1   \\[0.25em]
     m_2  \\[0.25em]
    0
\end{bmatrix} + \begin{bmatrix}
\int_{\omega \times Y} \,\widehat{\!\vect f}
     \\[0.55em]
    0
\end{bmatrix} - {\vect u_1}.
\end{aligned}
\end{equation}
Obviously, due to \eqref{u1chi2restart1}, the right-hand side of \eqref{u2chi2restart1} is orthogonal to infinitesimal \RRR stretching rigid-body motions. 
To verify that it is also orthogonal to vectors of the form $(c_1,c_2,0)^\top,$ $c_1,c_2 \in \C$, we perform the following calculation:
\begin{align*}
        &- \frac{1}{{\chi^2}}\int_{\omega \times  Y} {\mathbb A}\left(\simgrad\bigl({\vect u_2} + {\vect u_1^{(1)}} \bigr) +  \Lambda_{\chi,{\vect m}^{(1)}}^{\rm rod}(x) +{\rm i}X_{\chi}{\vect u_1}  \right) :\overline{{\rm i}X_{\chi}\begin{bmatrix}
      c_1   \\
     c_2  \\
    0
\end{bmatrix}}
\\[0.9em]
&=  \frac{1}{{\chi^2}}  \int_{\omega \times Y}{\mathbb A}\left(\simgrad \bigl({\vect u_2} + {\vect u_1^{(1)}}\bigr) +  \Lambda_{\chi,{\vect m}^{(1)}}^{\rm rod} +{\rm i}X_{\chi}{\vect u_1}   \right) :\overline{\simgrad\begin{bmatrix}
      0   \\
     0  \\
    - {\rm i}\chi(c_1 x_1 + c_2 x_2)
\end{bmatrix}} \\[0.9em]
&=\frac{1}{{\chi^2}}\int_{\omega \times  Y}{\mathbb A}\left(\simgrad {\vect u_2}  +{\rm i}X_{\chi}{\vect u_1}   \right): \overline{\simgrad \begin{bmatrix}
      0   \\
     0  \\
    - {\rm i}\chi(c_1 x_1 + c_2 x_2)
\end{bmatrix} }\quad \quad  (\mbox{due to \eqref{u1chi2restart1}})\\[0.9em]
&=-\frac{1}{{\chi^2}}\int_{\omega \times Y} {\mathbb A}\left(\simgrad {\vect u_1}  +\Lambda_{\chi,{\vect m}}^{\rm rod}\right) : \overline{{\rm i}X_\chi\begin{bmatrix}
      0   \\
     0  \\
    - {\rm i}\chi(c_1 x_1 + c_2 x_2) 
\end{bmatrix}} \\[0.9em] 
&\quad- \int_{\omega\times Y}\begin{bmatrix}
      0   \\
     0  \\
    - {\rm i}\chi(m_1 x_1 + m_2 x_2)
\end{bmatrix}\cdot\overline{\begin{bmatrix}
      0   \\
     0  \\
    - {\rm i}\chi(c_1 x_1 + c_2 x_2) 
\end{bmatrix} }
+ \int_{\omega\times Y}\begin{bmatrix}
      0   \\
     0  \\
    {f_3}
\end{bmatrix}\cdot\overline{\begin{bmatrix}
      0   \\
     0  \\
    - {\rm i}\chi(c_1 x_1 + c_2 x_2) 
\end{bmatrix}}\quad \quad  (\mbox{due to \eqref{secondcorrectorgeneralchi2}})
\\[0.9em]
&=-\frac{1}{{\chi^2}} \int_{\omega \times Y}{\mathbb A} \bigl(\simgrad {\vect u_1}  +\Lambda_{\chi,{\vect m}}^{\rm rod}   \bigr):\overline{\Lambda_{\chi,c_1,c_2}^{\rm bend}}  \\[0.9em]
&\quad- \int_{\omega\times Y}\begin{bmatrix}
      0   \\
     0  \\
    - {\rm i}\chi(m_1 x_1 + m_2 x_2)
\end{bmatrix}\cdot\overline{\begin{bmatrix}
      0   \\
     0  \\
    - {\rm i}\chi(c_1 x_1 + c_2 x_2) 
\end{bmatrix} }+ \int_{\omega\times Y}\begin{bmatrix}
      0   \\
     0  \\
    {f_3}
\end{bmatrix}\cdot\overline{\begin{bmatrix}
      0   \\
     0  \\
    - {\rm i}\chi(c_1 x_1 + c_2 x_2) 
\end{bmatrix}} \\[0.9em]
&=\int_{\omega\times Y}\begin{bmatrix}
      m_1   \\
     m_2  \\
    - {\rm i}\chi(m_1 x_1 + m_2 x_2)
\end{bmatrix}\cdot\overline{\begin{bmatrix}
      c_1   \\
     c_2  \\
    - {\rm i}\chi(c_1 x_1 + c_2 x_2) 
\end{bmatrix} }- \int_{\omega\times Y}{\vect f}
\cdot\overline{\begin{bmatrix}
      c_1   \\
     c_2\\
    - {\rm i}\chi(c_1 x_1 + c_2 x_2) 
\end{bmatrix}}\\[0.9em]
&\quad- \int_{\omega\times Y}\begin{bmatrix}
      0   \\
     0  \\
    - {\rm i}\chi(m_1 x_1 + m_2 x_2)
\end{bmatrix}\cdot\overline{\begin{bmatrix}
      0   \\
     0  \\
    - {\rm i}\chi(c_1 x_1 + c_2 x_2) 
\end{bmatrix} }+ \int_{\omega\times Y}\begin{bmatrix}
      0   \\
     0  \\
    {f_3}
\end{bmatrix}\cdot\overline{\begin{bmatrix}
      0   \\
     0  \\
    - {\rm i}\chi(c_1 x_1 + c_2 x_2) 
\end{bmatrix}} \ (\mbox{due to \eqref{leadingordertermgeneralchi2}, \eqref{firstcorrectorgeneralchi2}})\\[0.9em]
&=\begin{bmatrix}
      m_1   \\
     m_2  \\
0
\end{bmatrix}\cdot\overline{\begin{bmatrix}
      c_1   \\
     c_2  \\
0
\end{bmatrix} }- \begin{bmatrix}
\int_{\omega \times Y}\,\widehat{\!\vect f}
      \\[0.4em]
0
\end{bmatrix}\cdot\overline{\begin{bmatrix}
      c_1   \\
     c_2\\
0
\end{bmatrix}} = \left( \begin{bmatrix}
      m_1   \\
     m_2  \\
0
\end{bmatrix} - \begin{bmatrix}
\int_{\omega \times Y}\,\widehat{\!\vect f}
      \\[0.4em]
0
\end{bmatrix} \right)\cdot\overline{\begin{bmatrix}
      c_1   \\
     c_2\\
0
\end{bmatrix}}.
\end{align*}
It follows that the problem \eqref{u2chi2restart1} is well posed. The following  statement is  the key one  in obtaining the required bound for its solution. 
	\begin{lemma}
		\label{aux_lemma}
		The solution ${\vect m}\in \C^4$ of \eqref{leadingordertermgeneralchi2} satisfies the bound  
		\begin{equation}
				\left|\begin{bmatrix}
					m_1 \\[0.25em] m_2
				\end{bmatrix} - 
					\int_{\omega\times Y} \,\widehat{\!\vect f}
			\right| \leq C |\chi| \bigl\Vert  \vect f \bigr\Vert _{L^2(\omega\times Y;\C^3)}.
			\end{equation}
	\end{lemma}
	
	\begin{proof}
		Testing the equation \eqref{leadingordertermgeneralchi2} with a constant vector 
		\[
		\widetilde{\vect m}=(\widetilde{m}_1,\widetilde{m}_2,0,0):= \left(m_1-\int_{\omega\times Y} {f_1},\ m_2 -  \int_{\omega\times Y} {f_2}, 0,0 \right)^\top,
		\]
		one obtains 
			\begin{align*}
				\bigl\vert\widetilde{\vect m}\bigr\vert^2&=
				\biggl(-\chi^2\mathfrak{c}_1(\omega)m_1+{\rm i} \chi \int_{\omega\times Y}{f_3} x_1,\ -\chi^2\mathfrak{c}_2(\omega)m_2+{\rm i} \chi \int_{\omega\times Y}{f_3} x_2 \biggr)^\top\cdot(\widetilde{m}_1, \widetilde{m}_2)^\top
				-{\mathbb A}_\chi^{\rm rod}{\vect m}\cdot \overline{\widetilde{\vect m}},
			\end{align*}
		The proof is completed by noticing that 
				${\mathbb A}_\chi^{\rm rod}{\vect m}\cdot \overline{\widetilde{\vect m}}  \leq |\chi|^3 |\vect m|\bigl\vert\widetilde{\vect m}\bigr\vert.$ 
	\end{proof}
Combining Lemma \ref{aux_lemma} with the estimates for ${\vect m}^{(1)},$ ${\vect u}_1,$ ${\vect u}_2,$ and ${\vect u}_1^{(1)}$, obtained above,
yields 
\begin{equation}\label{nak25} 
\bigl\Vert {\vect u_2^{(1)}} \bigr\Vert _{H^1(\omega\times Y;\C^3)} \leq C|\chi|^3 \bigl\Vert  \vect f \bigr\Vert _{L^2(\omega\times Y;\C^3)}.
\end{equation}
\subsubsection*{4) Final approximation}
The final approximation
$$ 
\widetilde{\vect u}_{\rm approx}:=\vect u_0 + \vect u_0^{(1)} + {\vect u_1} + {\vect u_1^{(1)}} +{\vect u_2} + {\vect u_2^{(1)}} 
$$
satisfies the equation
\begin{equation}
\chi^{-2}\left((\simgrad)^* + \left({\rm i}X_{\chi}\right)^*\right){\mathbb A}\left(\simgrad + {\rm i}X_\chi\right) \widetilde{\vect u}_{\rm approx} + \widetilde{\vect u}_{\rm approx} - \vect f = \tilde{\vect R}_{\chi},
\end{equation}
where
\begin{align*}
\widetilde{\vect R}_{\chi}&= \chi^{-2}\left( \left({\rm i}X_{\chi}\right)^*{\mathbb A}\simgrad{\vect u_2^{(1)}} + (\simgrad)^* {\mathbb A}{\rm i}X_{\chi}{\vect u_2^{(1)}} + \left({\rm i}X_{\chi}\right)^* {\mathbb A}{\rm i}X_{\chi} {\vect u_2^{(1)}} + \left({\rm i}X_{\chi}\right)^* {\mathbb A}{\rm i}X_{\chi}{\vect u_2} \right)  \\[0.4em]
&+\chi^{-2}\left( \left({\rm i}X_{\chi}\right)^* {\mathbb A}{\rm i}X_{\chi}{\vect u_2} + \left({\rm i}X_{\chi}\right)^* {\mathbb A}{\rm i}X_{\chi} {\vect u_1^{(1)}} \right) +\bigl(({\vect m}^{(1)})_1,({\vect m}^{(1)})_2,0\bigr)^\top+ {\vect u_1^{(1)}} +{\vect u_2} +{\vect u_2^{(1)}} .
\end{align*}
Clearly, one has
\begin{equation}
\bigl\Vert\widetilde{\vect R}_{\chi}\bigr\Vert_{[H^1_\#(Y;H^1(\omega;\C^3))]^*} \leq C {\chi^2} \bigl\Vert  \vect f \bigr\Vert _{L^2(\omega\times Y;\C^3)}.
\end{equation}
The approximation error
$$
\vect u_{\rm error}:={\vect u} - \widetilde{\vect u}_{\rm approx}
$$
can now be estimated by observing that 
$$
\chi^{-2}\left((\simgrad)^* + \left({\rm i}X_\chi\right)^*\right){\mathbb A}\left(\simgrad + {\rm i}X_\chi\right) \vect u_{\rm error} + \vect u_{\rm error} = -\tilde{\vect R}_{\chi}, 
$$
which together with the estimate \eqref{aprioriestimateh-1} immediately implies
\begin{equation}\label{nak26} 
\Vert\vect u_{\rm error}\Vert_{H^1(\omega\times Y;\C^3)} \leq C{\chi^2} \bigl\Vert  \vect f \bigr\Vert _{L^2(\omega\times Y;\C^3)}.
\end{equation}
Finally, leaving out higher-order terms yields the following result.
\begin{proposition}\label{propfin1} 
Let $\vect u\in H^1_\#(Y;H^1(\omega;\C^3))$ be the solution of \eqref{generaltensorproblemchi2}. Then, there exists $C>0$ such that following estimates are valid for all ${\vect f}\in L^2(\omega\times Y;\C^3),$ $\chi\in[-\pi,\pi)\setminus\{0\}:$  
\begin{equation}
\begin{split}\label{generaltensorestimateschi2}
 \left\Vert\vect u - 
\vect u_0 \right\Vert_{H^1(\omega\times Y,\C^3)} \leq C|\chi|\bigl\Vert  \vect f \bigr\Vert _{L^2(\omega\times Y;\C^3)}, \\[0.4em]
\bigl\Vert\vect u-\vect u_0 -\vect u_0^{(1)}-{\vect u_1} \bigr\Vert_{H^1(\omega\times Y,\C^3)} \leq C{\chi^2}\bigl\Vert  \vect f \bigr\Vert _{L^2(\omega\times Y;\C^3)}, 
\end{split}
\end{equation}
where $\vect u_0,\vect u_0^{(1)},{\vect u_1}$ are defined with the above approximation procedure.
\end{proposition}
 
\begin{proof} 
The proof is a direct consequence of \eqref{nak26}, by taking into account the estimates \eqref{nak20}, \eqref{nak21}, \eqref{nak22}, \eqref{nak23}, \eqref{nak24}, and \eqref{nak25}. 	
\end{proof}

\subsection{Asymptotics for ${\chi^4}$-scaled resolvent problem}
Here we focus on deriving the asymptotics for the following resolvent problem: find $\vect u\in H_\#^1(Y;H^1(\omega;\C^3))$ such that
\begin{equation}
\label{generaltensorproblemchi4}
\frac{1}{{\chi^4}}\int_{\omega \times Y}{\mathbb A}(\simgrad \vect u + {\rm i}X_\chi \vect u) : \overline{(\simgrad \vect v + {\rm i}X_\chi \vect v)} + \int_{\omega \times Y} \vect u\cdot\overline{\vect v} = \int_{\omega \times Y} S_{\!|\chi|} \vect f\cdot\overline{\vect v} \qquad \forall \vect v \in H_\#^1\bigl(Y;H^1(\omega;\C^3)\bigr),
\end{equation}
which is formally written as
\begin{equation}
\chi^{-4}\left((\simgrad)^* + \left({\rm i}X_\chi\right)^*\right){\mathbb A}\left(\simgrad + {\rm i}X_\chi\right) \vect u + \vect u = S_{\!|\chi|}\vect f.
\end{equation}
\subsubsection*{1) Leading order term, the first and the second corrector}
The leading order term in the asymptotic expansion is defined with the solution to the following homogenised equation:
\begin{equation}
\left(\chi^{-4}{\mathbb A}_\chi^{\rm rod}+\mathfrak{C}^{\rm rod}_{\chi}\right) 
{\vect m}=\widetilde{\mathcal{M}}_\chi^{\rm rod}S_{\!|\chi|} \vect f .
\end{equation}
By using the apriori estimates on ${\mathbb A}_\chi^{\rm rod} $ we derive the following estimates by testing with $m\in \C^4$: 
\begin{align*}
        \chi^{-4}\left({\chi^4} \bigl|(m_1,m_2)^\top\bigr|^2+ {\chi^2} \bigl|(m_3,m_4)^\top\bigr|^2\right) &\leq  \left\Vert \vect f \right\Vert_{L^2(\omega \times Y; \C^3)} \left(\bigl|(m_1,m_2)^\top\bigr| + \bigl|m_3, |\chi|^{-1}m_4\bigr| \right) \\[0.4em]
        &\leq \left\Vert \vect f \right\Vert_{L^2(\omega \times Y; \C^3)} \bigl(\bigl|(m_1,m_2)^\top\bigr| +|\chi|^{-1}\bigl|(m_3,m_4)^\top\bigr| \bigr),
\end{align*}
from where we read: 
\begin{equation}
    \bigl|(m_1,m_2)^\top\bigr|\leq \left\Vert \vect f \right\Vert_{L^2(\omega \times Y; \C^3)}, \quad\bigl|(m_3,m_4)^\top\bigr| \leq |\chi| \left\Vert \vect f \right\Vert_{L^2(\omega \times Y; \C^3)}.
\end{equation}
The leading-order term, defined by $\vect u_0:= \widetilde{\mathcal{I}}^{\rm rod}_{\chi}{\vect m},$ satisfies the bounds 
\begin{equation}
\bigl\Vert ({\vect u}_0)_\alpha\bigr\Vert _{H^1(\omega\times Y ; \C)} \leq C\bigl\Vert  \vect f \bigr\Vert _{L^2(\omega\times Y;\C^3)},\quad \alpha=1,2, \qquad 
\bigl\Vert ({\vect u}_0)_3\bigr\Vert _{H^1(\omega\times Y ; \C)} \leq C|\chi|\bigl\Vert  \vect f \bigr\Vert _{L^2(\omega\times Y;\C^3)}
\end{equation}  
The next two ``corrector" terms ${\vect u_1},{\vect u_2} \in H$ are defined as the solutions to the well-posed problems 
\begin{align}
(\simgrad)^* {\mathbb A} \simgrad {\vect u_1}&= -(\simgrad)^* {\mathbb A} \Lambda_{\chi,{\vect m}}^{\rm rod},
 \\[0.7em]
    \chi^{-4}(\simgrad)^* {\mathbb A}\simgrad {\vect u_2}
     &=-\chi^{-4}\left(\left({\rm i}X_\chi\right)^*{\mathbb A}\simgrad {\vect u_1} + (\simgrad)^* {\mathbb A}{\rm i}X_{\chi}{\vect u_1} + \left({\rm i}X_\chi\right)^* {\mathbb A} \Lambda_{\chi,{\vect m}}^{\rm rod}\right)  \\[0.8em]
  &\quad-\begin{bmatrix}
    m_3 x_2   \\[0.25em]
    -m_3 x_1 \\[0.25em]
     m_4 - {\rm i}\chi(m_1 x_1 + m_2 x_2)
\end{bmatrix}  + \begin{bmatrix}
\,\widehat{\!\vect f} - \int_{\omega \times Y}\,\widehat{\!\vect f}
    \\[0.6em]
    |\chi|^{-1}{f_3}
\end{bmatrix} 
\end{align}
and clearly satisfy the bounds 
\begin{equation}\label{nak30} 
\Vert{\vect u_1}\Vert_{H^1(\omega\times Y ; \C^3)} \leq C{\chi^2}\bigl\Vert  \vect f \bigr\Vert _{L^2(\omega\times Y;\C^3)}, \qquad \Vert{\vect u_2}\Vert_{H^1(\omega\times Y ; \C^3)} \leq C|\chi|^3\bigl\Vert  \vect f \bigr\Vert _{L^2(\omega\times Y;\C^3)}.
\end{equation}
\subsubsection*{2) First refinement of the approximation}
We proceed by updating the leading-order term. 
Define $\vect u_0^{(1)} = \widetilde{\mathcal{I}}_\chi^{\rm rod}{\vect m}^{(1)}$, where ${\vect m}^{(1)}$ is the solution to
\begin{equation}
	\label{u01_generalchi4}
\begin{split}
\left(\frac{1}{{\chi^4}}{\mathbb A}_\chi^{\rm rod}+\mathfrak{C}^{\rm rod}_{\chi}\right)
    {\vect m}^{(1)}\cdot \overline{\vect d}
=-\frac{1}{{\chi^4}}\int_{\omega \times Y} {\mathbb A} \left(\simgrad {\vect u_2} + {\rm i}X_{\chi}{\vect u_1} \right): \overline{\Lambda_{\chi,{\vect d}}^{\rm rod}} \qquad \forall {\vect d}\in \C^4.
\end{split}
\end{equation}
We have 
\begin{equation}
	\label{u0^1_general}
\begin{aligned}
\Bigl|\bigl(({\vect m}^{(1)})_1,({\vect m}^{(1)})_2\bigr)^\top\Bigr|&\leq C|\chi| \bigl\Vert  \vect f \bigr\Vert _{L^2(\omega\times Y,\C^3)},\qquad \Bigl|\bigl(({\vect m}^{(1)})_3,({\vect m}^{(1)})_4\bigr)^\top\Bigr| \leq C{\chi^2} \left\Vert \vect f \right\Vert_{L^2(\omega \times Y; \C^3)}\\[0.3em]
 \bigl\Vert({\vect u}_0^{(1)})_\alpha\bigr\Vert_{H^1(\omega\times Y ; \C^3)}&\leq C|\chi|\bigl\Vert  \vect f \bigr\Vert _{L^2(\omega\times Y;\C^3)},\quad \alpha=1,2, \qquad 
\bigl\Vert({\vect u}_0^{(1)})_3\bigr\Vert_{H^1(\omega\times Y ; \C^3)} \leq C{\chi^2}\bigl\Vert  \vect f \bigr\Vert _{L^2(\omega\times Y;\C^3)}
\end{aligned}
\end{equation}
Next, we introduce a corrector  ${\vect u_1^{(1)}}$ as the solution to
\begin{equation}
 (\simgrad)^* {\mathbb A} \simgrad {\vect u_1^{(1)}}  = - (\simgrad)^* {\mathbb A} \Lambda_{\chi,{\vect m}^{(1)}}^{\rm rod}, 
\end{equation}
which satisfies the bound
\begin{equation}\label{nak31} 
\bigl\Vert{\vect u_1^{(1)}} \bigr\Vert_{H^1(\omega\times Y ; \C^3)} \leq C|\chi|^3\bigl\Vert  \vect f \bigr\Vert _{L^2(\omega\times Y;\C^3)}.
\end{equation}
The next-order corrector ${\vect u_2^{(1)}} \in H$ is defined as the solution to
\begin{align}
\chi^{-4}(\simgrad)^* {\mathbb A} \simgrad {\vect u_2^{(1)}} =& 
-\chi^{-4}\left( \left({\rm i}X_\chi\right)^* {\mathbb A} \simgrad \bigl({\vect u_2} + {\vect u_1^{(1)}}\bigr) + (\simgrad)^* {\mathbb A} {\rm i}X_{\chi}\bigl({\vect u_2} + {\vect u_1^{(1)}}\bigr) \right)
 \\[0.7em]
&-\chi^{-4}\left( \left({\rm i}X_\chi\right)^* {\mathbb A} \Lambda_{\chi,{\vect m}^{(1)}}^{\rm rod} + \left({\rm i}X_\chi\right)^* {\mathbb A}{\rm i}X_{\chi} {\vect u_1} \right) \\[0.7em] 
  &-\begin{bmatrix}
    ({\vect m}^{(1)})_3 x_2   \\[0.3em]
    -({\vect m}^{(1)})_3 x_1 \\[0.3em]
     ({\vect m}^{(1)})_4 - {\rm i}\chi\bigl(({\vect m}^{(1)})_1 x_1 + ({\vect m}^{(1)})_2 x_2\bigr)
\end{bmatrix} + \begin{bmatrix}\int_{\omega \times Y}\,\widehat{\!\vect f}
     \\[0.6em]
    0
\end{bmatrix} - \begin{bmatrix}
      m_1   \\[0.2em]
     m_2  \\[0.2em]
    0
\end{bmatrix}
\end{align}
and satisfies the bound
\begin{equation}
\bigl\Vert{\vect u_2^{(1)}} \bigr\Vert_{H^1(\omega\times Y ; \C^3)} \leq C{\chi^4}\bigl\Vert  \vect f \bigr\Vert _{L^2(\omega\times Y;\C^3)}.
\end{equation}
\subsubsection*{3) Second refinement of the approximation}
Adding the correctors $\vect u_0^{(2)} = \widetilde{\mathcal{I}}^{\rm rod}_\chi{\vect m}^{(2)}, {\vect u_1^{(2)}} , {\vect u_2^{(2)}} \in H$ further decreases the approximation error. They are constructed sequentially using the following relations:
\begin{align*}
\left(\frac{1}{{\chi^4}}{\mathbb A}_\chi^{\rm rod}  + \mathfrak{C}^{\rm rod}_\chi\right) 
    {\vect m}^{(2)}\cdot \overline{\vect d}
&=-\frac{1}{{\chi^4}}\int_{\omega \times Y} {\mathbb A} \left(\simgrad {\vect u_2^{(1)}} + {\rm i}X_{\chi}{\vect u_1^{(1)}} + {\rm i}X_{\chi}{\vect u_2} \right): \overline{\Lambda_{\chi,{\vect d}}^{\rm rod}}\qquad \forall {\vect d}\in \C^4,\\[0.6em]
 (\simgrad)^* {\mathbb A} \simgrad {\vect u_1^{(2)}} &= 
 -(\simgrad)^*{\mathbb A} \Lambda_{\chi,{\vect m}^{(2)}}^{\rm rod},\\[0.6em]
\chi^{-4}(\simgrad)^* {\mathbb A} \simgrad {\vect u_2^{(2)}} &= 
 -\chi^{-4}\left(\left({\rm i}X_\chi\right)^*{\mathbb A}\simgrad\bigl({\vect u_1^{(2)}}  + {\vect u_2^{(1)}}\bigr) + (\simgrad)^* {\mathbb A}{\rm i}X_{\chi}\bigl({\vect u_1^{(2)}}  + {\vect u_2^{(1)}}\bigr) \right) \\[0.6em]
 &\quad-\chi^{-4}\left( \left({\rm i}X_\chi\right)^* {\mathbb A} \Lambda_{\chi,{\vect m}^{(2)}}^{\rm rod} + \left({\rm i}X_\chi\right)^* {\mathbb A}{\rm i}X_{\chi}\bigl({\vect u_2} + {\vect u_1^{(1)}}\bigr) \right)  \\[0.6em]
&\quad-\begin{bmatrix}
      ({\vect m}^{(1)})_1   \\[0.25em]
     ({\vect m}^{(1)})_2  \\[0.25em]
    0
\end{bmatrix}  - \begin{bmatrix}
    ({\vect m}^{(2)})_3 x_2   \\[0.35em]
    -({\vect m}^{(2)})_3 x_1 \\[0.35em]
     ({\vect m}^{(2)})_4 - {\rm i}\chi\bigl(({\vect m}^{(2)})_1 x_1 + ({\vect m}^{(2)})_2 x_2\bigr)
\end{bmatrix}.
\end{align*}
All these problems are well posed, which can be seen by checking that 
the right-hand sides vanish when tested against functions in $H$. The correctors satisfy the following estimates:
\begin{align}\label{nak32} 
\bigl\Vert({\vect u}_0^{(2)})_\alpha\bigr\Vert_{H^1(\omega\times Y ; \C^3)} &\leq C{\chi^2}\bigl\Vert  \vect f \bigr\Vert _{L^2(\omega\times Y;\C^3)},\quad \alpha=1,2, \qquad 
\bigl\Vert({\vect u}_0^{(2)})_3\bigr\Vert_{H^1(\omega\times Y ; \C^3)} \leq C|\chi|^3\bigl\Vert  \vect f \bigr\Vert _{L^2(\omega\times Y;\C^3)},\\[0.3em]
\bigl\Vert{\vect u_1^{(2)}} \bigr\Vert_{H^1(\omega\times Y ; \C^3)}&\leq C{\chi^4}\bigl\Vert  \vect f \bigr\Vert _{L^2(\omega\times Y;\C^3)},\qquad
\bigl\Vert{\vect u_2^{(2)}} \bigr\Vert_{H^1(\omega\times Y ; \C^3)} \leq C|\chi|^5\bigl\Vert  \vect f \bigr\Vert _{L^2(\omega\times Y;\C^3)}.
\end{align}

\subsubsection*{4) Third refinement of the approximation}

The final approximation cycle consists of defining the corrector terms $\vect u_0^{(3)} = \widetilde{\mathcal{I}}^{\rm rod}_\chi {\vect m}^{(3)}, {\vect u_1^{(3)}}, {\vect u_2^{(3)}}\in H$  using the following relations:
\begin{align*}
\left(\frac{1}{{\chi^4}}{\mathbb A}_\chi^{\rm rod}  + \mathfrak{C}^{\rm rod}_\chi\right) 
    {\vect m}^{(3)}\cdot \overline{\vect d}
&= -\frac{1}{{\chi^4}}\int_{\omega \times Y} {\mathbb A} \left(\simgrad {\vect u_2^{(2)}}  + {\rm i}X_{\chi}{\vect u_1^{(2)}}  + {\rm i}X_{\chi}{\vect u_1^{(1)}} \right): \overline{\Lambda_{\chi,{\vect d}}^{\rm rod}} \\[0.6em]
&\quad -{\rm i}\chi\int_{\omega \times Y}\left(x_1u_1, x_2u_2 \right)^\top\cdot \overline{(d_1,d_2)^\top}\qquad \forall {\vect d}=(d_1, d_2, d_3, d_4)\in{\mathbb C}^4,\\[0.6em]
 (\simgrad)^* {\mathbb A} \simgrad {\vect u_1^{(3)}}&=- (\simgrad)^*{\mathbb A} \Lambda_{\chi, {\vect m}^{(3)}}^{\rm rod},\\[0.6em]
\chi^{-4}(\simgrad)^* {\mathbb A} \simgrad {\vect u_2^{(3)}}&= 
 -\chi^{-4}\left(\left({\rm i}X_\chi\right)^*{\mathbb A}\simgrad\bigl({\vect u_1^{(3)}} + {\vect u_2^{(2)}}\bigr) + (\simgrad)^* {\mathbb A}{\rm i}X_{\chi}({\vect u_1^{(3)}} + {\vect u_2^{(2)}} ) \right) \\[0.6em]
 &\quad-\chi^{-4}\left( \left({\rm i}X_\chi\right)^* {\mathbb A} \Lambda_{\chi, {\vect m}^{(3)}}^{\rm rod} + \left({\rm i}X_\chi\right)^* {\mathbb A}{\rm i}X_{\chi} \bigl({\vect u_2^{(1)}} + {\vect u_1^{(2)}}\bigr) \right)  \\[0.6em]
&\quad-\begin{bmatrix}
      ({\vect m}^{(2)})_1   \\[0.3em]
     ({\vect m}^{(2)})_2  \\[0.3em]
    0
\end{bmatrix}  - \begin{bmatrix}
    ({\vect m}^{(3)})_3 x_2   \\[0.4em]
    -({\vect m}^{(3)})_3 x_1 \\[0.4em]
     ({\vect m}^{(3)})_4 - {\rm i}\chi\bigl(({\vect m}^{(3)})_1 x_1 + ({\vect m}^{(3)})_2 x_2\bigr)
\end{bmatrix} - {\vect u_1}.
\end{align*}
The above problems define unique correctors, which satisfy the estimates
\begin{align}\label{nak33} 
\bigl\Vert({\vect u}_0^{(3)})_\alpha\bigr\Vert_{H^1(\omega\times Y ; \C)}&\leq C|\chi|^3\bigl\Vert  \vect f \bigr\Vert _{L^2(\omega\times Y;\C^3)},\quad \alpha=1,2, \qquad 
\bigl\Vert({\vect u}_0^{(3)})_3\bigr\Vert_{H^1(\omega\times Y ; \C)} \leq C{\chi^4}\bigl\Vert  \vect f \bigr\Vert _{L^2(\omega\times Y;\C^3)},\\[0.5em]
\bigl\Vert{\vect u_1^{(3)}}\bigr\Vert_{H^1(\omega\times Y ; \C^3)}&\leq C|\chi|^5\bigl\Vert  \vect f \bigr\Vert _{L^2(\omega\times Y;\C^3)},\qquad
\bigl\Vert{\vect u_2^{(3)}}\bigr\Vert_{H^1(\omega\times Y ; \C^3)} \leq C\chi^6\bigl\Vert  \vect f \bigr\Vert _{L^2(\omega\times Y;\C^3)}.
\end{align}

\subsubsection*{5) Final approximation}
The function $\widetilde{\vect u}_{\rm approx}$, defined by
$$ 
\widetilde{\vect u}_{\rm approx}:= \vect u_0 + \vect u_0^{(1)} + \vect u_0^{(2)} + \vect u_0^{(3)} + {\vect u_1} + {\vect u_1^{(1)}} + {\vect u_1^{(2)}}  + {\vect u_1^{(3)}} + {\vect u_2} + {\vect u_2^{(1)}} +  {\vect u_2^{(2)}}  + {\vect u_2^{(3)}}, 
$$
is the solution to the following problem:
\begin{equation}
\chi^{-4}\left((\simgrad)^* + \left( {\rm i}X_\chi\right)^*\right){\mathbb A}\left(\simgrad + {\rm i}X_\chi\right) \widetilde{\vect u}_{\rm approx} + \widetilde{\vect u}_{\rm approx} - S_{\!|\chi|}\vect f = \widetilde{\vect R}_{\chi},
\end{equation}
where
\begin{align*}
&\widetilde{\vect R}_{\chi} = \chi^{-4}\left( \left({\rm i}X_{\chi}\right)^*{\mathbb A}(\simgrad{\vect u_2^{(3)}} + (\simgrad)^* {\mathbb A}{\rm i}X_{\chi}{\vect u_2^{(3)}} + \left({\rm i}X_{\chi}\right)^* {\mathbb A}{\rm i}X_{\chi} {\vect u_2^{(3)}} + \left({\rm i}X_{\chi}\right)^* {\mathbb A}{\rm i}X_{\chi}{\vect u_2^{(2)}}  \right)  \\[0.3em]
&+\chi^{-4} \left( \left({\rm i}X_{\chi}\right)^* {\mathbb A}{\rm i}X_{\chi}{\vect u_2^{(2)}}  + \left({\rm i}X_{\chi}\right)^* {\mathbb A}{\rm i}X_{\chi} {\vect u_1^{(3)}} \right) + \begin{bmatrix}
      ({\vect m}^{(3)})_1   \\[0.3em]
     ({\vect m}^{(3)})_2  \\[0.3em]
    0
\end{bmatrix} + {\vect u_2}+ {\vect u_1^{(1)}} + {\vect u_2^{(1)}} + {\vect u_1^{(2)}} +{\vect u_2^{(2)}} +{\vect u_1^{(3)}}  +{\vect u_2^{(3)}}   .
\end{align*}
is bounded as follows:
\begin{equation}
\bigl\Vert\widetilde{\vect R}_{\chi}\bigr\Vert_{[H^1_\#(Y;H^1(\omega;\C^3))]^*} \leq C |\chi|^3 \bigl\Vert  \vect f \bigr\Vert _{L^2(\omega\times Y,\C^3)}.
\end{equation}
Noting that the approximation error 
$$\vect u_{\rm error}:= \vect u - \widetilde{\vect u}_{\rm approx}. $$
satisfies
$$
\chi^{-4}\left((\simgrad)^* + \left( {\rm i}X_\chi\right)^*\right){\mathbb A}\left(\simgrad + {\rm i}X_\chi\right) \vect u_{\rm error} + \vect u_{\rm error} = -\widetilde{\vect R}_{\chi} 
$$
and using a suitable version of \eqref{aprioriestimateh-1}, we obtain
\begin{equation}\label{nak34} 
\bigl\Vert\vect u_{\rm error}\bigr\Vert_{H^1(\omega\times Y;\C^3)} \leq C|\chi|^3 \bigl\Vert  \vect f \bigr\Vert _{L^2(\omega\times Y;\C^3)}.
\end{equation}
As before, discarding higher-order terms yields the following result.
\begin{proposition}\label{propfin2} 
Let $\vect u\in H^1_\#(Y;H^1(\omega;\C^3))$ be the solution of \eqref{generaltensorproblemchi4}. Then, the following estimates are valid:
\begin{equation}
\begin{aligned}
	\label{general_tensor_estimates_chi4}
\RRR \left\Vert P_i\bigl(\vect u - \vect u_0 \bigr) \right\Vert_{H^1(\omega\times Y,\C^2)}  \leq \left\{ \begin{array}{ll}
        C|\chi|\bigl\Vert  \vect f \bigr\Vert _{L^2(\omega\times Y;\C^3)}, & \mbox{ $i = 1,2$},\\[0.55em]
         C{\chi^2}\bigl\Vert  \vect f \bigr\Vert _{L^2(\omega\times Y;\C^3)}, & \mbox{ $i = 3$}.\end{array} \right.  
     \\[0.5em]
 \bigl\Vert P_i\bigl(\vect u-\vect u_0 -\vect u_0^{(1)}-{\vect u_1} \bigr)\bigr\Vert_{H^1(\omega\times Y,\C^2)} \leq \left\{ \begin{array}{ll}
        C{\chi^2}\bigl\Vert  \vect f \bigr\Vert _{L^2(\omega\times Y;\C^3)}, & \mbox{ $i = 1,2$},\\[0.55em]
         C|\chi|^3\bigl\Vert  \vect f \bigr\Vert _{L^2(\omega\times Y;\C^3)}, & \mbox{ $i = 3$}.\end{array} \right. 
\end{aligned}
\end{equation}
where ${\vect u_0}, {\vect u_0}^{(1)}, {\vect u_1}$ are defined by the above approximation procedure.
\end{proposition}

\begin{proof} 
 The proof follows from \eqref{nak34}, by using also \eqref{nak30}, \eqref{u0^1_general}, \eqref{nak31}, \eqref{nak32},  and \eqref{nak33}.  
 \end{proof} 
 
\begin{remark}
	In the same fashion as in Remark \ref{replacement_rem}, we argue that throughout this section the matrix ${\mathfrak C}^{\rm rod}_\chi$ can be replaced by ${\mathfrak C}^{\rm rod}$ with no impact on the error estimates. In particular, this allows us to use ${\mathfrak C}^{\rm rod}$ in the statements of Section \ref{gen_tens_sec} below.
\end{remark}

\subsection{Norm-resolvent estimates for the general elasticity tensor}
\label{gen_tens_sec}
\begin{proof}[Proof of the Theorem \ref{THML2L2}, formula \eqref{genres}]
The proof largely follows the approach of the proof of Theorem \ref{THML2L2} under the Assumption \ref{matsym}, see Section \ref{L2toL2}. In what follows, we mainly focus on the differences between the two arguments. 

First, we replace $\mathfrak{C}^{\rm rod}$ by $I$ in the same way as before. 
Here in the case of a general elasticity tensor, the spectrum of the operator $\mathcal{A}_\chi$ consists of two eigenvalues of order ${\chi^4}$, two eigenvalues of order ${\chi^2},$ and the remaining eigenvalues of order one. Thus, by providing estimates of the scaled resolvent problems \eqref{generaltensorproblemchi2} and \eqref{generaltensorproblemchi4}, we have actually estimated the resolvents in the two eigenspaces with eigenvalues of orders $\chi^4$ and $\chi^2$. In order to combine them, we employ the following argument. Using the uniform estimates on the eigenvalues provided by Proposition \ref{Rayleighestim},  we construct a closed contour $\Gamma_{{\chi^2}}$ surrounding the two eigenvalues of order one of the operator $\chi^{-2}\mathcal{A}_{\chi}$, and a closed contour $\Gamma_{{\chi^4}}$ surrounding the two eigenvalues of order one of the operator $\chi^{-4}\mathcal{A}_{\chi}$. (Note that even though the notation suggests that these contours depend on $|\chi|$,  they can actually be chosen independent of $|\chi|$ if it is small enough, i.e. below some apriori estimated constant.)  Next, we use the scaling functions $g_{\varepsilon,\chi}$, $h_{\varepsilon,\chi},$ which are analytic on the neighbourhoods of $\Gamma_{{\chi^4}}$,  $\Gamma_{{\chi^2}}$, defined as follows:
\begin{equation}
    h_{\varepsilon,\chi}(z) := \left(\frac{|\chi|^{4}}{\varepsilon^{\gamma + 2}}z + 1 \right)^{-1}, \quad  g_{\varepsilon,\chi}(z) := \left(\frac{{\chi^2}}{\varepsilon^{\gamma + 2}}z + 1 \right)^{-1},\; \; \Re(z) > 0,
\end{equation}
Clearly, for the orthogonal projections $P_{\Gamma_{{\chi^4}}}, P_{\Gamma_{{\chi^2}}}$ onto the two mentioned eigenspaces, one then has
\begin{equation}
\label{twoprojectors}
    \begin{aligned}
        \left(\frac{1}{\varepsilon^{\gamma + 2}}\mathcal{A}_\chi + I\right)^{-1}&\left(P_{\Gamma_{{\chi^4}}} + P_{\Gamma_{{\chi^2}}}\right)
        =\left(\frac{1}{\varepsilon^{\gamma + 2}}\mathcal{A}_\chi + I\right)^{-1} P_{\Gamma_{{\chi^2}}} + 
        \left(\frac{1}{\varepsilon^{\gamma + 2}}\mathcal{A}_\chi + I\right)^{-1} P_{\Gamma_{{\chi^4}}}\\[0.6em]
        &=\frac{1}{2\pi{\rm i}} \oint_{\Gamma_{{\chi^2}}} g_{\varepsilon,\chi}(z)\left(zI-\frac{1}{{\chi^2}}\mathcal{A}_\chi\right)^{-1}   dz + \frac{1}{2\pi{\rm i}} \oint_{\Gamma_{{\chi^4}}} h_{\varepsilon,\chi}(z)\left(zI-\frac{1}{{\chi^4}}\mathcal{A}_\chi\right)^{-1}  dz.
    \end{aligned}
\end{equation}
Furthermore, since
\begin{align*}
        \Bigg\Vert
        \left( \frac{1}{\varepsilon^{\gamma + 2}}\mathcal{A}_\chi + I\right)^{-1}
        &-\left(\frac{1}{\varepsilon^{\gamma + 2}}\mathcal{A}_\chi + I\right)^{-1}\left(P_{\Gamma_{{\chi^4}}} + P_{\Gamma_{{\chi^2}}}\right)
        \Bigg\Vert_{L^2 \to H^1} \\[0.6em]
        &
        =\left\Vert
        \left( \frac{1}{\varepsilon^{\gamma + 2}}\mathcal{A}_\chi + I\right)^{-1}
        \bigl(I-P_{\Gamma_{{\chi^4}}} - P_{\Gamma_{{\chi^2}}}\bigr)
        \right\Vert_{L^2 \to H^1} \leq C\varepsilon^{\gamma + 2},
\end{align*}
it suffices to estimate 
each of the two terms in \eqref{twoprojectors}. To this end, we combine the first row of fiberwise norm-resolvent estimates in \eqref{general_tensor_estimates_chi4}, \eqref{generaltensorestimateschi2} with straightforwards estimates of the functions $g_{\varepsilon,\chi}$, $h_{\varepsilon,\chi}$, which yields
\begin{align*}
       \Bigg\Vert P_i&\left(\left( \frac{1}{\varepsilon^{\gamma + 2}}\mathcal{A}_\chi + I \right)^{-1}
-(\mathcal{M}^{\rm rod}_\chi)^*\left(\frac{1}{\varepsilon^{\gamma + 2}}{\mathbb A}^{\rm rod}_\chi + \mathfrak{C}^{\rm rod}\right)^{-1}
\mathcal{M}^{\rm rod}_\chi 
		\right)\Bigg\Vert_{L^2 \to L^2} 
		\\[0.5em] 
 &\hspace{10em} \leq  
 \left\{\begin{array}{ll}
         C|\chi|\left(\max\left\{\dfrac{{\chi^2}}{\varepsilon^{\gamma + 2}}, 1\right\}\right)^{-1} + C|\chi| \left(\max\left\{\dfrac{{\chi^4}}{\varepsilon^{\gamma + 2}}, 1\right\}\right)^{-1}, & \mbox{ $i = 1,2$},\\[0.7em]
         C|\chi| \left(\max\left\{\dfrac{{\chi^2}}{\varepsilon^{\gamma + 2}}, 1\right\}\right)^{-1} + C{\chi^2} \left(\max\left\{\dfrac{{\chi^4}}{\varepsilon^{\gamma + 2}}, 1\right\}\right)^{-1}, & \mbox{ $i = 3.$}
     \end{array}\right. 
     \\[0.6em]
    &\hspace{10em}\leq
   {C}\left\{\begin{array}{ll}
    \varepsilon^{\tfrac{\gamma + 2}{4}}, & \mbox{ $i = 1,2$},\\[0.5em]
    \varepsilon^{\tfrac{\gamma + 2}{2}}, & \mbox{ $i = 3$}.\end{array}\right. 
\end{align*}
The proof is concluded by applying the inverse Gelfand transform.
\end{proof}
\begin{remark}
\label{rem61}
The parts of Theorems \ref{l2h1theorem}, \ref{thm_gen_L2L2high} pertaining to the general case (i.e. the estimates in the $L^2 \to H^1$ norm and the higher-order accuracy estimates in the $L^2 \to L^2$ norm, respectively) are obtained similarly to the case of invariant subspaces, by utilising a version of the argument provided in the proof of Theorem \ref{THML2L2}, formula \eqref{genres}. This, in particular, involves defining a corrector $\mathcal{A}^{\rm corr}_{\rm rod}(\varepsilon)$ by analogy with $\mathcal{A}^{\rm corr}_{\rm bend}(\varepsilon)$ and $\mathcal{A}^{\rm corr}_{\rm stretch}(\varepsilon)$ (see \eqref{trans_back_corr1_str}) 
as well as a corrector $\mathcal{\widetilde{A}}^{\rm corr}_{\rm rod}(\varepsilon)$ by analogy with $\mathcal{\widetilde{A}}^{\rm corr}_{\rm bend}(\varepsilon)$ and $\mathcal{\widetilde{A}}^{\rm corr}_{\rm stretch}(\varepsilon)$ (see \eqref{higherorderl2l2correctors}).
\end{remark}
\begin{remark}
\label{lastremark}
An argument similar to that of Corollary \ref{absenceofforceterms} can be used to demonstrate that one can drop the smoothing operator from the $L^2 \to L^2$ norm-resolvent estimates \eqref{genres} while keeping the same order of accuracy. Namely, there exists $C>0$ such that for every $\varepsilon> 0$ one has
\begin{equation}
\left\Vert 
	P_i 
		\left( 
			\left( \frac{1}{\varepsilon^\gamma}\mathcal{A}_\varepsilon + I \right)^{-1} 
- (\mathcal{M}^{\rm rod}_\varepsilon)^*\left(\frac{1}{\varepsilon^\gamma} \mathcal{A}^{\rm rod}_{\varepsilon} + \mathfrak{C}^{\rm rod} \right)^{-1}\mathcal{M}^{\rm rod}_\varepsilon S_\infty
		\right) 
 \right\Vert_{L^2 \to L^2} \leq  
 {C}\left\{ \begin{array}{ll}
         \varepsilon^{\tfrac{\gamma + 2}{4}}, & \mbox{ $i = 1,2$},\\[0.5em]
         \varepsilon^{\tfrac{\gamma + 2}{2}}, & \mbox{ $i = 3$}.\end{array} \right. 
\end{equation}
\end{remark}

\section{Appendix} 
\subsection{Proof of Proposition \ref{proposition21}}\label{appproof} 
\begin{proof}
	The only non-trivial fact is the coercivity estimate in  \eqref{coerc}. To establish it, we use the pointwise coercivity estimate in \eqref{pointwisecoercivity}:
	\begin{equation}
	\label{combine1}
	\begin{aligned}
	{\mathbb A}^{\rm rod}{\vect m}\cdot{\vect m}= & \int_{\omega \times Y} {\mathbb A}(y) \left( \mathcal{J}^{\rm rod}_{\vect m}(\widehat{x}) + \simgrad {\vect u_{\vect m}}(\widehat{x},y)\right):\left( \mathcal{J}^{\rm rod}_{m}(\widehat{x}) + \simgrad {\vect u_{\vect m}}(\widehat{x},y)\right )  d\widehat{x}dy \\[0.3em]
	\geq & C\left\Vert \mathcal{J}^{\rm rod}_{\vect m} + \simgrad {\vect u_{\vect m}}\right\Vert_{L^2(\omega \times Y;\R^{3 \times 3})}^2 \\[0.3em]
	\geq & C\left(\bigl\Vert x_2 m_3 - \bigl( \partial_1 ({\vect u}_{\vect m})_3 + \partial_3 {({\vect u}_{\vect m})_1} \bigr) \bigr\Vert_{L^2(\omega\times Y)}^2 +\bigl\Vert -x_1 m_3 - \bigl( \partial_2 ({\vect u}_{\vect m})_3 + \partial_3 {({\vect u}_{\vect m})_2} \bigr) \bigr\Vert_{L^2(\omega\times Y)}^2 \right)\\[0.3em] 
	&\qquad\qquad+C\bigl\Vert m_4 - x_1 m_1 - x_2 m_2 - \partial_3 ({\vect u}_{\vect m})_3\bigr\Vert_{L^2(\omega\times Y)}^2.
	\end{aligned}
	\end{equation}
	One the one hand, it is clear that $\partial_3 ({\vect u}_{\vect m})_3 \perp  m_4 - x_1 m_1 - x_2 m_2$ in $L^2(\omega \times Y)$, so one has
	\begin{equation}
	\bigl\Vert m_4 - x_1 m_1 - x_2 m_2 - \partial_3 ({\vect u}_{\vect m})_3  \bigr\Vert_{L^2(\omega\times Y)}^2 \geq  \left\Vert m_4 - x_1 m_1 - x_2 m_2 \right\Vert_{L^2(\omega\times Y)}^2 \geq C\left(|m_1|^2 + |m_2|^2 + |m_4|^2 \right),
	\label{combine2}
	\end{equation}
	by virtue of the third condition in \eqref{coordinatesymmetries}. On the other hand, one has 
	\begin{equation}
	\label{combine3}
	\begin{aligned}
	\bigl\Vert x_2 m_3 - \bigl( \partial_1 ({\vect u}_{\vect m})_3 + \partial_3 {({\vect u}_{\vect m})_1} \bigr)\bigr\Vert_{L^2(\omega\times Y)}^2 +\bigl\Vert -x_1 m_3 - \bigl( \partial_2({\vect u}_{\vect m})_3 + \partial_3 {({\vect u}_{\vect m})_2}\bigr) \bigr\Vert_{L^2(\omega\times Y)}^2\\[0.5em]
	=\left\Vert m_3 \begin{bmatrix}
	x_2 \\ - x_1
	\end{bmatrix} - \partial_3\widehat{{\vect u}}_{\vect m} 
	- \nabla_{\widehat{x}}({\vect u}_{\vect m})_3 \right\Vert_{L^2(\omega\times Y;\R^2)}^2.
	\end{aligned}
	\end{equation}
	Consider the operator $P_G$ in $L^2(\omega;\R^2)$ of the orthogonal projection onto the set $G:=\left\{\nabla v,  v \in H^1(\omega) \right\}$. The operator $P_G$ is well defined and bounded since $G$ is closed. Furthermore, one has 
	\begin{equation}
	\begin{aligned}
	\left\Vert m_3 \begin{bmatrix}
	x_2 \\ - x_1
	\end{bmatrix} - \partial_3\widehat{{\vect u}}_{\vect m}
	- \nabla_{\widehat{x}} {u_3} \right\Vert_{L^2(\omega\times Y;\R^2)}^2 & \geq C \left\Vert (I -P_G)\left( m_3 \begin{bmatrix}
	x_2 \\ - x_1
	\end{bmatrix} - \partial_3 \widehat{{\vect u}}_{\vect m}
	\right) \right\Vert_{L^2(\omega\times Y;\R^2)}^2 \\[0.4em]
	&= C\left(m_3\left\Vert (I -P_G) \begin{bmatrix}
	x_2 \\ - x_1
	\end{bmatrix} \right\Vert_{L^2(\omega\times Y;\R^2)}^2 +  \bigl\Vert (I -P_G) \partial_3 \widehat{{\vect u}}_{\vect m}
	\bigr\Vert_{L^2(\omega\times Y;\R^2)}^2  \right) \\[0.4em]
	& + C\int_Y \left\langle m_3(I -P_G) \begin{bmatrix}
	x_2 \\ - x_1
	\end{bmatrix},(I -P_G) \partial_3 \widehat{{\vect u}}_{\vect m}
	\right\rangle_{L^2(\omega;\R^2)},
	\end{aligned}
	\end{equation}
	and noting that
	\begin{equation}
	\begin{aligned}
	\int_Y \left\langle m_3(I -P_G) \begin{bmatrix}
	x_2 \\ - x_1
	\end{bmatrix},(I -P_G) \partial_3 \widehat{{\vect u}}_{\vect m}
	\right\rangle_{L^2(\omega;\R^2)}  
	& = \int_Y \left\langle m_3(I -P_G) \begin{bmatrix}
	x_2 \\ - x_1
	\end{bmatrix},\partial_3\widehat{{\vect u}}_{\vect m}
	\right\rangle_{L^2(\omega;\R^2)} \\[0.4em]
	& = m_3\int_\omega \left\langle (I -P_G) \begin{bmatrix}
	x_2 \\ - x_1
	\end{bmatrix},\partial_3 \widehat{{\vect u}}_{\vect m}
	\right\rangle_{L^2(Y;\R^2)} = 0,
	\end{aligned}
	\end{equation}
	one obtains the bound
	\begin{equation}
	\left\Vert m_3 \begin{bmatrix}
	x_2 \\ - x_1
	\end{bmatrix} - \partial_3\widehat{{\vect u}}_{\vect m}
	- \nabla_{\widehat{x}}({\vect u}_{\vect m})_3 \right\Vert_{L^2(\omega\times Y;\R^2)}^2  \geq C |m_3|^2.
	\label{combine4}
	\end{equation}
	Combining \eqref{combine1}, \eqref{combine2}, \eqref{combine3}, and \eqref{combine4}, we  obtain  the required estimate \eqref{coerc}.
\end{proof}
\subsection{Resolvent formalism}\label{appRes}
For a closed operator $\mathcal{A}$ on a Banach space $X$, with domain $\mathcal{D}(\mathcal{A})\subset X,$ we consider its resolvent set $$\rho(\mathcal{A}):= \left\{z \in \C: 
{\rm Operator}\ \mathcal{A}-zI:\mathcal{D}(\mathcal{A}) \to X \mbox{ is bijective}\right\}.$$ For every $z \in \rho(\mathcal{A})$, the resolvent of $A$ is defined by
$$
R(z,\mathcal{A}):=(\mathcal{A}-zI)^{-1}:X \to X.
$$
It is known that we have the following two identities: 
\begin{equation}
\begin{aligned}
&\mbox{First resolvent identity: } \quad R(z,\mathcal{A}) - R(w,\mathcal{A}) = (z-w)R(z,\mathcal{A})R(w,\mathcal{A}) \quad \forall z,w \in \rho(\mathcal{A}), \\[0.3em]
&\mbox{Second resolvent identity: } \quad
R(z,\mathcal{A}) -R(z,\mathcal{B}) = R(z,\mathcal{A})(\mathcal{A}-\mathcal{B})R(z,\mathcal{B}) \quad \forall z \in \rho(\mathcal{A})\cap \rho(\mathcal{B}). 
\end{aligned}
\end{equation}

To establish a norm-resolvent estimate is to provide the estimate for the difference of two resolvents in the operator-norm topology. Note that estimates of resolvents that depend on the spectral parameter $z \in \C$ can usually be reduced to a single resolvent estimate where the value of $z$ is fixed, by appropriately modifying the corresponding resolvent problem.
On the basis of this observation, the following lemma is easily shown to hold,using the first and the second resolvent identity. 
\begin{lemma} \label{lemres}
	Let $w,z \in \rho(\mathcal{A})\cap \rho(\mathcal{B})$, where $\mathcal{A}$, $\mathcal{B}$ are closed operators on $X$. Then we have 
	\begin{equation}
	\left\Vert R(z,\mathcal{A})- R(z,\mathcal{B}) \right\Vert_{X \to X} \leq C(z,w)\left\Vert R(w,\mathcal{A})- R(w,\mathcal{B})\right\Vert_{X \to X},
	\end{equation}
	where $$C(z,w):=\max \left\{ 1, \frac{|z - w|}{{\rm dist}(z, \sigma(\mathcal{A}))}\right\}\max \left\{ 1, \frac{|z - w|}{{\rm dist}(z, \sigma(\mathcal{B}))}\right\}.$$
\end{lemma}
\subsection{Korn's inequality} \label{appKorn} 
We state the so-called second Korn's inequality in the following form, see \cite{Hor}.
\begin{proposition}
	\label{quasiperiodickorn}
	For every $\vect u \in H^1(\omega \times Y;\C^3),$ the following estimate holds with $C>0$ dependent only on $\omega:$
	\begin{equation}
	\bigl\Vert \vect u(x) - (A x + \vect c) \bigr\Vert_{H^1(\omega\times Y;\C^3)} \leq C \bigl\Vert  \simgrad \vect u \bigr\Vert _{L^2(\omega \times Y;\C^{3\times 3})},
	\label{second_Korn} 
	\end{equation}
 for some "infinitesimal rigid-body motion" $Ax+\vect c,$ $x\in\omega\times Y,$ with
	\begin{align*}
	A&= \begin{bmatrix}
	0 & d & a \\[0.25em]
	-d & 0 & b \\[0.25em]
	-a & -b & 0 \\
	\end{bmatrix} \in \C^{3 \times 3}, 
	\qquad \vect c = \begin{bmatrix}
	c_1 \\[0.2em]
	c_2 \\[0.2em]
	c_3
	\end{bmatrix} \in \C^3. 
	\end{align*}
\end{proposition}

\section*{Acknowledgements}

KC is grateful for the support of the Engineering and Physical Sciences Research Council (EPSRC): Grant
EP/L018802/2 “Mathematical foundations of metamaterials: homogenisation, dissipation and operator
theory”. IV and J\v{Z} have been supported by the Croatian Science Foundation under Grant agreement No. 9477
(MAMPITCoStruFl) and Grant agreement No. IP-2018-01-8904 (Homdirestroptcm). J\v{Z} is also supported by the Ministry of Science and Higher Education of the Russian Federation, agreement 075-15-2019-1620 date 08/11/2019.

\section*{Data availability}

Data sharing is not applicable to this article as no datasets were generated or analysed during the current study.

\end{document}